\UseRawInputEncoding
\pdfoutput=1
\documentclass[10pt,reqno]{amsart}

\usepackage{amscd,amsmath,amsfonts,amsthm,epsfig, graphics, latexsym, mathrsfs}
\usepackage{amssymb,color}
\usepackage[all]{xy}
\usepackage{longtable}
\usepackage{bbm}

\usepackage{amssymb}
\usepackage{mathdots}
\usepackage{mathtools}

\usepackage{tikz,tkz-euclide}
\usepackage{tikz-3dplot}

\def\red#1{{\color{red} #1}}

\usepackage[breakable,skins]{tcolorbox}

\newtcolorbox{tbox}[1][]{%
    breakable,
    enhanced,
    colframe=black,
    coltitle=white,
    #1
}

\usepackage[alphabetic,initials]{amsrefs}

\theoremstyle{plain}

\newtheorem{theorem}{Theorem}[section]
\newtheorem{proposition}[theorem]{Proposition}
\newtheorem{lemma}[theorem]{Lemma}
\newtheorem{corollary}[theorem]{Corollary}

\newtheorem{lemdef}[theorem]{Lemma--Definition}

\theoremstyle{definition}

\newcommand{\appsection}[1]{\let\oldthesection\thesection
\renewcommand{\thesection}{Appendix \oldthesection}
\section{#1}\let\thesection\oldthesection}

\newtheorem{definition}[theorem]{Definition}
\newtheorem{notation}[theorem]{Notation}

\theoremstyle{remark}

\newtheorem{remark}[theorem]{Remark}
\newtheorem{example}[theorem]{Example}

\newtheorem{computation}[theorem]{Computation}
\newtheorem{database}[theorem]{Database}

\DeclareMathOperator{\Aut}{Aut}
\DeclareMathOperator{\End}{End}

\DeclareMathOperator{\Bl}{Bl}
\DeclareMathOperator{\Cl}{Cl}
\DeclareMathOperator{\GL}{GL}
\DeclareMathOperator{\Gal}{Gal}
\DeclareMathOperator{\rd}{res}
\DeclareMathOperator{\redd}{red}
\DeclareMathOperator{\tors}{tors}

\DeclareMathOperator{\ored}{\overline{\rd}}
\DeclareMathOperator{\dra}{\dashrightarrow}

\DeclareMathOperator{\Cox}{Cox}

\DeclareMathOperator{\barMov}{\overline{\text{Mov}}}

\DeclareMathOperator{\Nef}{Nef}
\DeclareMathOperator{\barEff}{\overline{\hbox{\rm Eff}}}
\DeclareMathOperator{\Eff}{Eff}

\DeclareMathOperator{\Pic}{Pic}

\DeclareMathOperator{\Spec}{Spec}

\DeclareMathOperator{\ch}{char}
\DeclareMathOperator{\Ker}{Ker}

\DeclareMathOperator{\ord}{ord}

\DeclareMathOperator{\Imm}{Im}

\newcommand{\QED}{\ifhmode\unskip\nobreak\fi\quad {\rm Q.E.D.}} 

\newcommand{\e}{e}
\newcommand{\vol}{\operatorname{Vol}}
\newcommand{\Num}{\operatorname{Num}}
\newcommand{\NP}{\operatorname{NP}}
\newcommand{\al}{\alpha}
\newcommand{\be}{\beta}
\newcommand{\ga}{\gamma}
\newcommand{\de}{\delta}

\newcommand{\bA}{\mathbb A}

\newcommand{\bC}{\mathbb C}
\newcommand{\bD}{\mathbb D}
\newcommand{\bE}{\mathbb E}
\newcommand{\bF}{\mathbb F}
\newcommand{\bG}{\mathbb G}

\newcommand{\bP}{\mathbb P}
\newcommand{\bQ}{\mathbb Q}
\newcommand{\bR}{\mathbb R}

\newcommand{\bZ}{\mathbb Z}

\newcommand{\cB}{\mathcal B}
\newcommand{\cC}{\mathcal C}

\newcommand{\cE}{\mathcal E}
\newcommand{\cF}{\mathcal F}

\newcommand{\cL}{\mathcal L}

\newcommand{\cO}{\mathcal O}
\newcommand{\cP}{\mathcal P}

\newcommand{\cX}{\mathcal X}
\newcommand{\cY}{\mathcal Y}
\newcommand{\cZ}{\mathcal Z}

\newcommand{\ra}{\rightarrow}

\newcommand{\oM}{\overline{M}}
\newcommand{\oLM}{\overline{LM}}

\newcommand{\oEff}{\overline{\Eff}}

\newcounter{et}[section]

\numberwithin{equation}{section}

\begin{document}
\bibliographystyle{amsplain}

\title[Blown-up toric surfaces with non polyhedral 
effective cone]{Blown-up toric surfaces with\\ 
non-polyhedral 
effective cone}

\author[A-M.~Castravet]{Ana-Maria Castravet}

\address{Universit\'e Paris-Saclay, UVSQ, CNRS, Laboratoire de Math\'ematiques de Versailles, 78000, Versailles, France}
\email{ana-maria.castravet@uvsq.fr }

\author[A.~Laface]{Antonio Laface}
\address{
Departamento de Matem\'atica,
Universidad de Concepci\'on,
Casilla 160-C,
Concepci\'on, Chile}
\email{alaface@udec.cl}

\author[J.~Tevelev]{Jenia Tevelev}

\address{Department of Mathematics and Statistics, University of Massachusetts Amherst, 710 North Pleasant Street, Amherst, MA 01003, USA and 
Laboratory of Algebraic Geometry and its Applications, HSE, Moscow, Russia}
\email{tevelev@math.umass.edu}

\author[L.~Ugaglia]{Luca Ugaglia}
\address{
Dipartimento di Matematica e Informatica,
Universit\`a degli studi di Palermo,
Via Archirafi 34,
90123 Palermo, Italy}
\email{luca.ugaglia@unipa.it}

\begin{abstract}
We construct  examples
of projective toric surfaces whose blow-up 
at a general point has a non-polyhedral
pseudo-effective cone, both in characteristic~$0$ and in every prime characteristic~$p$.
As a consequence, 
we prove that the pseudo-effective cone 
of the Grothendieck--Knudsen moduli space $\overline M_{0,n}$ of stable rational curves is not polyhedral
for $n\geq 10$ in characteristic~$0$
and in characteristic $p$, for all primes $p$. 
Many of these toric surfaces are related to a very interesting 
class of arithmetic threefolds that we call arithmetic
elliptic pairs of infinite order. Their analysis in characteristic $p$ relies on tools of arithmetic geometry and Galois representations in the spirit of the Lang--Trotter conjecture, producing toric surfaces whose blow-up at a general point has a non-polyhedral
pseudo-effective cone in characteristic~$0$ and in characteristic~$p$, for an infinite set of primes $p$ of positive density.  
\end{abstract}


\maketitle

\section{Introduction}

An effective cone of a projective variety $X$ and its closure, the pseudo-effective cone $\oEff(X)$,
contain an impressive amount of information about the birational geometry of $X$.
An even finer invariant is the Cox ring $\Cox(X)$, at least when the class group 
$\Cl(X)$ is finitely generated. 
If $X$ is a Mori Dream Space (MDS) then $\Cox(X)$ is finitely generated, which in turn implies
that $\oEff(X)$ is polyhedral. A~basic example of a MDS is a projective toric variety \cite{Cox}.
Its effective cone is generated by classes of toric boundary divisors.
For a toric variety $\bP$, we denote by $\Bl_e\bP$ its blow-up at the identity element of the torus.
Our main result contributes to the growing body of evidence that this is
a very intriguing class of varieties.


\begin{theorem}\label{kjshbwgwH}
In every characteristic, there exist projective toric surfaces $\bP$
such that the pseudo-effective cone $\oEff(\Bl_e\bP)$ is not polyhedral. 
\end{theorem}

In order to prove Theorem~\ref{kjshbwgwH}, we introduce two types of lattice polygons, Lang--Trotter polygons and Halphen polygons.
The blow-ups $X=\Bl_e\bP$ of toric surfaces associated to these 
polygons 
are examples of elliptic pairs studied in \S\ref{sDGSHSH}.
An  {\em elliptic pair} $(C,X)$ is a projective rational surface $X$,
with log terminal singularities, and a curve $C$
contained in the smooth locus of $X$, such that $p_a(C)=1$ and  $C^2=0$. 
Much of the geometry is encoded in the restriction map
$\rd:\,C^\perp\to\Pic^0(C)$, 
where $C^\perp\subseteq\Cl(X)$ is the orthogonal complement.
The order of an elliptic pair is the order of $\rd(C)$.
A~familiar example of an elliptic pair of infinite order in any characteristic is
the blow-up of $\bP^2$ in $9$ general points.
By contrast, elliptic pairs $X=\Bl_e\bP$ associated with a toric surface are defined over the base field.
In~particular, their order is automatically finite in characteristic $p$.

If~the order of an elliptic pair $(C,X)$ is infinite and  $\rho(X)\ge3$, then $\oEff(X)$ is not polyhedral
(Lemma~\ref{adfbafb}). By contrast, polyhedrality of $\oEff(X)$ is harder to control for elliptic pairs of finite order (e.g. for our blow-ups of toric surfaces in characteristic $p$)
unless the pair is {\em minimal}.
We use the logarithmic minimal model program  to construct a $(K+C)$-minimal model $(C,Y)$
of any elliptic pair $(C,X)$ and focus on the study of polyhedrality of $\oEff(Y)$.
Of course if $\oEff(Y)$ is not polyhedral then $\oEff(X)$ is also not polyhedral. 
Remarkably, $Y$ has Du Val singularities
if the order is {\em infinite} (Corollary~\ref{kjsHFkjshf}).
On the other hand, if the order is {\em finite} and $Y$ has Du Val singularities,
there is a simple criterion for polyhedrality (Corollary~\ref{concrete})
in terms of the restriction map and the root sublattice $T\subset\bE_8$.
The~synthesis of these disjoint scenarios is  the notion of an {\em arithmetic elliptic pair of infinite order},
a~flat pair of schemes $(\cC,\cX)$ over an open subset in the spectrum of a ring of algebraic integers
with elliptic pairs as geometric fibers, of infinite (resp., finite) order over an infinite (resp., finite) place.
In \S\ref{asgasrhasrh} we study distribution of polyhedral primes
for arithmetic toric elliptic pairs $(\cC,\cX)$ of infinite order.
Here we call a prime $p$ polyhedral if $\oEff(Y)$ is polyhedral,
where $(C,Y)$ is the minimal model of the geometric fiber $(C,X)$ in characteristic $p$.
This distribution is an intriguing question of arithmetic geometry,
which we reduce to the question about reductions
of points on the elliptic curve in the spirit of the Lang and Trotter analysis \cite{LT}.


We found many examples of {\em Lang--Trotter polygons}
that give rise to arithmetic elliptic pairs of infinite order, see the list of $135$ polygons displayed in Database~\ref{vewfvefvef}.
For some of these polygons, $\oEff(\Bl_e\bP)$ is not polyhedral in characteristic~$p$ for an infinite set of primes~$p$ of positive density. 
We also checked that for every prime $p<2000$, there exists 
a Lang--Trotter  polygon such that $\oEff(\Bl_e\bP)$ is not polyhedral in characteristic~$p$
(see Database~\ref{adhadthstj}). It seems likely that 
one can find such a Lang--Trotter polygon $\bP$ for every $p$, 
but this seems out of reach with our methods. We also found infinite series of Lang--Trotter polygons
such that $\oEff(\Bl_e\bP)$ is polyhedral in every positive characteristic $p$ (Theorem \ref{asgsgsrSRH}). 

We observe a different behaviour in \emph{Halphen polygons},
which give rise to elliptic pairs of finite order with Du Val singularities both in characteristic $0$ and in prime characteristic. Here the condition on singularities of the minimal model is not guaranteed by a general theory and needs to be checked by hand. We exhibit an example in Theorem \ref{ex3}
of a Halphen polygon such that 
$\oEff(\Bl_e\bP)$ is not polyhedral  in characteristic~$0$ and in characteristic~$p$ for all but a finite set of primes~$p$. Empirically, Halphen polygons seem to be harder to find than Lang--Trotter polygons.



In 
\S\ref{asfgSrhasrh},
we describe the Magma package  and databases that can be used for   computer-aided calculations, for example to check that a given polygon is a Lang--Trotter polygon.
However, this analysis 
involves the step of computing a member~$\Gamma$ of a linear system on the toric surface $\bP$
that has a point of large multiplicity at the identity of the torus
(the  elliptic curve $C$ is the normalization of $\Gamma$).
While it is relatively straightforward to implement our algorithms on the computer, this
{\em implicit method}
has obvious disadvantages, for example it is not clear how to apply it
to construct an infinite sequence of examples. To remedy the situation, in \S\ref{asrgasrharh}
we develop a {\em parametric method}. We start with an elliptic curve~$C$
and construct its map to $\bP$ that folds many points of $C$ onto a point of high multiplicity.
We do this for an infinite  sequence of elliptic curves.
We hope that this new method may help with other problems loosely related to the Nagata approximation conjecture, where it is desirable to geometrically construct curves with points of high multiplicity.

While most of the Lang--Trotter polygons that we found are not smooth, in \S\ref{sgasrhasrh} we describe a smooth toric elliptic pair $(C,X)$.
It has a large Picard number $\rho = 18$ and the Mordell--Weil rank of $C$ is equal to $9$.
We don't know if there is an upper bound on the Picard number 
(or the Mordell--Weil rank)
of a toric elliptic pair.

\medskip

Our main application of Theorem~\ref{kjshbwgwH} is to the birational geometry of the Gro\-then\-dieck--Knudsen
moduli space $\oM_{0,n}$ of stable rational curves
with $n$ marked points.
The study of effective cones of moduli spaces has a long history, 
starting with the pioneering work of 
Harris and Mumford \cite{HM}, who
used computations of effective divisors to show that $\oM_g$ is 
not unirational for $g\gg0$.

While the moduli space $\oM_{0,n}$ is a rational variety, 
its birational geometry is far from understood, in spite of numerous efforts, see for example 
\cites{AS, BG, CT_Crelle, CT_Duke, DGJ, F, FS, G, GG, GJM, GKM, GM_nef, KMcK}. The Picard number
of $\oM_{0,n}$ grows exponentially and it is not a Fano variety for $n\ge6$, in fact 
its anticanonical class is not pseudo-effective if $n\geq8$. 
In this regard, $\oM_{0,n}$ looks similar to the blow-up of $\bP^2$ in $n$ points (a connection was indeed found in \cite{CT_Comp}) but there is an important difference: $\oM_{0,n}$ is rigid, i.e., it ~has no moduli.

A question attributed to Fulton, which received a lot of attention, is whether, similarly to the case of toric varieties, any  subvariety of $\oM_{0,n}$ is numerically equivalent to a sum of strata. 
For the case of curves, the statement is known as the F-conjecture. 
A result of Gibney, Keel and Morrison \cite{GKM} proves that the F-conjecture, if known for all $n$, implies the similar statement for $\oM_{g,n}$, for all genera $g$ and number of marked points $n$, thus giving an explicit combinatorial description to the ample cone of  $\oM_{g,n}$. The conjecture holds for $n\leq 7$ and is open for $n\geq8$. 

For the case of divisors, Fulton's question is whether the class of every effective divisor on $\oM_{0,n}$ is a sum of boundary divisors. 
Every boundary divisor is an extremal ray of $\oEff(\oM_{0,n})$; 
in fact, these divisors are exceptional, i.e., they can be contracted by birational contractions.
For example, $\oM_{0,5}$ is a degree~$4$ del Pezzo surface, and its boundary divisors
form the Petersen graph of ten $(-1)$-curves, which generate $\Eff(\oM_{0,5})$.
Extremal rays of a different type for $\oM_{0,6}$  were found by Keel and Vermeire~\cite{Ver}, thus giving a negative answer to 
Fulton's question for divisors when $n\geq6$\footnote{Using forgetful maps, one has a negative answer  for all cycles of dimension $\geq2$ when $n\geq6$.}. 
However, Hassett and Tschinkel proved in \cite{HasT} that $\oEff(\oM_{0,6})$ is still fairly simple, 
namely it is a polyhedral cone, generated by the boundary and the Keel--Vermeire divisors (only one up to $S_6$ symmetry). 

A large class of exceptional divisors on $\oM_{0,n}$ was discovered by Castravet and Tevelev \cite{CT_Crelle}.
They are parametrized by  irreducible hypertrees, 
which can be obtained, for example, from bi-colored triangulations of the $2$-sphere.
Up~to~the action of the symmetric group $S_n$, this gives  
$1,2,11,93,1027,\ldots$ new types of exceptional divisors on $\oM_{0,n}$ for $n=7,8,9,10,11,\ldots$.
Equations of these divisors appear as numerators of leading singularities scattering amplitude forms for  $n$  particles in  
$N=4$ super-symmetric Yang--Mills theory \cites{MHV,Scattering}.

New extremal rays of $\oEff(\oM_{0,n})$  were found by Opie \cite{Opie} 
disproving an over-optimistic conjecture from \cite{CT_Crelle}.
Further extremal rays  were found by Doran, Giansiracusa, and Jensen \cite{DGJ}.
Our second  result explains  this complexity.
\begin{theorem}\label{asgarh}
The cone $\oEff(\oM_{0,n})$ is not polyhedral for $n\geq 10$,
both in characteristic $0$ and in characteristic $p$, for all primes $p$. 
\end{theorem}

The moduli space $\oM_{0,n}$ is related to blown-up toric varieties via the notion of a {\em rational contraction},
a dominant rational  map $X\dra Y$ of projective varieties that can be decomposed into a sequence
of small $\bQ$-factorial modifications \cite{HK} and surjective morphisms.
Given a rational contraction, if  $\oEff(X)$ is a (rational) polyhedral cone then 
$\oEff(Y)$ is also (rational) polyhedral (Lemma~\ref{fvqevf}). By \cite{CT_Duke},
there exist rational contractions
$\Bl_e\oLM_{n+1}\dra\oM_{0,n}\dra\Bl_e\oLM_n$,
where $\oLM_n$ is the Losev--Manin moduli space of chains
of rational curves, see
\cites{LM,CT_AG}.
This is a toric variety associated with the permutohedron.
Thus $\oM_{0,n}$ has essentially the same birational geometry
as the blow-up of a toric variety in the identity element of the torus.
Moreover, a  feature of $\oLM_n$, 
noticed in \cite{CT_Duke} and proved in Theorem~\ref{srgwrG},
is its ``universality'' among all projective toric varieties $\bP$.
Specifically, for any projective toric variety $\bP$ there exist  
rational contractions
$\oLM_n\dra \bP$ and $\Bl_e\oLM_n\dra \Bl_e\bP$
for $n$ sufficiently large. 
In particular, if the cone $\oEff(\Bl_e\bP)$ is not polyhedral for some toric variety $\bP$
then $\oEff(\Bl_e\oLM_n)$, and therefore $\oEff(\oM_{0,n})$, are not polyhedral either, for $n$ sufficiently large.

A similar strategy was used in \cite{CT_Duke} to show that $\oM_{0,n}$ is not a MDS in characteristic~$0$
for $n\ge134$, answering a question of Hu and Keel \cite{HK}.
The~bound was lowered to $13$ 
by Gonzalez and Karu \cite{GKK} and then to $10$ by  Hausen, Keicher, and Laface \cite{HKL}.
Theorem~\ref{asgarh} gives the same bound $n\ge10$, but it exhibits an even wilder behaviour than previously expected, as effective 
cones are a much rougher invariant than Cox rings
(the Cox ring is graded and the effective cone is the semigroup of possible weights of the grading).
For instance, 
the toric surfaces used in \cite{CT_Duke} were the weighted projective planes $\bP(a,b,c)$.
Of~course, $\Bl_e\bP(a,b,c)$ has Picard number~$2$ and its effective cone
is polyhedral. Nevertheless, Goto, Nishida and Watanabe \cite{GNW} proved
(motivated by a question of Cowsik in commutative algebra) that 
$\Bl_e\bP(a,b,c)$ is not a MDS in characteristic $0$
for certain values of $a,b,c$, by exhibiting
a nef but not semi-ample line bundle. However, 
in~characteristic~$p>0$, this line bundle is semi-ample, and therefore this space is a MDS, by Artin's  criterion \cite{Artin},  
hence, this technique cannot be used. The following  corollary of Thm.~\ref{asgarh}
is  new:
\begin{corollary}\label{cor-pos}
If $n\geq 10$, the moduli space $\oM_{0,n}$ is not a MDS in characteristic $p$, for all primes $p$. 
\end{corollary}
By contrast, $\oM_{0,n}$ is a MDS in all characteristics
if $n\le 6$ \cites{HK,Cas}. This leaves open the cases $n=7,8,9$. 

\medskip 

{\em Acknowledgements. }
A.-M. C. ~ has been partially 
supported by the ANR-20-CE40-0023 grant. 
A.L.~ has been partially 
supported by Proyecto FONDECYT Regular
n.~1190777.
J.T.~was partially supported by the NSF grant DMS-1701704, 
Simons Fellowship, and by the HSE University Basic Research Program and Russian Academic Excellence Project '5-100'.

We thank Igor Dolgachev, Gavril Farkas and Brian Lehmann
for useful discussions  and for answering our questions. 
We are especially grateful to Tom Weston
for sharing and  explaining his paper \cite{W}.
In the REU directed by J.T.~in 2017, Stephen Obinna \cite{Obinna}
has started to collect evidence for existence of blown-up toric surfaces with non-polyhedral effective cone.
The software Magma~\cite{mag} was used extensively.
Some of the graphics are by the 
 \href{https://www.plainformstudio.com}{Plain Form Studio}.


\setcounter{tocdepth}{1}

\tableofcontents

\section{Polyhedrality of  effective cones}\label{asrgasrg}

Let $k$ be an algebraically closed field of arbitrary characteristic. 
We recall some definitions (see for example \cites{LazI, LazII}). 
If  $X$ is a normal projective irreducible variety over $k$, let $\Cl(X)$ be the divisor class group and let 
$\Pic(X)$ be the Picard group of $X$. As usual, we denote by $\sim$ the linear equivalence of divisors
and by $\equiv$ the numerical equivalence.
Recall that for  Cartier divisors $D_1$, $D_2$, we have
 $D_1\equiv D_2$ if and only if $D_1\cdot C=D_2\cdot C$, for any curve $C\subseteq X$. 
We let
$$\Num^1(X):=\Pic(X)/\equiv$$ be the group of numerical 
equivalence classes of Cartier divisors on $X$. 
We denote $\Num^1(X)_{\bR}=\Num^1(X)\otimes_{\bZ}\bR$, $\Num^1(X)_{\bQ}=\Num^1(X)\otimes_{\bZ}\bQ$. 

Sometimes we extend $\sim$ to the linear equivalence of $\bQ$-divisors in a usual way
(for $\bQ$-divisors, $A\sim B$ if $kA\sim kB$ as Cartier divisors for some $k>0$) but mostly
we use numerical equivalence of $\bQ$-divisors to avoid confusion.

Similarly, we define $\text{Z}_1(X)_{\bR}$ to be the group of $\bR$-linear combinations of irreducible curves in $X$, i.e., formal sums
$$\ga=\sum a_i C_i,\quad a_i\in\bR$$
with all $C_i\subseteq X$  irreducible curves. 
As in \cite{LazI}*{Def.~1.4.25}, we let
$$\Num_1(X)_{\bR}={\text{Z}_1(X)_{\bR}}/\equiv,$$
where for two one-cycle classes $\ga_1,\ga_2\in{\text{Z}_1(X)}_{\bR}$ we have numerical equivalence $\ga_1\equiv \ga_2$ if and only if 
$D\cdot\ga_1=D\cdot \ga_2$ for all Cartier divisors $D$ on $X$. 
It follows from the definitions that 
$$\Num^1(X)_{\bR}\otimes \Num_1(X)\ra\bR,\quad (\de,\ga)\mapsto \de\cdot \ga\in\bR$$  
is a perfect pairing, so $\Num^1(X)_{\bR}$ and $\Num_1(X)_{\bR}$ are dual vector spaces. 
In particular, both 
 $\Num^1(X)_{\bR}$ and  $\Num_1(X)_{\bR}$ are finite dimensional real vector spaces. 
 We define the pseudo-effective cone
$$\barEff(X)\subseteq \Num^1(X)_{\bR},$$ 
as the closure of the effective cone $\Eff(X)$, i.e., the convex cone generated by numerical 
classes of effective Cartier divisors (\cite{LazII}*{Def.~2.2.25}). We let 
$\Nef(X)\subseteq \Num^1(X)_\bR$ be the
cone generated by the classes of {\em nef divisors}.  
We define 
$$\barMov_1(X)\subseteq \Num_1(X)_{\bR}$$ 
the closure of the cone generated by numerical classes of  \emph{movable} $1$-cycles, see 
\cite{LazII}*{Def. 11.4.16}. The cones $\barEff(X)$ and $\barMov_1(X)$ are dual to each other. 
This was proved first in \cite{BDPP} 
for the case when $X$ is a smooth projective variety in characteristic $0$, 
but it holds in general. For $X$ irreducible projective variety 
over a field $k$ of characteristic $0$ this is proved in \cite{LazI}*{Thm. 11.4.19}. 
For the case of arbitrary characteristic, the same proof holds, see for example 
\cite{Fulger}*{Rmk. 2.1}.

\begin{definition}
A convex cone $\cC\subseteq \bR^s$ 
is called \emph{polyhedral} if there are finitely many vectors 
$v_1,\ldots v_s\in \bR^s$ such that 
$\cC=\bR_{\geq0}v_1+\ldots+\bR_{\geq0}v_s$. The cone is said to be \emph{rational polyhedral} if one can choose
the $v_i$'s in  $\bQ^s$.  
\end{definition}

\begin{lemma}\label{fvqevf}
Let $f:\,X\to Y$ be a surjective morphism of  normal projective irreducible varieties. 
If $\barEff(X)$ is (rational) polyhedral  then the same is true for $\barEff(Y)$.
\end{lemma}

\begin{proof}
Suppose $\barEff(X)$ is a (rational) polyhedral cone.
By the duality between the cones $\barEff(X)$ and $\barMov_1(X)$, 
it follows that 
$\barMov_1(X)$ is also a
(rational) polyhedral cone. 
The proper push-forward of $1$-cycles induces a map of $\bR$-vector spaces
$$f_*:{\Num_1(X)}_{\bR}\ra{\Num_1(Y)}_{\bR}.$$
 
 By \cite{FL}*{Cor. 3.12}, $f_*(\barMov_1(X))=\barMov_1(Y)$. The definitions of $\Num_1(X)$ and $\barMov_1(X)$ given in
\cite{FL} coincide with the ones given above, see 
\cite{FL}*{Section~2.1, Ex.~3.3}.
It follows that $\barMov_1(Y)$ is a (rational) polyhedral cone. Again by the duality between the cones $\barEff(Y)$ and $\barMov_1(Y)$, it follows 
that $\barEff(Y)$ is a (rational) polyhedral cone.
\end{proof}

We concentrate on the case of 
surfaces. The cone and contraction theorems hold in any characteristic
with very mild assumptions, see 
\cites{KM, Tanaka, FT, Fujino}.

\begin{proposition}\label{prop1}
\label{prop:nonpol}
Let $X$ be a normal projective $\bQ$-factorial
surface with Picard number at least $3$
and such that the cone $\barEff(X)$
is polyhedral. Then:
\begin{enumerate}
\item Every class $C\in \Num^1(X)$ of self intersection $0$ (or its opposite $-C$) is in
the relative interior of either the cone $\barEff(X)$ or its codimension one facet.
\item The effective cone $\Eff(X)$ is generated
by 
finitely many negative curves\footnote{A negative curve is an irreducible curve  $B$ with $B^2<0$.}.
In~particular, $\barEff(X)=\Eff(X)$ is a rational polyhedral cone.
\end{enumerate}
\end{proposition}

Part (2) of Proposition \ref{prop1} appears also in \cite{Nikulin}. 

\begin{proof}
(1) Fix $h$ an ample divisor. Let
\begin{equation}\label{Q}
Q := \{\,\omega\ \mid 
\ \omega^2 \geq 0,\ \omega\cdot h \geq 0\,\}\subseteq \Num^1(X)_\bR
\end{equation}
be the non-negative part of the light cone.
Then either $C$ or its opposite $-C$ lies on the boundary $\partial Q$. 
By Riemann-Roch, the cone $Q$ is contained in $\overline{\Eff}(X)$.
Since the Picard number of $X$ is
at least $3$, the cone $Q$ is round. In particular, 
$\partial Q$ can intersect only a facet of $\overline{\Eff}(X)$ of codimension~$1$ and only 
 in its relative interior.

(2) By (1), any  $\omega\in\Num^1(X)$ generating an extremal ray
of $\barEff(X)$ has $\omega^2 < 0$. 
By \cite{deb}*{Lemma 6.2(e)}\footnote{The proof in  \cite{deb}
is for smooth surfaces but the argument works verbatim in our case.}, 
for any such $\omega$
there exists an irreducible  curve $E$ such that $\omega$ is a positive multiple of the
class of $E$. 
\end{proof}

\begin{remark}
\label{rem:nonpol}
In the settings of Proposition~\ref{prop:nonpol},
if the class $C$ admits a positive integer
multiple $nC$ such that $|nC|$ is a base 
point free pencil, then $C$ is not big. Thus
it lies in the relative interior of a maximal facet 
$\tau$ of $\overline{\Eff}(X)$, and 
by the Hodge Index Theorem, 
the supporting hyperplane of $\tau$ is $C^\perp$.
In particular, any class of an irreducible curve $R$
which generates an extremal ray of $\tau$ satisfies
$R\cdot C = 0$, so that $R$ is an irreducible 
component of a fiber of the fibration $\pi\colon
X\to\mathbb P^1$ induced 
by $|nC|$. Since the contribution of the 
components of a fiber to the ``vertical''
rank of the Picard group is the number
of components minus one, it follows that
in order for $\overline{\Eff}(X)$ to be polyhedral
it must be 
\begin{equation}
 \label{eq:fib}
 1+\sum_{b\in\mathbb P^1}(|\text{Comp. of $f^{-1}(b)$}|-1)
 = 
 {\rm rk}({\rm Pic}(X)) - 1.
\end{equation}
\end{remark}

\begin{proposition}\label{prop2}
Let $X$ be a normal projective $\bQ$-factorial
surface with Picard number at least $3$.
Assume that $C\subseteq X$ is an irreducible curve with $C^2=0$
and $C\equiv -\alpha K_X$ 
with
$\alpha\in\mathbb Q_{>0}$. Then the following
are equivalent:
\begin{enumerate}
\item
There exist irreducible negative curves 
$B_1,\ldots, B_s$, 
that generate $C^{\perp}\subseteq\Num^1(X)_{\mathbb Q}$, and  such that
\begin{equation}\label{asrgasgas}
C\equiv a_1B_1+\ldots+a_sB_s\quad  \text{with}\quad  a_1,\ldots a_s\in\mathbb Q_{>0}.
\end{equation}
\item
$\barEff(X)$ is a rational polyhedral cone generated by 
negative curves.
\end{enumerate}
\end{proposition}

\begin{proof}
Proposition~\ref{prop:nonpol} gives $(2)\Rightarrow (1)$. 
We prove  $(1)\Rightarrow (2)$ under 
our additional assumptions. 
Note that $C$ (hence, $-K$) is nef. Recall 
that any $\omega\in\Num^1(X)_{\mathbb R}$ 
generating an extremal ray must have $\omega^2\leq0$ and  
if $\omega^2<0$ then $\omega$ is the class of a multiple of a curve 
 \cite{deb}*{Lemma 6.2}. 
The same is true when 
$\omega^2=0$, as if $\omega\cdot C=0$, by the Hodge Index theorem, $\omega$ and $C$ generate the same ray, while if
$\omega\cdot C>0$, then $\omega\cdot K<0$ and $\omega$ is generated by the class of a curve by the Cone theorem. 
Hence, it suffices to prove that $X$ contains
finitely many irreducible  curves $E$ with $E^2\leq0$ such that $E$ is not numerically equivalent
to a rational multiple of $C$. We can also assume that
 $E\neq B_i$ for all $i$.
 
We consider two cases.
If $E\cdot C=0$  then $E\cdot B_i=0$ for all $i$
by \eqref{asrgasgas} and by our assumption that
$E\neq B_i$ for all $i$. 
Since $B_1\ldots, B_s$ generate 
$C^\perp$ over $\mathbb Q$, $E$ must be numerically equivalent to a rational multiple of $C$,
which we have also ruled out.

Suppose $E\cdot C>0$.
Since $B_1\ldots, B_s$ generate 
$C^\perp$ over $\mathbb Q$, the classes which have fixed intersections
with the $B_i$'s form an affine subspace of dimension one in 
$\Num^1(X)_{\mathbb Q}$, differing one from another by a multiple of the class of $C$.
Since $E\cdot C>0$ and $C\cdot C=0$, there is at most one such class
with $E^2$ also fixed. 
Hence, it suffices to prove that 
$E\cdot B_i$ and $E^2$ belong to a finite set. 
By assumption~(1) and adjunction, we have 
$$\frac{1}{\al}\sum a_i(E\cdot B_i)=E\cdot(-K)\leq E^2+2\leq 2.$$
Hence, $0\leq E\cdot B_i\leq 2\al/a_i$. As there exists  $l\in\mathbb Z_{>0}$ (the index of $\Pic(X)$ in $\Cl(X)$) such that the $lD$ is Cartier for any curve 
$D$ (hence, $l(D\cdot E)$ is an integer), it follows that $E\cdot B_i$ belongs to a finite set. We have 
$-2\leq E^2$ by adjunction and nefness of~$-K$. As $E^2\leq0$, it follows similarly that $E^2$ must belong to a finite set.  
\end{proof}

\section{Elliptic pairs: general theory}\label{sDGSHSH}

As in \S~\ref{asrgasrg}, we work over an algebraically closed field $k$ of arbitrary characteristic.  
While Propositions~\ref{prop1} and \ref{prop2} address
polyhedrality  of $\barEff(X)$ for a general surface~$X$,
in this section we study polyhedrality further 
for a rational surface in the presence of a curve $C$ with self-intersection $0$ under some additional assumptions.

\begin{definition}\label{sGSRG}
An  {\em elliptic pair} $(C,X)$ consists of  a projective rational surface $X$
with log terminal singularities and an irreducible curve $C\subseteq X$, of  arithmetic genus one, disjoint from the singular locus of $X$ and such that $C^2=0$. 
 Let $C^\perp\subseteq\Cl(X)$ be the orthogonal complement to $C$. 
We~define the {\em restriction map}
$$\rd:\,C^\perp\to\Pic^0(C),\quad D\mapsto \cO(D)|_C.$$
Since $K\cdot C=0$ by adjunction, we can also define
the {\em reduced restriction map}
$$
\ored:\,\Cl_0(X):=C^\perp/\langle K\rangle\to\Pic^0(C)/\langle\rd(K)\rangle.
$$
We will often study a birational morphism $X\to Y$, which is an 
isomorphism in a neighborhood of $C$. We will then use notation 
$C_X$, $C_Y$, etc, to avoid confusion.
\end{definition}

The most familiar elliptic pairs are rational elliptic fibrations $X\to\bP^1$ with a fiber $C$
(which can be a multiple fiber).
However, we do not make this assumption. 
Note that as $X$ is rational,  $h^1(X,\cO_X)=0$, and 
hence $\Pic(X)_{\bQ}=\Num^1(X)_{\bQ}$. 

\begin{lemdef}\label{asfar}
We define the order $e=e(C,X)$ of the elliptic pair $(C,X)$ to be the  positive integer
satisfying any of the following equivalent conditions (or $\infty$ if none of them are~met):
\begin{enumerate}
\item $\rd(C)\in\Pic^0(C)$ is a torsion line bundle of order $e$.
\item $e$  is the smallest positive integer such that $h^0(C,\rd(eC))=1$.
\item  $e$  is the smallest positive integer such that $h^0(X,eC)=2$.
\item $e$ is the smallest positive integer such that $h^0(X,eC)>1$.
\end{enumerate}
The order $e(C,X)$ only depends on a Zariski neighborhood of $C$ in $X$. 
\end{lemdef}

\begin{proof}
The equivalence of (1) and (2) is clear. We use this as a  definition of $e$.
In~particular, $e(C,X)$ only depends on a Zariski neighborhood of $C$ in $X$.
Since log terminal singularities are rational and $C$ is disjoint from the singular locus of~$X$,
if~$\tilde X$ is a resolution of singularities of $X$, then $h^0(\tilde X,nC_{\tilde X})=h^0(X,nC_X)$ for any integer~$n$.
Hence, to prove the remaining equivalences we may assume that $X$ is smooth.
For~any $n\ge0$, we have $h^2(X,nC)=h^0(X,K_X-nC)=0$,
since otherwise $K_X$ would be effective.
Moreover, by Riemann-Roch $\chi(\cO_X(nC))=1$ for all $n$.  
Thus either $h^0(X,nC)=1$ and $h^0(C,\rd(nC))=0$
for every $n>0$, or for some $n>0$ we have
$h^0(X,nC)=2$, $h^0(C,\rd(nC))=1$ and $h^0(X,lC)=1$, $h^0(C,\rd(lC))=0$ for $1\le l<n$.
\end{proof}

\begin{lemma}\label{adfbafb} 
Suppose $(C,X)$ is an elliptic pair. Let $e=e(C,X)$. Then 
\begin{enumerate}
\item $e<\infty$ if and only if $C$ is a (multiple) fiber of a (quasi)-elliptic fibration\footnote{
If $C$ is smooth or if $\ch k\ne 2,3$ then the fibration is automatically elliptic \cite{BoMu}. If not, it can be quasi-elliptic, i.e.~have cuspidal generic fibers.}.
\item If $e=\infty$, then $C$ is rigid, which means that $h^0(nC)=1$ for all~$n>0$. In~this case
$\barEff(X)$ is not polyhedral if the Picard number $\rho(X)\geq3$.   
\end{enumerate}
\end{lemma}

\begin{proof}
Suppose $e<\infty$. Then 
$eC\sim \sum D_i$, for some irreducible curves $D_i\neq C$ by Lemma~\ref{asfar}~(3). 
As $C^2=0$, it 
follows that the $D_i$'s are disjoint from $C$ and $|eC|$ is a base-point-free  pencil. 
Since $C^2=K\cdot C=0$ by adjunction, $\varphi_{|eC|}:\,X\to\bP^1$ is a (quasi)-elliptic fibration. 
Suppose $e=\infty$. Then $C$ is rigid by Lemma \ref{asfar}~(4). 
By~Prop. ~\ref{prop:nonpol}, if $\barEff(X)$ is polyhedral and the Picard number of $X$ is at least $3$, then 
$\barEff(X)$ is generated by negative curves and $C$ is contained in the interior of a facet. Thus $h^0(X,kC)>1$ for some~$k$ and therefore
$e(C,X)<\infty$ by Lemma~\ref{asfar}~(4). 
\end{proof}

\begin{lemma}\label{adfharh}
If $(C,X)$ is an elliptic pair, then $K_X+C$ is an effective divisor. 
\end{lemma}

\begin{proof}
As $C$ is contained in the smooth locus of $X$, we can pass to a resolution of singularities and prove for a smooth surface ~$X$ that $h^2(-C)=h^0(K+C)>0$. 
By~adjunction $\cO_X(K+C)|_C\simeq\omega_C\simeq\cO_C$, so there is an exact 
sequence\footnote{
This trick is from the proof of the canonical bundle formula for elliptic fibrations in~\cite{BoMuII}}
$$0\to \cO_X(K)\to\cO_X(K+C)\to\cO_C\to 0.$$
The statement follows from the vanishing $h^0(X,K)=h^1(X,K)=0$.\end{proof}

\begin{definition}
We say that $(C,X)$ is a {\em minimal} elliptic pair if it does not contain 
irreducible curves $E$
such that $K\cdot E<0$ and $C\cdot E=0$.
\end{definition}

\begin{remark}\label{wfvwevw}
A curve $E$ as in the definition must have $E^2<0$.  Indeed,  $E^2\leq0$
by the Hodge Index theorem, with equality if and only if the classes of $C$ and $E$ are multiples of each other.
But since $E\cdot K<0$ and $C\cdot K=0$, the latter is not possible. 
Moreover, $E$ is a rational curve \cite{Fujino}*{Thm~ 5.6}.  By~the contraction theorem, there exists a
morphism $\phi: X\ra Y$ contracting only~$E$. As $\phi$ is an isomorphism in a Zariski neighborhood of $C$ and $Y$ is log terminal, $(C,Y)$ is  an elliptic pair. 
Moreover, $K_X\equiv\phi^*K_Y+aE$, for some $a\in\mathbb Q$.
 Since $E\cdot K_X<0$ and $E^2<0$, it follows 
that $a>0$. Furthermore, $K^2_X<K^2_Y$.  
\end{remark}

\begin{lemma}\label{arharhadj}
Let $(C,X)$ be an elliptic pair. The following conditions are equivalent:
\begin{enumerate}
\item $(C,X)$ is minimal;
\item $K+C$ is nef;
\item $C\sim\alpha(-K)$, for some $\alpha\in\mathbb Q_{>0}$, a linear equivalence of $\bQ$-divisors;
\item $K^2=0$.
\end{enumerate}
\end{lemma}

\begin{proof}
To prove $(1)\Rightarrow (2)$, assume that $K+C$ is not nef.
By the cone theorem\footnote{Note that there are no singularity assumptions on $K+C$ in the cone theorem for surfaces.} 
for a log surface $(X,C)$ \cites{Tanaka, Fujino},
there is an irreducible curve $E$ such that $(K+C)\cdot E<0$ and $E^2<0$. 
Since $K+C$ is effective, $E$ must be one of its components. Since $C\cdot (K+C)=0$ and $C$ is nef,  
we must have $C\cdot E=0$, hence, $K\cdot E<0$. This contradicts the minimality of $(C,X)$.

Next we prove
$(2)\Rightarrow (3)$. Since 
$(K+C)\cdot C=0$, by the Hodge Index theorem we must have 
$(K+C)^2\leq0$. But since $K+C$ is nef, $(K+C)^2\geq0$.
Thus $(K+C)^2=0$ and it must be that $K+C\equiv\lambda C$, for some $\lambda\in\mathbb Q$. As no multiple of $K$ is effective,
it follows that $C\equiv\alpha(-K)$, for some $\alpha\in\mathbb Q_{>0}$. 
Since $X$ is rational, in fact $C\sim\alpha(-K)$,  a linear equivalence of $\bQ$-divisors;

The  implication $(3)\Rightarrow (4)$ is clear.
To see  $(4)\Rightarrow (1)$, suppose $(X,C)$ is not minimal.
By Remark~\ref{wfvwevw},  
there is a contraction $\phi: X\ra Y$ of a curve
$E$ such that $K\cdot E<0$, $E^2<0$ and $C\cdot E=0$.
Moreover, $K_Y^2>K_X^2=0$. But 
$(C,Y)$ is an elliptic pair and 
so $K_Y^2\le 0$ by the Hodge Index theorem, which gives a contradiction.
\end{proof}

\begin{theorem}\label{sHAHA}
Let $(C,Z)$ be an elliptic pair with smooth $Z$. Then  $(C,Z)$ is minimal if and only if 
$\rho(Z)=10$, or equivalently, $K^2=0$.
If $(C,Z)$ is minimal then
\begin{enumerate}
\item[(i) ] $C\sim n(-K)$ for some positive integer $n$; 
\item[(ii) ] $Z$ is a blow-up of $\bP^2$ at $9$ points (possibly infinitely near) and the intersection pairing on $Z$ makes $\Cl_0(Z)$ isomorphic to the negative definite lattice $\bE_8$.
\end{enumerate}
Suppose that $(C,Z)$ is minimal and 
$e(C,Z)<\infty$. The following  are equivalent:
\begin{enumerate}
\item $\oEff(Z)$ is polyhedral and generated by $(-2)$ and $(-1)$ curves.
\item $\oEff(Z)$ is polyhedral.
\item $\Ker(\ored)\subseteq \bE_8$ contains $8$ linearly independent roots of~$\bE_8$.
\end{enumerate}
\end{theorem}

\begin{remark}
Smooth projective rational surfaces $Z$
for which there is an integer $m>0$ such that the linear system $|-mK_Z|$ is base-point free and of dimension $1$, are called 
Halphen surfaces of \emph{index} $m$ and have been studied from many different points of view, see for example \cites{ADHL, CD, Grivaux}. If $(C,Z)$ is an elliptic pair as in Thm. \ref{sHAHA}, 
then $Z$ is a Halphen surface with index $n\cdot e$, where $e:=e(C,Z)$. Let $N$ be the sub-lattice of $\bE_8$ that is generated by roots contained in $\Ker(\ored)$, i.e., $N$ is  generated by the classes of all the $(-2)$ curves on $Z$ (see the proof of Thm. \ref{sHAHA}). By the Hodge Index Theorem, the $(-2)$ curves on $Z$  are precisely the irreducible components of reducible fibers of the fibration induced by the linear system $|eC|$, call them $S_1,\ldots, S_\lambda$. If $\mu_j$ denotes the number of irreducible components of $S_j$, the rank of $N$ (i.e., the maximum number of linearly independent roots of~$\bE_8$ contained in $\Ker(\ored)$, or equivalently, the maximum number of  $(-2)$ curves that are linearly independent modulo $K_Z$) equals $\sum_{i=1}^\lambda(\mu_i-1)$. 
By a result of Gizatullin, $\sum_{i=1}^\lambda(\mu_i-1)<8$ if and only if the automorphism group $\Aut(Z)$ is infinite; in this case there exists a free abelian group $G$ of rank $8-\sum_{i=1}^\lambda(\mu_i-1)$, of finite index in $\Aut(Z)$, such that any non-zero element in $G$ is an automorphism that acts by translation on each fiber of the elliptic fibration \cite{Grivaux}[Thm. 7.11, Cor. 7.12], i.e., $\oEff(Z)$ is not polyhedral if and only if $\Aut(Z)$ is infinite. 
\end{remark}

\begin{proof}[Proof of Thm. \ref{sHAHA}]
By Lemma \ref{arharhadj}, the elliptic pair $(C,Z)$ is minimal if and only if $K^2=0$. 
Since $Z$ is a smooth rational surface, it is an iterated blow-up of $\bP^2$ or a Hirzebruch surface $\mathbb F_e$. 
As $K^2$ goes down by one and the Picard number goes up by one when blowing-up a smooth point,  $K^2=0$ 
if and only if $\rho(Z)=10$. We claim that $Z$ is the blow-up of $\bP^2$ at $9$ points. Assume not. Then $Z$ is the iterated blow-up of a Hirzebruch surface $\mathbb F_e$ ($e=0$ or $e\geq2$) at $8$ points. A negative curve $B$ on $\mathbb F_e$ has $B^2=-e$ and 
$B^2$ goes down by blow-up. By adjunction, and since $-K_Z$ is nef,  the only negative curves on $Z$ are $(-1)$ and $(-2)$-curves, so we must have $e=0$, or $e=2$ and none of the blown up (possibly infinitely near) points on $\mathbb F_2$ lie on the negative section. If $e=0$, we are done, as $\mathbb \Bl_p\bP^1\times\bP^1$ is isomorphic to the blow-up of $\bP^2$ at $2$ points. If $e=2$, we are also done, as a blow-up of $\mathbb F_2$ at a point not lying on the negative section, is isomorphic, via an elementary transformation, to a blow-up of $\mathbb F_1$ at one point. This proves the claim. It follows that $\Cl_0(Z)\cong\bE_8$. Since $-K$ is a primitive vector of $\Pic(Z)$, it follows by Lemma \ref{arharhadj}(3) that  $C\sim n(-K)$ for some integer $n>0$.


 
Suppose that $(C,Z)$ is minimal and 
$e=e(C,Z)<\infty$. 
By Lemma~\ref{adfbafb}, $|eC|$ gives a (quasi)-elliptic fibration $Z\rightarrow \bP^1$. Clearly  $(1)\Rightarrow (2)$ and Prop.~\ref{prop1} (2) implies $(2)\Rightarrow (1)$, as 
the only negative curves are $(-1)$ and $(-2)$ curves. Assume (1). By Proposition \ref{prop2}, $C\equiv\sum a_iB_i$ for $a_i\in\mathbb Q_{>0}$, with irreducible negative curves $B_i$ 
generating $C^\perp$ over $\mathbb Q$. 
Since $B_i$ is irreducible, $\rd(B_i)=0$.
Since $B_i\cdot K=0$, each
$B_i$ is a   $(-2)$ curve. Since the curves $B_i$ generate $C^\perp$ over~$\mathbb Q$, 
eight of them are linearly independent modulo $K$. This proves (3). 

Assume (3). Let $\be_1,\ldots, \be_8$ be $(-2)$-classes in $C^{\perp}$, linearly independent modulo~$K$, 
and such that $\ored(\be_i)=0$. Adding to each $\beta_i$ a multiple of $K$, we may assume that each $\be_i$ restricts trivially to $C$.
We claim that, for each $i$, either $\beta_i$ or $(K+C)-\beta_i$ is effective. Indeed, for each $\be:=\be_i$ we have a short exact sequence
$$0\to\cO(\beta-C)\to\cO(\beta)\to\cO_C\to0.$$
If $\beta$ is not effective, $\beta-C$ is not effective either. Hence, $h^1(Z,\cO(\beta-C))>0$.
But $\chi(\cO(\beta-C))=0$ by Riemann--Roch.
Thus $h^2(Z,\cO(\beta-C))>0$ and so $(K+C)-\beta$ is effective.  
We have found $8$ effective divisors $D_1,\ldots,D_8$ with $\rd(D_i)=0$,  $D_i^2=-2$ and linearly independent modulo $K$. 
Each of the divisors $D_i$ belongs to a union of the fibers of the (quasi)-elliptic fibration (and no $D_i$ is a rational multiple of $C$). Since $(-2)$-curves are precisely the irreducible components of reducible fibers, it follows that $(-2)$-curves generate $C^\perp$ over $\bQ$. Clearly, for some integer $l\gg 0$, $lC$ is an effective combination of $(-2)$ curves. Then Prop.~\ref{prop2}~(1) implies (2).
\end{proof}


\begin{theorem}\label{asfhadrhad}
For any elliptic pair $(C,X)$, there exists a minimal elliptic pair 
$(C,Y)$ and a morphism $\pi:\,X\to Y$, which is an isomorphism over a neighborhood~of~$C$.
Consider the Zariski decomposition on $X$ of $K+C$,
$$K+C\sim N+P,\quad N=a_1C_1+\ldots+a_sC_s,\quad a_i\in\bQ_{>0},$$
the linear equivalence of $\bQ$-divisors\footnote{Recall that the $C_i$'s are irreducible curves with a negative-definite
intersection matrix and $P$ is a nef effective $\bQ$-divisor such that $P\cdot C_i=0$ for all $i$. 
The $\bQ$-divisor $N$ is determined uniquely.}.
Then: 
\begin{enumerate}
\item $Y$ is obtained by contracting curves $C_1,\ldots,C_s$ on $X$.
\item $P\equiv0$ if and only if $-K_Y\sim C_Y$, in which case 
$N$ is an integral combination of $C_1,\ldots,C_s$ and
$Y$ has Du Val singularities.
\end{enumerate}
\end{theorem}

\begin{definition}
We call an elliptic pair $(C,Y)$ a {\em minimal model} of $(C,X)$.
\end{definition}

\begin{corollary}\label{kjsHFkjshf}
Let $(C,Y)$ be a minimal model of an elliptic pair $(C,X)$ such that  $e(C,X)=\infty$. Then $Y$ has Du Val singularities. Consider the Zariski decomposition
$$K+C\sim N+P$$
on $X$. Then $P\sim 0$ and 
$K+C\sim N$ is an \underline{integral} effective
combination of irreducible curves $C_1,\ldots,C_s$ with a negative-definite
intersection matrix. The minimal model $Y$ is obtained by contracting curves $C_1,\ldots,C_s$ and $C_Y\sim -K_Y$.
\end{corollary}

\begin{proof}
We first prove the theorem and then its corollary.
We  obtain a minimal model $\pi: X\ra Y$ by running a $(K+C)$-MMP \cites{Tanaka, Fujino}.
Equivalently (by~Lemma~\ref{arharhadj}), 
$\pi$ is a composition of contractions of the form
$\phi: X\ra Y$, where each $\phi$ is the contraction of a $K$-negative curve $E$ such that $E\cdot C=0$.
On each step, $K_X+C_X\sim \phi^*(K_Y+C_Y)+aE$, with $a\in\mathbb Q_{>0}$, a linear equivalence
of $\bQ$-divisors.  
At the end we obtain that $K_Y+C_Y$ is nef, i.e., $(C, Y)$ is minimal.   
If the curves contracted by $\pi$ are $C_1,\ldots,C_s\subseteq X$, then 
$K_X+C_X\sim N+P$, with 
$$P=\pi^*(K_Y+C_Y),\quad N=\sum_{i=1}^s a_i C_i\quad a_i\in\bQ_{>0},$$
  a linear equivalence
of $\bQ$-divisors.  The divisor
$P$ is nef and effective (Lemma~\ref{adfharh}) and $P\cdot C_i=0$ for all $i$. Hence, this is the Zariski decomposition of 
$K+C$. Moreover, $P\equiv0$ if and only if $K_Y+C_Y\sim 0$. 

Assume now that $P\equiv 0$. Recall an algorithm 
for computing the Zariski decomposition \cite{bauer}. Write $K+C\sim b_1B_1+\ldots+b_tB_t$ 
as an integral, effective sum of irreducible curves $B_i$. 
Let $N':=\sum x_i B_i$, where $0\leq x_i\leq b_i$ are maximal such that $P':=\sum (b_i-x_i)B_i$ intersects all $C_i$ non-negatively.
Then $N'$ and $P'$ give a Zariski decomposition of $K+C$. Since $N=N'$ is unique
and  $P'\equiv P\equiv0$,  the Zariski decomposition is $K+C\sim b_1B_1+\ldots+b_tB_t$.
To prove the singularity statement, note that
$-K_Y\sim C_Y$ implies that $K_Y$ is Cartier. Thus $Y$ has Du Val singularities.

Finally, we prove the corollary. Suppose that $e(C,X)=e(C_Y,Y)=\infty$. If~$P\not\equiv 0$, we have
$C_Y\sim \al(-K_Y)$, for some $\al\in\mathbb Q$, $\al\neq 1$. Then
$C_Y\sim \frac{\alpha}{\alpha-1}(K_Y+C_Y)$, a linear equivalence of $\bQ$-divisors.
But $K_Y+C_Y$ restricts trivially to $C_Y$ by adjunction, and therefore
 $\rd(C)$ is torsion, which is a 
contradiction. So we must have $P\equiv0$ and this finishes the proof of the corollary by (1)--(2) of the theorem. 
\end{proof}

\begin{remark}
We give an example of a minimal rational elliptic fibration that does not 
satisfy $C\sim-K$. Let~$W$ be a minimal smooth rational elliptic fibration
with a nodal fiber $I_0$. Blow-up the node of the fiber and contract the proper transform of the fiber
(which has self-intersection $-4$). This produces a minimal rational 
elliptic fibration $Y$ with a $\frac14(1,1)$ singularity, which is log terminal.
The fiber $C_0$ through the singularity is a nodal multiple fiber of multiplicity $2$. We have 
$C\sim 2C_0\sim-2K$.
\end{remark}

\begin{lemma}\label{lkJSBwkjbg}
Let $(C,Y)$ be an elliptic pair such that $Y$ has Du Val singularities.
Let $\pi:\,Z\to Y$ be its minimal resolution. 
\begin{enumerate}
\item $(C,Y)$ is minimal if and only if $(C,Z)$ is minimal.
Equivalently,
$$\rho(Y)=10-R,$$
where $R$ is the rank
of the root system associated with singularities of $Y$. 

\item
Assume  $(C,Y)$ is a minimal elliptic pair. 
Then the following are equivalent:
\begin{itemize} 
\item $\barEff(Y)$ is a polyhedral cone;
\item $\barEff(Y)$ is a rational polyhedral cone;
\item $\barEff(Z)$ is  a polyhedral cone.
\end{itemize}
\end{enumerate}
When $\rho(Y)=2$, all the above statements hold. 
\end{lemma}

\begin{proof}
As $K_Z=\pi^*K_Y$, the pair $(C,Z)$ is minimal if and only if $(C,Y)$ is minimal
by Lemma \ref{arharhadj}. As 
$\rho(Y)=\rho(Z)-R$, the first statement follows. 
If $\barEff(Z)$ is (rational) polyhedral then $\barEff(Y)$ is (rational) polyhedral by Lemma \ref{fvqevf}.
Assume now $\barEff(Y)$ is polyhedral. If $\rho(Y)\geq3$,  then $e(C,Y)<\infty$, by Lemma \ref{adfbafb}
and by Proposition \ref{prop1}(1), $\barEff(Y)$ is a rational polyhedral cone with $C_Y$ contained in the interior
of a maximal facet. If $\rho(Y)=2$ (the smallest possible), then $\Eff(Y)$ is a rational polyhedral cone by the 
Cone theorem (it~is spanned by the class of $C$ and by the class of the unique negative curve).
Note that this doesn't provide any information about $e(C,Y)$.
In both cases, it~follows that $C_Y^\perp$ contains $\rho(Y)-2$ effective divisors 
which are linearly independent modulo $K_Y$ and restrict trivially to $C$. 
As $\Cl(Z)_{\mathbb Q}$ decomposes as  $\pi^*\Cl(Y)_{\mathbb Q}\oplus T_{\mathbb Q}$,
where $T$ is a sublattice spanned by classes of $(-2)$-curves over singularities of~$Y$,
we have $(C_Z^\perp)_{\mathbb Q}=(\pi^*C_Y^{\perp})_{\mathbb Q}\oplus T_{\mathbb Q}$.
It  follows that $C_Z^\perp$ contains $\rho(Y)-2+R=8$ effective divisors which are linearly independent modulo $K_Z$
and restrict trivially to $C$. As in the proof of Theorem \ref{sHAHA}, it follows that $\barEff(Z)$ is a polyhedral cone. 
\end{proof}

\begin{remark}\label{index e}
In the set-up of Lemma \ref{lkJSBwkjbg}, if $(C,Y)$ has Du Val singularities and 
$C_Y\sim -K_Y$, with $\pi:Z\to Y$ is its minimal resolution, then $Z$ is a 
Halphen surface of index $e(C,Z)$, as $C_Z\sim -K_Z$. Indeed, this follows from 
$\pi^*K_Y=K_Z$, $\pi^*C_Y=C_Z$. 
\end{remark}

\begin{definition}
Let $(C,X)$ be an elliptic pair such that  the minimal model $(C,Y)$ has Du Val singularities.
Let $\pi: Z\ra Y$ be the minimal resolution of $Y$. Let
$$T\subseteq \bE_8=\Cl_0(Z)$$ 
be a root sublattice spanned by classes of $(-2)$-curves over singularities of~$Y$.
We~call $T$ {\em the root lattice of $(C,X)$} and we denote by
$\hat T$ its saturation~$\bE_8\cap(T\otimes\bQ)$. 

The push forward $\pi_*: \Cl(Z)\ra\Cl(Y)$ induces 
a map $\Cl_0(Z)\ra \Cl_0(Y)$ with kernel $T$, i.e., $\Cl_0(Y)\simeq \bE_8/T$ and the map $\ored_Z$ factors through
$\ored_Y$. Moreover, 
$$\Cl_0(Y)/\hbox{\rm torsion}\simeq\bE_8/\hat T.$$ 
The intersection pairing on $Y$ and pull back
of $\bQ$-divisors realizes $\bE_8/\hat T$ as a sublattice of the vector space
$(T\otimes\bQ)^\perp\subseteq \bE_8\otimes\bQ$ with the intersection pairing on~$Z$. 
\end{definition}

\begin{remark}\label{sGASRHA}\label{asgasrh}
Root lattices $T\subset \bE_8$ were classified by Dynkin
\cite{Dynkin}*{Table~11}. The quotient group $\Cl_0(Y)\simeq \bE_8/T$ was computed, e.g.,  in \cite{OS}.
\end{remark}

\begin{corollary}\label{concrete}
Let $(C,Y)$ be a minimal elliptic pair with Du Val singularities and 
$\rho(Y)\geq3$. 
Let $R$ be the rank of the root lattice of $(C,Y)$ and suppose $e(C,Y)<\infty$.
Then  $\oEff(Y)$ is polyhedral
if and only if there are roots  $\beta_1,\ldots,\be_{8-R} \in\bE_8\setminus\hat T$, linearly independent modulo $\hat T$ and such that
$\ored(\beta_i)=0$. In particular, if 
$R=7$ then $\oEff(Y)$ is polyhedral
if and only if $\ored(\beta)=0$ for some root $\beta\in\bE_8\setminus\hat T$.
\end{corollary}


The reader will notice a  discrepancy between Corollary~\ref{concrete}, 
which provides an effective criterion of polyhedrality for minimal elliptic pairs $(C,Y)$ with Du Val singularities
and $e(C,Y)<\infty$
and Corollary~\ref{kjsHFkjshf}, which shows that a minimal model $(C,Y)$  
of an elliptic pair $(C,X)$ with  $e(C,X)=\infty$ has Du Val singularities. These disjoint scenarios are reconciled in the following definition:

\begin{definition}\label{eSGsgSG}
Let $(C,X)$ be an elliptic pair with $e(C,X)=\infty$ defined over $K$, a finite extension of $\bQ$.
Let $R\subset K$ be the corresponding ring of algebraic integers. There exists an open subset $U\subset\Spec R$
and a pair of schemes $(\cC,\cX)$ flat over $U$, which we call an \emph{arithmetic elliptic pair of infinite order}, such that
\begin{itemize}
\item Each geometric fiber $(C,X)$ of $(\cC,\cX)$ is an elliptic pair of order $e_b$ which depends only on the corresponding point $b\in U$. We have $e_b<\infty$ for $b\ne0$.
\item The contraction morphism $X\to Y$ to the minimal model extends to the contraction of schemes $\cX\to\cY$ flat over $U$.
\item All geometric fibers $(C,Y)$ of $(\cC,\cY)$ over $U$ are minimal elliptic pairs with Du Val singularities
and the same root lattice $T\subset\bE_8$. 
\end{itemize} 
Let $X$, $Y$ be geometric fibers over a place $b\in U$, $b\ne 0$.
We call $b$ a \emph{polyhedral prime} if $\oEff(Y)$ is polyhedral.
If $b$ is not polyhedral then $\oEff(X)$ is also not polyhedral.
\end{definition}

Distribution of polyhedral primes is an intriguing question in arithmetic geometry
that we will start to address for arithmetic toric elliptic pairs.


\section{Lang--Trotter polygons and toric elliptic pairs}\label{asfasgsrh}

At the beginning we work over an algebraically closed field $k$ of any characteristic.
We recall that a polygon $\Delta\subseteq \bR^2$
is called a lattice polygon if its vertices 
are in~$\mathbb Z^2$.
If $\Delta$ is a lattice polygon, we will 
denote by $\vol(\Delta)$ its 
{\em normalized volume},
i.e.~twice its euclidean area
(so that $\vol(\Delta)$ is always a
non-negative integer).
We~recall that given any Laurent 
polynomial 
\begin{equation}\label{sdcqfsv}
f = \sum_{u\in \mathbb Z^2}\alpha_ux^u
 \in k[x_1^{\pm 1},x_2^{\pm 2}],
 \end{equation}
 where $x^u
:= x_1^{u_1}x_2^{u_2}$,
we can construct a lattice polygon $\NP(f)$, 
called the {\em Newton polygon} of $f$, 
by taking the convex hull of the points 
$u\in \mathbb Z^2$ 
such that $\alpha_u\neq 0$.

A lattice polygon $\Delta$
defines a morphism
\[
 g_\Delta\colon \bG_m^2
 \to\bP^{|\Delta\cap \mathbb Z^2|-1},
 \qquad
 x \mapsto [x^u : u\in \Delta\cap \mathbb Z^2],
\]
where $x = (x_1,x_2)\in(k^*)^2$.
We will denote by $\bP_{\Delta}$
the projective toric surface defined by $\Delta$,
i.e. the closure of the image of $g_\Delta$, and 
by $e\in\bP_{\Delta}$ the image $g_\Delta(1,1)$.
A hyperplane section is denoted by
$H_\Delta$. 
The linear system $|H_\Delta|$ is denoted by $\cL_{\Delta}$,
  and, given a positive integer $m$, we let
  $\cL_{\Delta}(m)$ to be the subsystem of $\cL_{\Delta}$ 
  consisting of the curves having
  multiplicity at least $m$ at $\e$.
  We will denote 
by $\pi_\Delta \colon X_{\Delta}\to \bP_{\Delta}$ 
the blowing up at $e\in\bP_{\Delta}$
and by $E$ the exceptional divisor of~$\pi_\Delta$.

%

\begin{notation}\label{werfwev2e}
Given a triple  
$(\Delta, m, \Gamma)$
where $\Delta$ is a lattice polygon, $m$ a positive integer and $\Gamma\in\cL_{\Delta}(m)$,
the curve $\Gamma$ is given by a Laurent polynomial \eqref{sdcqfsv}
and the curve $V(f)=\Gamma\cap\bG_m^2$ will also be denoted by $\Gamma$.
We denote by $C$ the proper transform of $\Gamma$ in $X_\Delta$.
In~this section we will investigate properties of pairs $(C,X_\Delta)$.
We drop the subscript $\Delta$ from notation $\bP_{\Delta},  X_\Delta$ ~if no confusion arises.
\end{notation}


\begin{proposition}\label{prop:gen}
Consider a triple 
$(\Delta, m, \Gamma)$ as in Notation~\ref{werfwev2e}.
Suppose $\Gamma$ is irreducible and its Newton polygon is~$\Delta$.
The following hold:
\begin{itemize}
\item[(i)] the arithmetic genus of $C$ is
\[
 p_a(C) = \frac12
 \left(
 \vol(\Delta) - m^2 + m - |\partial \Delta \cap \mathbb Z^2|
 \right) + 1;
\]
\item[(ii)] 
any edge $F$ of $\Delta$ of lattice length $1$ gives a smooth  point $p_F\in C$ defined as 
the~intersection of $C$ with the  toric boundary divisor corresponding to $F$. This point is defined over 
the field of definition of $\Gamma$.
\end{itemize}
\end{proposition}

\begin{proof}
Since $\Delta$ is the Newton polygon of $\Gamma$,
$\Gamma\subseteq\bP$ does not
contain any torus-invariant point of $\bP$. In particular,
$\Gamma$ is contained in the smooth
locus of $\bP$, and hence $C$ is contained in
the smooth locus of $X$. 
By adjunction formula,
\[
p_a(C) = \frac12(C^2 + C\cdot K_X) + 1 =  
\frac12(\vol(\Delta) - m^2 + C\cdot K_X) + 1,
\]
where the second equality follows 
from~\cite{CLS}*{Prop.~10.5.6}. 
But $C\cdot K_X = 
\Gamma\cdot K_{\bP} + m$, so that 
in order to prove (i) we only need
to show that 
\begin{equation}\label{qqefv2ev}
\Gamma\cdot K_{\bP} = 
- |\partial \Delta \cap \mathbb Z^2|.
\end{equation}
Observe that $-K_{\bP}$
is the sum of all the prime invariant divisors of $\bP$
and each prime invariant divisor $D\subseteq
\mathbb P$ corresponds to an edge $F$ of $\Delta$, see \cite{CLS}*{Prop.~10.5.6}. 
Let us fix such an edge $F$.
By a monomial change of variables, we can assume that $F$ lies on
the $x_2$ axis. The inclusion of algebras $k[x_1,x_2^{\pm1}] \to 
k[x_1^{\pm1},x_2^{\pm1}]$
gives the inclusion $\bG_m^2 \to \bG_m\times\bA^1$, and 
$V(x_1)\subseteq\bG_m\times\bA^1$ is an affine open subset of
$D$. Since $\Gamma$ does not
contain any torus-invariant points of $\bP$, 
$\Gamma\cap D = \Gamma\cap V(x_1)$,
and the latter intersection has equation
\begin{equation}\label{sdcqfvqfe}
 f|_F :=
 \sum_{u\in F\cap\mathbb Z^2}\alpha_ux^u 
 = f(0,x_2)
 = 0.
\end{equation}
The degree of this Laurent polynomial 
is the lattice length of $F$, so that 
\eqref{qqefv2ev} holds.

Moreover, if $F$ has length $1$, the equation \eqref{sdcqfvqfe} has degree $1$, 
which means that $\Gamma$ intersects 
the prime divisor $D$ transversally
at a smooth point $p_F\in\Gamma$. Since $D$ is defined over the base field,
if $\Gamma$ is defined over a subfield $k_0\subset k$ then so is $p_F$.
\end{proof}


\begin{definition}
 \label{def:good}
 Let $\Delta\subseteq \bR^2$ be a lattice polygon
 with at least $4$ vertices. We say that $\Delta$ is {\em good}, 
 if, for some integer $m$, the
 following hold:
 \begin{itemize}
 \item[(i)] $\vol(\Delta) = m^2$;
 \item[(ii)] $|\partial\Delta\cap \bZ^2| = m$;
 \item[(iii)] $\dim\cL_{\Delta}(m) = 0$, and the only 
 curve $\Gamma\in\cL_{\Delta}(m)$ is irreducible;
  \item[(iv)] the Newton polygon of $\Gamma$ coincides with $\Delta$;
 \end{itemize} 
A good polygon is said to be:
\begin{itemize}
    \item a {\em Halphen polygon} if 
    $\rd(C)=\cO_X(C)|_C$ is torsion;
    \item a {\em Lang--Trotter polygon} if
    $\rd(C)=\cO_X(C)|_C$ is not torsion.
\end{itemize}
\end{definition}

\begin{theorem}\label{goodpairsthm}
If $\Delta$ is a good  polygon then $(C,X_\Delta)$ is an elliptic pair
(we~call it a \underline{toric elliptic pair}),
$e(C,X_\Delta)>1$, and  $C$ is defined over the base field.
If $\Delta$ is Lang--Trotter then $\ch k=0$, $e(C,X_\Delta)=\infty$,
and $\barEff(X_\Delta)$ is not polyhedral.
\end{theorem}

%

%

\begin{proof} 
Let $\Delta$ be a good lattice polygon.
The curve $\Gamma$ is irreducible by (iii) and does not pass through the torus-invariant points of $\bP_\Delta$
by (iv). It follows that $C$ is contained in the smooth locus of $X_\Delta$.
Toric surface singularities, i.e.~cyclic quotient singularities, are log terminal.
By (iv) and~\cite{CLS}*{Prop.~10.5.6}, 
$\Gamma^2 = \vol(\Delta)$, so that 
(i) is equivalent to $C^2=0$.
Finally, conditions (i) and (ii), together with 
Prop.~\ref{prop:gen},
imply that $p_a(C) = 1$.
Thus  $(C,X_\Delta)$ is an elliptic pair.
Observe that
$\mathcal O_X(C)|_C = \rd(C) \in \Pic^0(C)$ 
(see Definition~\ref{sGSRG}), so that 
being Lang--trotter is equivalent to $e(C,X_\Delta)=\infty$. 
Suppose this is the case.
Since $\dim\cL_{\Delta}(m) = 0$, the curve $\Gamma$, and thus also the curve $C$,
and thus also the line bundle $\cO_X(C)|_C$, are all defined over the base field. 
In  characteristic~$p$, the group $(\Pic^0C)(\bF_p)$ is torsion,
which contradicts $e(C,X_\Delta)=\infty$.
 Thus $\ch k=0$.
Since $\Delta$ has at least $4$ vertices, $\rho(X_\Delta)\ge3$ and 
$\barEff(X)$ is 
not polyhedral
by Lemma~\ref{adfbafb}.
\end{proof}

\begin{remark}
We don't know examples of Lang--Trotter quadrilaterals.
Indeed, in the following proposition we are going to prove that they don't exist, under the additional
hypothesis that the multiplicity $m$ coincides with the lattice width of the polygon. 
\end{remark}

\begin{proposition}
There are no Lang--Trotter
quadrilaterals $\Delta$ such
that $m = {\rm width}(\Delta)$. \end{proposition}
\begin{proof}
Assume $\Delta$ is a good quadrilateral
and let $(C,X_\Delta)$ be the corresponding 
elliptic pair. The divisor $K+C$ is linearly equivalent 
to an effective one whose components
in the support are in $C^\perp$.
In particular, if $C^\perp$ contains the classes of 
two negative curves $R_1, R_2$ then, 
being this space two-dimensional,
there are integers $a_i$, with $a_0\neq 0$, 
such that $a_0C+a_1R_1+a_2R_2 \sim 0$.
Taking restriction to $C$ one deduces 
that $\Delta$ is not Lang--Trotter.
If $K+C\sim \alpha R + \beta C$, with
$\alpha,\beta\in\mathbb Q_{>0}$,
then, after cleaning denominators and
restricting to $C$ one again concludes
that $\Delta$ is not Lang--Trotter.
If $K+C\sim 0$, then by considering multiplicities at $e$, we must have
$m=1$, which is impossible since $m=|\partial\Delta\cap \bZ^2|\geq4$.
It remains to analyze the case  
$K+C\sim nR$, with $n>0$ and $R$ 
is an irreducible curve in $C^\perp$.
Since we are assuming that $m = {\rm width}(\Delta)$, 
the class of the one-parameter subgroup defined 
by one width direction lies in $C^\perp$, 
so that $R$ must be this class, and in particular its
Newton polygon is a segment of lattice length $1$.
Moreover, by considering multiplicities at $e$, it must be that $n = m-1$, so that $\Delta$ has $m$ 
interior lattice points, lying on a line. 
If $\Delta$ were Lang--Trotter, Pick's formula
would give $m^2 = 3m - 2$, which has 
integer solutions $m=1,2$, but this is again impossible. 

\end{proof}

\begin{example}
The polygon $\Delta$ with vertices
\[
\left[
\begin{matrix}
0 & 12 & 14 & 9\\
0 & 4 & 5 & 15
\end{matrix}
\right]
\]
satisfies the conditions $\vol(\Delta) = 169$,
$|\partial\Delta\cap {\mathbb Z}^2| = 13$ and 
$\cL_\Delta(13)$ contains only one curve $\Gamma$,
irreducible. Moreover, ${\rm width}(\Delta) = 14$, 
so that this is an example of good polygon with 
$m < {\rm width}(\Delta)$. In particular, the 
proof of the above proposition does not apply. 
Nevertheless it is possible to show that $\Delta$
is not Lang--Trotter since $e(C,X_\Delta) = 6$.

We remark that even if in all the examples of 
Lang--Trotter polygons appearing in 
Database~\ref{asfgSrhasrh} the condition 
$m = {\rm width}(\Delta)$ is satisfied, it is
possible to find examples in which $m$ is 
smaller. For instance, one can check by a computation similar to Computation~\ref{asdcvq}  that the
polygon with vertices 
$
[ 0, 0 ],
[ 12, 4 ],
[ 11, 7 ],
[ 9, 12 ],
[ 8, 12 ]
$
is Lang--Trotter and it has $m = 11$ and 
${\rm width}(\Delta) = 12$.
\end{example}

\begin{example}\label{ex:exp}{\bf Polygon~111} is the 
polygon $\Delta$ with vertices:
\[
\left[
\begin{matrix}
6 & 5 & 1 & 8 & 0 & 0 & 3\\
1 & 4 & 3 & 2 & 6 & 7 & 0
\end{matrix}
\right]
\]
which appears in Table~\ref{tab:good}
for $m = 7$ (where it corresponds to the blue matrix) and
we will use it later in the proof of Theorem~\ref{asgarh}.
We claim that $\Delta$  is  Lang--Trotter.

First of all, $\vol(\Delta) = 49$ and $|\partial \Delta\cap\bZ| = 7$
(see Computation~\ref{sdfvwefvwefv}).
By Computation~\ref{asarsgwRG}, $\cL_\Delta(7)$ 
has dimension $0$ and the unique curve 
$\Gamma\in\cL_\Delta(7)$  has equation
\begin{gather*}
   -\textcolor{red}{u^8v^2} + 4u^7v^2 + 8u^6v^3 - 5u^6v^2 - \textcolor{red}{3u^6v} - \textcolor{red}{5u^5v^4} - 50u^5v^3 + 21u^5v^2 +\\ 
    6u^5v + 40u^4v^4 + 85u^4v^3 - 55u^4v^2 - 6u^3v^5 - 85u^3v^4 - 40u^3v^3 + 56u^3v^2 -\\ 
    10u^3v + \textcolor{red}{u^3} + 15u^2v^5 +  80u^2v^4 - 40u^2v^3 + u^2v^2 + 3uv^6 - 30uv^5 +\\ 
    5uv^4 + \textcolor{red}{2uv^3} -    \textcolor{red}{v^7} + \textcolor{red}{4v^6} = 0.
\end{gather*}
The exponents of the red monomials
are the vertices of $\Delta$, so that 
the Newton polygon of $\Gamma$ is
$\Delta$. By Computation~\ref{adfafgarg} 
the curve $\Gamma$ is irreducible and
its strict transform $C\subseteq X_\Delta$ 
is a smooth elliptic curve.
It has the minimal equation
$$y^2 + xy = x^3 - x^2 - 4x + 4$$
by Computation~\ref{asdcvq}.
This is the curve labelled 
\href{https://www.lmfdb.org/EllipticCurve/Q/446/a/1}{446.a1}
in the LMFDB database~\cite{lmfdb}.
Since $e(C,X)>1$,  $\rd(C)\in\Pic^0(C)$ 
is not trivial. Since the Mordell--Weil group is $\bZ^2$, 
$\rd(C)$ is not torsion and therefore $\Delta$ is Lang--Trotter.
\end{example}

\begin{remark}
 \label{ex:g2}
 If in Definition~\ref{def:good} we substitute condition 
 (ii) with $|\partial\Delta\cap \bZ^2| < m$, 
 the curve $C$ will have arithmetic genus 
 $p_a(C) > 1$, so that $(C,X_\Delta)$ is no
 longer an elliptic pair. However, if 
 $\rd(C)$ is not torsion, we can still conclude that $\oEff(X_\Delta)$
 is not polyhedral by Proposition~\ref{prop:nonpol}.
 In the database~\cite{bel}, there are only two polygons 
satisfying $|\partial\Delta\cap \bZ^2| < m$ together with (i), (iii) and (iv).
Both polygons have volume $49$ and $5$
boundary points, so that by Proposition~\ref{prop:gen}
the corresponding curve $C$ has genus $2$. 
In the first case we verified 
that $2C$ moves (Computation~\ref{asarsgwRG}), so $\rd(C)$ 
is torsion. 
The second polygon has the following vertices
\[
\left[
\begin{matrix}
 0 & 5 & 7 & 3 & 1 \\
 0 & 2 & 3 & 8 & 3  
\end{matrix}
\right]
\]
and we claim that in this case
$\rd(C)$ is not torsion.
Indeed 
the curve $C$ is isomorphic to 
a hyperelliptic curve with equation
\[
 y^2 + (x^2 + x + 1)y = x^5 - 3x^4 + x^3 - x.
\]
This is the curve labelled 
\href{https://www.lmfdb.org/Genus2Curve/Q/1415/a/1415/1}{1415.a.1415.1}
in the LMFDB database~\cite{lmfdb} and the Mordell-Weil
group of the corresponding jacobian surface 
is  isomorphic to $\mathbb Z
\oplus\mathbb Z/2\mathbb Z$. By Computation~\ref {asarsgwRG},
$\dim |2C| = 0$
and we conclude that $\rd(C)$
is non-torsion.
\end{remark}


\section{Arithmetic toric elliptic pairs of infinite order}\label{asgasrhasrh}

\begin{notation}\label{arithmetic set-up}
Given a lattice polygon $\Delta\subseteq\bZ^2$,
let $\cP$ be the projective  toric scheme over $\Spec\bZ$ given by the normal fan of $\Delta$, 
with a relatively ample invertible sheaf $\cL$ given by the polygon $\Delta$.
Let $\cX$ be the blow-up of $\cP$ along the identity section of the torus group scheme.
Let $\cE\simeq\bP^1_\bZ$ be the exceptional divisor. 
For any  field $k$, we denote by $\bP_k,L_k,X_k,E_k$
the corresponding base change (or simply 
by $\bP,L,X,E$ if $k$ is clear from the context).
We will assume that $\Delta$ is a Lang--Trotter polygon,
i.e., 
$(C_\bC,X_\bC)$ is an elliptic pair of order $e(C_\bC,X_\bC)=\infty$.
Then
$(C_\bC,X_\bC)$ is a geometric fiber of an arithmetic elliptic pair $(\cC,\cX)$ of infinite order
flat over an open subset $U\subset\Spec \bZ$, see Definition~\ref{eSGsgSG}.
We  assume that  $C_\bC$ is {\em a smooth elliptic curve}.
A geometric fiber $(C,X)$ of $(\cC,\cX)$ over a prime $p\in U$
is an elliptic pair of finite order $e_p$. There is a 
morphism of schemes $\cX\to\cY$ flat over $U$, inducing 
a morphism $X\to Y$ to the minimal model for any geometric fiber.
Geometric fibers $(C,Y)$ of $(\cC,\cY)$ over $U$ are minimal elliptic pairs with Du Val singularities and the same root lattice $T$, 
which we call {\em the root lattice of~$\Delta$}.
Recall that we call $p$ a polyhedral prime of $\Delta$ if 
$\oEff(Y)$ is a polyhedral cone in characteristic~$p$.
We are interested in the distribution of polyhedral and non-polyhedral primes.
Recall that polyhedrality is governed by Corollary~\ref{concrete}:
$p$ is polyhedral
if and only if there are roots  $\beta_1,\ldots,\be_{8-R} \in\bE_8\setminus\hat T$, linearly independent modulo $\hat T$ and such that
$\ored(\beta_i)=0$ in $C(\bF_p)/\rd(C)$. Here $R$ is the rank of $T$.
\end{notation}

We will need a lemma on arithmetic geometry of elliptic curves.

\begin{lemma}\label{qfvqefvefv}
Let $C$ be an elliptic curve defined over~$\bQ$ without complex multiplication over~$\bar\bQ$.
Fix points $x_0,\ldots,x_r\in C(\bQ)$
of infinite order. 
Suppose the subgroup  $\langle 
x_1,\ldots,x_r
\rangle\subset C(\bQ)$ generated by $x_1,\ldots,x_r$
is free abelian and does not contain a multiple of $x_0$.
Then the reductions $\bar x_1,\ldots,\bar x_r$ modulo  $p$
are not contained in the cyclic subgroup  generated by the reduction $\bar x_0$ 
for a set of primes 
of  positive density.
\end{lemma}

\begin{remark}
Note that  $x_1,\ldots,x_r\in  C(\bQ)$ are not assumed linearly independent.
\end{remark}

\begin{proof}
For a fixed integer~$q$, let $C[q]\subset C(\bar\bQ)$ be the set of $q$-torsion points, so that $C[q]\simeq(\bZ/q\bZ)^2$ as a group.
Let~$K$  be the field
$\bQ(C[q])$. 
Since $C$ does not have complex multiplication,  
\begin{equation}
\Gal(K/\bQ)\simeq\GL_2(\bZ/q\bZ)
\end{equation} 
for almost all primes $q$ by Serre's theorem~\cite{Se}.
Choose a basis $y_1,\ldots,y_s$ of $\langle 
x_1,\ldots,x_r\rangle$. Since $x_0$ has infinite order,
$y_0=x_0$, $y_1,\ldots,y_s$ is a basis of the free abelian group 
$\langle x_0,\ldots,x_r\rangle$.
Choose points $y_0/q,\ldots,y_s/q\in C(\bar\bQ)$.
Let  $K_{y_0,\ldots,y_s}$ be a field extension of $K$
generated by $y_0/q,\ldots,y_s/q$ (any choice of quotients gives the same field).
By~Bashmakov's theorem \cite{Bash}, for almost all primes $q$ we have 
$$\Gal(K_{y_0,\ldots,y_s}/\bQ)\simeq \GL_2(\bZ/q\bZ)
\ltimes((\bZ/q\bZ)^2)^{s+1}.$$

For any $x\in C(\bQ)$, let $i(\bar x)$ denote the index of the subgroup $\langle\bar x\rangle\subset C(\bF_p)$.
It~suffices to prove that $i(\bar x_1),\ldots,i(\bar x_r)$ 
are not divisible by~$q$ but $i(\bar x_0)$ is divisible by~$q$ for a set of primes $p$ of positive density. 
By~\cite{LT},
$i(\bar x)$ is divisible by~$q$ if and only if the Frobenius element 
$$\sigma_p=(\gamma_p,\tau_p)\in \Gal(K_x/\bQ)\simeq\GL_2(\bZ/q\bZ)\ltimes(\bZ/q\bZ)^2$$
belongs to one of the following conjugacy classes: either $\gamma_p= 1$ 
or $\gamma_p$ has eigenvalue~$1$ and $\tau_p\in\Imm(\gamma_p-1)$.
We can express $x_i=\sum\limits_{j=1}^s a_{ij}y_j$ for $i=1,\ldots,r$, $a_{ij}\in \bZ$.
To~apply the Chebotarev  density theorem \cite{Chebotarev}, 
it remains to note that the 
subset of tuples
$(\gamma,\tau_0,\ldots\tau_s)\in\GL_2(\bZ/q\bZ)\ltimes((\bZ/q\bZ)^2)^{s+1}$ such that 
$\gamma$ has eigenvalue~$1$, $\tau_0\in\Imm(\gamma-1)$ and 
$\sum\limits_{j=1}^s a_{ij}\tau_j\not\in\Imm(\gamma-1)$ for $i=1,\ldots,r$, is non-empty for $q\gg0$.
\end{proof}

\begin{remark}
We were inspired by the following theorem of Tom Weston \cite{W}.
Suppose we are
given an abelian variety  $A$ over a number field $F$ such that $\End_FA$ is commutative,
an~element $x\in A(F)$ and a subgroup $\Sigma\subset A(F)$.
If  ${\redd}_vx\in{\redd}_v\Sigma$
for almost all places $v$ of $F$ then $x\in\Sigma+A(F)_{\tors}$.
\end{remark}

Here is another variation on the same theme:

\begin{lemma}\label{jhgfjhgf}
Let $C$ be an elliptic curve defined over~$\bQ$ with points 
 $x, y\in C(\bQ)$ of infinite order such that $y=dx$ for 
 a square-free integer $d$. 
Suppose there exists a prime $p$ of good reduction
and coprime to $d$ such that the index of 
$\langle\bar x\rangle$  is coprime to $d$
but the index of $\langle\bar y\rangle$ is divisible by~$d$.
Then $\bar x,2\bar x,\ldots,(d-1)\bar x\not\in\langle\bar y\rangle$
for a set of primes of positive density.
\end{lemma}

\begin{proof}
We need to prove positive density
of primes such that the index of the subgroup $\langle\bar y\rangle$ in $\langle\bar x\rangle$ is equal to~$d$.
It is enough to prove positive density for the set of primes such that the index of 
$\langle\bar x\rangle$ in $C(\bF_p)$ is coprime to $d$
but the index of $\langle\bar y\rangle$ is divisible by~$d$.
Arguing as in the proof of Lemma~\ref{qfvqefvefv},
we can express this condition as a condition that the Frobenius element $\sigma_p$ is contained
in the union of certain conjugacy classes in the Galois group $\Gal L/\bQ$, where
$L$ is obtained by adjoining the $d$-torsion $C[d]$ 
and the point $x/d$.
To apply Chebotarev density theorem, we need to know that this conjugacy class is non-empty.
Arguing in reverse, it suffices to find a specific $p$ such that the condition holds.
\end{proof}

\begin{theorem}\label{sgasrg}
Consider Lang--Trotter polygons from Table~\ref{sdfsg} (numbered as in Table~\ref{tab:good}).
We~list the root lattice $T$, the minimal equation of the elliptic curve~$C$, its Mordell-Weil group $C(\bQ)$
and $\rd(C)$.
\begin{table}[htbp]
\begin{tabular}{|c|c|c|c|c|}
\hline
$N$  & $T$ & $C$ & MW & $\rd(C)$ \cr
\hline
$19$ & $A_7$ & $y^2 + y = x^3 - x^2 - 24x + 54$ & $\bZ^2$&
$-(1,5)$\cr
\hline
$24$  & $A_6\oplus A_1$ & $y^2 + y = x^3 + x^2$ & $\bZ$ &
$6\left(0, 0\right) $
\cr
\hline
$111$  & $A_6\oplus A_1$ & $y^2+xy = x^3 - x^2 -4x+4 $& $\bZ^2$ 
&$(-1, -2)$
\cr
\hline
$128$  & $A_3\oplus A_3$ & $y ^2+y=x^3+x^2-240x+1190$ & $\bZ^3$ 
 &$(15,34)$
\cr
\hline
\end{tabular}
\caption{}\label{sdfsg}
\end{table}
The set
of non-polyhedral primes is infinite of  positive density
and includes primes under $2000$ from Table~\ref{sdfghjghjsg}.
\begin{table}[htbp]
\begin{tabular}{|c|c|}
\hline
$N$  & {\rm primes} \cr
\hline
&\\
$19$ & 
$\begin{smallmatrix}
11, 41, 67, 173, 307, 317, 347, 467, 503, 523, 571, 593, 631, 677, 733, 797, 
809, 811, 827, 907,\\ 
937, 1019, 1021, 1087, 1097, 1109, 1213, 1231,
1237, 1259, 1409, 1433, 1439, 1471, 1483,\\ 1493, 1567, 1601, 1619, 1669, 1709, 1801, 1811, 
1823, 1867, 1877, 1933, 1951, 1993
\end{smallmatrix}
$\\
&\\
\hline
&\\
$24$ & 
$\begin{smallmatrix}
29, 59, 73, 137, 157, 163, 223, 257, 389, 421, 449, 461, 607, 641, 647, 673, 691, 743, 797, 929, 937,\\ 
983, 991, 1049, 1087, 1097, 1103, 1151, 1171, 1217, 1223, 1259, 1279, 1319, 1367, 1399, 1427,\\ 
1487, 1549, 1567, 1609, 1667, 1697, 1747, 1861, 1867, 1871, 1913
\end{smallmatrix}
$
\\
&\\
\hline
&\\
$111$ & 
$\begin{smallmatrix}
47,71,103,197,233,239,277,313,367,379,409,503,563,599,647,677,683,691,719,\\
727,761,829,911,997,1103,1123,1151,1171,1187,1231,1283, 1327, 1481, 1493,\\ 
1709, 1723, 1861, 1907, 1997
\end{smallmatrix}$
\\
&\\
\hline

&\\
$128$ & 
$\begin{smallmatrix}
13, 17, 23, 71, 101, 103, 109, 191, 233, 277, 281, 283, 311, 349, 379, 397, 419, 433, 439, 443, 449, 457,\\ 
479, 509, 547, 557, 571, 631, 647, 653, 691, 701, 727, 743, 811, 829, 877, 929, 953, 1021, 1031, 1033,\\ 
1097, 1123, 1129, 1151, 1187, 1213, 1237, 1277, 1297, 1423, 1459, 1471, 1483, 1499, 1531, 1549,\\ 
1559, 1583, 1621, 1637, 1699, 1753, 1783, 1879, 1889, 1907, 1979
\end{smallmatrix}$
\\
&\\
\hline
\end{tabular}
\caption{}\label{sdfghjghjsg}
\end{table}
\end{theorem}

\begin{proof}
We first explain an outline of the argument and then proceed case-by-case.
We compute the normal fan of $\Delta$ and the fan of the minimal resolution $\tilde\bP_\Delta$ of $\bP_\Delta$
using Computation~\ref{sdfvwefvwefv}. We use Computation~\ref{sfsvwefv} to compute the Zariski decomposition
of $K_X+C$, which by Theorem~\ref{asfhadrhad} gives curves 
$C_1,\ldots,C_s$ contracted by the morphism to the minimal model $Y$,
and the classes of proper transforms of these curves in  $\tilde\bP_\Delta$.
Whenever $\Delta$ has lattice width $m$ in horizontal and vertical directions,
these curves  include $1$-parameter subgroups $C_1=(v=1)$ and $C_2=(u=1)$.
We~use Computation~\ref{efvwefvwef} to compute the root lattice $T$, $\Cl_0(Y)$,
and the push-forward map to $\Cl_0(Y)$.
Computation~\ref{asarsgwRG} gives the equation of the unique member $\Gamma$
of the linear system $\cL_\Delta(m)$ and its Newton polygon
and Computation \ref{adfafgarg} shows that the proper transform $C$ of this curve
in $X$ is an elliptic curve.
We use Computation~\ref{asdcvq} to compute the minimal equation of $C$,
 intersection points of $C$ with the toric boundary divisors,
$\rd(C)$ and the images of roots in $\bE_8$.
Reading off the Mordell-Weil group of $C$ from the LMFDB database~\cite{lmfdb}, we can deduce that
$\Delta$ is Lang--Trotter. 
In the same Computation \ref{asdcvq}, we apply 
Corollary~\ref{concrete} to test polyhedrality of specific primes from Table~\ref{sdfghjghjsg}.
Finally, we apply Lemma~\ref{qfvqefvefv} or Lemma~\ref{jhgfjhgf} 
to prove positive density of non-polyhedral primes.
\end{proof}

\begin{example}
{\bf Polygon 19} has vertices 
\begin{equation}\label{awgarg}
\left[
\begin{array}{cccccc}
 4 & 3 & 1 & 0 & 6 & 5 \\
 6 & 5 & 2 & 0 & 1 & 4 \\
\end{array}
\right]
\end{equation}
The minimal resolution $\tilde\bP_\Delta$ has the  fan from  Figure~\ref{hgJGHVfhg},
where bold arrows indicate  the fan of $\bP_\Delta$.
Note that $\tilde\bP_\Delta$ has a toric map to $\bP^1\times\bP^1$ and proper transforms of 
$1$-parameter subgroups $C_1$, $C_2$ are preimages of rulings. Thus they have self-intersection $-1$
after blowing up $e$.
The minimal resolution of $X$
contains the configuration of curves from the right of 
Figure~\ref{hgJGHVfhg} (toric boundary divisors and curves $C_1, C_2$).
\begin{figure}[htbp]
\includegraphics[width=2in]{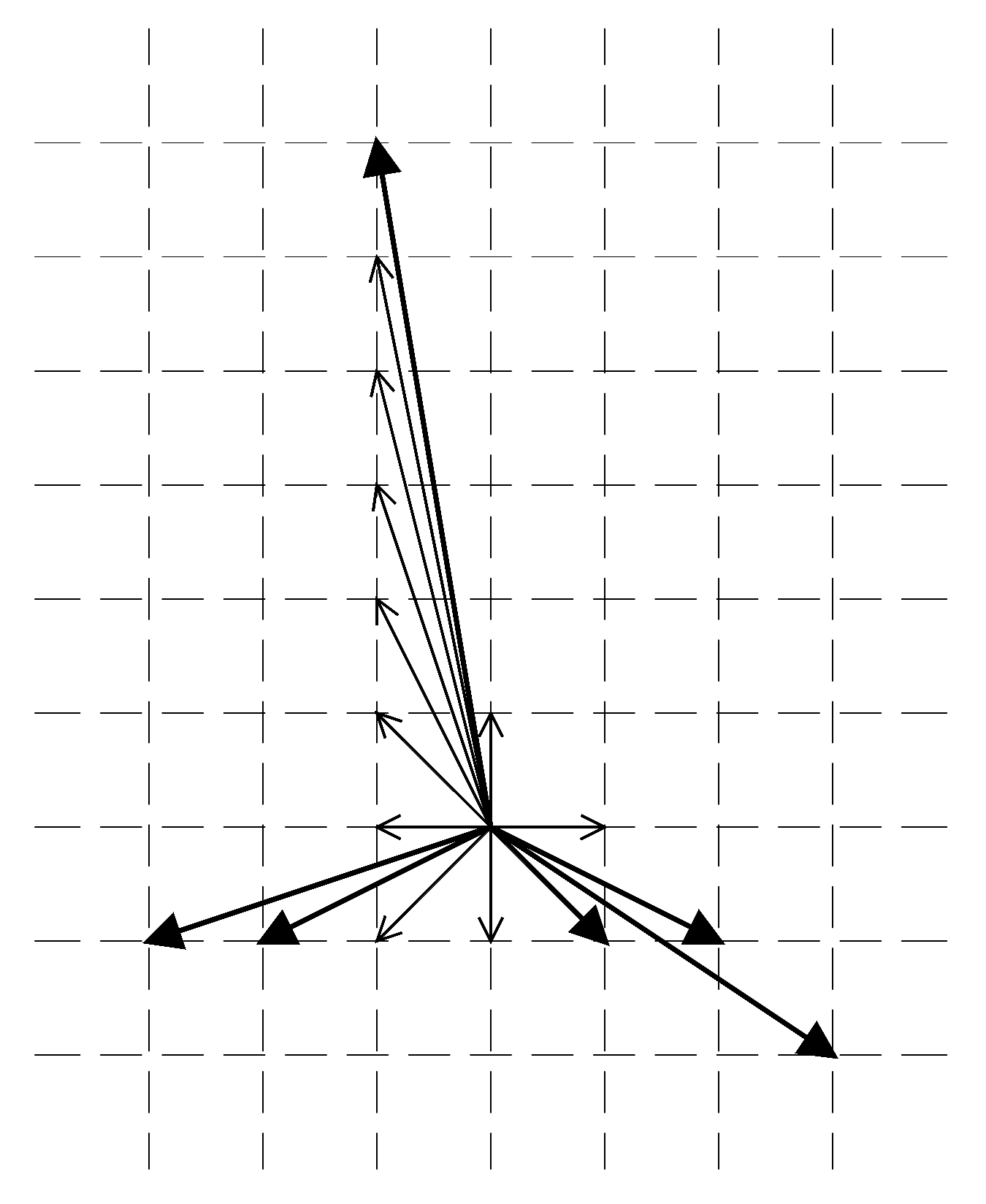}\raisebox{30pt}{\includegraphics[width=2.8in]{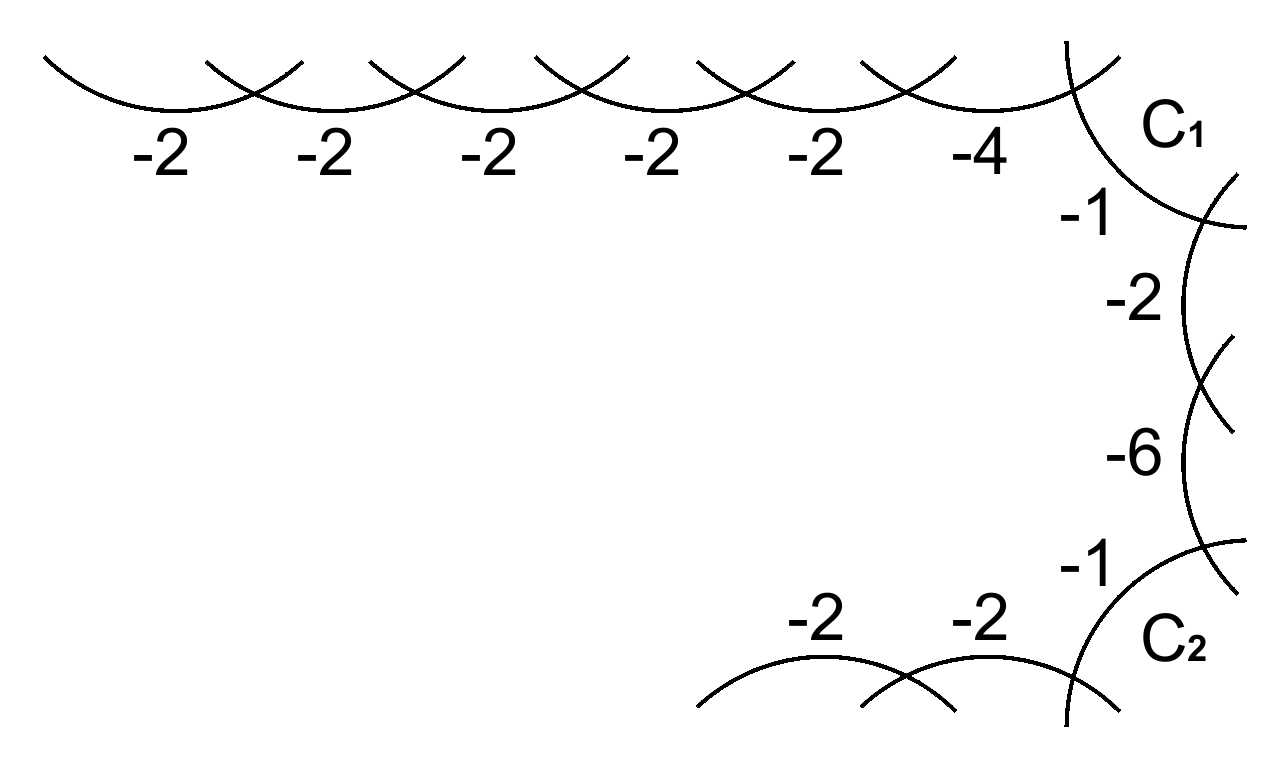}}
\caption{Polygon 19}\label{hgJGHVfhg}
\end{figure}
Only curves $C_1$ and $C_2$ contribute to the Zariski decomposition of $K+C$
and are contracted by the morphism $X\to Y$. Equivalently, 
the surface $Y$ is obtained by contracting the chain of rational curves above.
After blowing down $(-1)$-curves, this is equivalent to contracting a chain of seven $(-2)$-curves.
Thus $Y$ has an $A_7$ singularity and Picard number $3$.
There are two conjugate classes of root sublattices of type $A_7$ in $\bE_8$ (\cite{OS}[p. 85]).
In our case $\Cl_0(Y)\simeq\bZ$  is torsion-free, thus the embedding is primitive.
More precisely, we have $\Cl_0(Y)=\bE_8/\bA_7$,
which corresponds to the $\bZ$-grading of the Lie algebra
${\mathfrak e}_8=\bigoplus\limits_{\bar\beta\in\Cl_0(Y)}({\mathfrak e}_8)_{\bar\beta}$
of the form
$$\begin{matrix}
\bC^8&\Lambda^2\bC^8&\Lambda^3\bC^8& \red{\mathfrak{gl}_8} &\Lambda^3\bC^8&\Lambda^2\bC^8&\bC^8\cr
 8 & 28 & 56 &                \red{64}   & 56                      & 28                      & 8       \cr
\end{matrix}$$
Let $\alpha$ be a generator of $\Cl_0(Y)$.
The images of the roots of $\bE_8$ are $\pm k\alpha$ for $k\le3$.
Thus polyhedrality condition is that $k\rd(\alpha)\not\in\langle\rd(C)\rangle$ in $\ch p$ for $k=1,2,3$.

Next we compute $\rd(\alpha)$ and $\rd(C)$.
The curve $\Gamma$ has equation
$$
f=\red{u^4v^6}+\red{6u^5v^4}-2u^4v^5-14u^5v^3-17u^4v^4-\red{4u^3v^5}+\red{u^6v}+
11u^5v^2
$$
$$+38u^4v^3+26u^3v^4-9u^5v-27u^4v^2-34u^3v^3+22u^4v+16u^3v^2-
10u^2v^3$$
$$ 
-24u^3v+10u^2v^2+15u^2v+\red{5uv^2}-11uv+\red{1}=0,
$$ 
and passes through $e$ with multiplicity $m=6$.
When ~$p\ne2,3,5$, $C$ has Newton polygon \eqref{awgarg}
and is isomorphic to an elliptic curve with the  minimal equation
$$y^2 + y = x^3 - x^2 - 24x + 54.$$
The curve $C$ is labelled 
\href{https://www.lmfdb.org/EllipticCurve/Q/997/a/1}{997.a1}
in the LMFDB database~\cite{lmfdb} and its Mordell-Weil group  $C(\bQ)\simeq\bZ^2$ is generated by
$Q=(1,5)$ and $P=(6,-10)$. 
We have $\rd(C)=-Q$, $\rd(\alpha)=P-Q$, in particular
 $\rd(C)$ is not torsion in characteristic~$0$ and thus
$\Delta$ is Lang--Trotter. Thus $\oEff(X)$ and $\oEff(Y)$ are not polyhedral in characteristic~$0$.

In characteristic $p$,
$k\bar P$ is not contained in the cyclic subgroup 
of $C(\bF_p)$ generated by $\bar Q$ for $k=1,2,3$ for all primes in Table~\ref{sdfghjghjsg}.
According to the LMFDB database~\cite{lmfdb}, $C$ has no complex multiplication.
To prove positive density of non-polyhedral primes, we apply Lemma~\ref{qfvqefvefv} to $x_0=Q$ and $x_k=kP$ for $k=1,2,3$.
\end{example}

\begin{remark}
Empirically, about $18\%$ of primes are not polyhedral for this polygon.
It would be interesting to obtain heuristics for density of non-polyhedral primes.
\end{remark}

\begin{remark}
Since $C$ contains an irrational $2$-torsion point, the Lang--Trotter conjecture \cite{LT}
predicts that $\bar Q$ generates $C(\bF_p)$ for a set of primes $p$
of positive density.
If true, the Lang--Trotter conjecture
implies that $\overline{\Eff}(Y)$ is polyhedral  in characteristic~$p$ for
a set of primes
of positive density. However, 
the Lang--Trotter conjecture is only known for curves with complex multiplication \cite{GM_compositio}. 
\end{remark}

\begin{example}
{\bf Polygon 24} has  vertices
\[
\left[
\begin{array}{cccccc}
 0 & 2 & 5 & 6 & 1 & 0 \\
 0 & 1 & 3 & 4 & 6 & 1 \\
\end{array}
\right]
\]
The minimal resolution $\tilde \bP_\Delta$ of $\bP_\Delta$ has the  fan from the left side of Figure~\ref{EGwegwE},
where bold arrows indicate  the fan of $\bP_\Delta$.
As for the Polygon~$19$, the proper transforms of 
$1$-parameter subgroups $C_1$, $C_2$ in $X$ have self-intersection $-1$
and are the only curves contracted by the map to $Y$,
which therefore can be obtained by contracting the configuration of rational curves from the right of 
Figure~\ref{EGwegwE}.
\begin{figure}[htbp]
\includegraphics[width=2in]{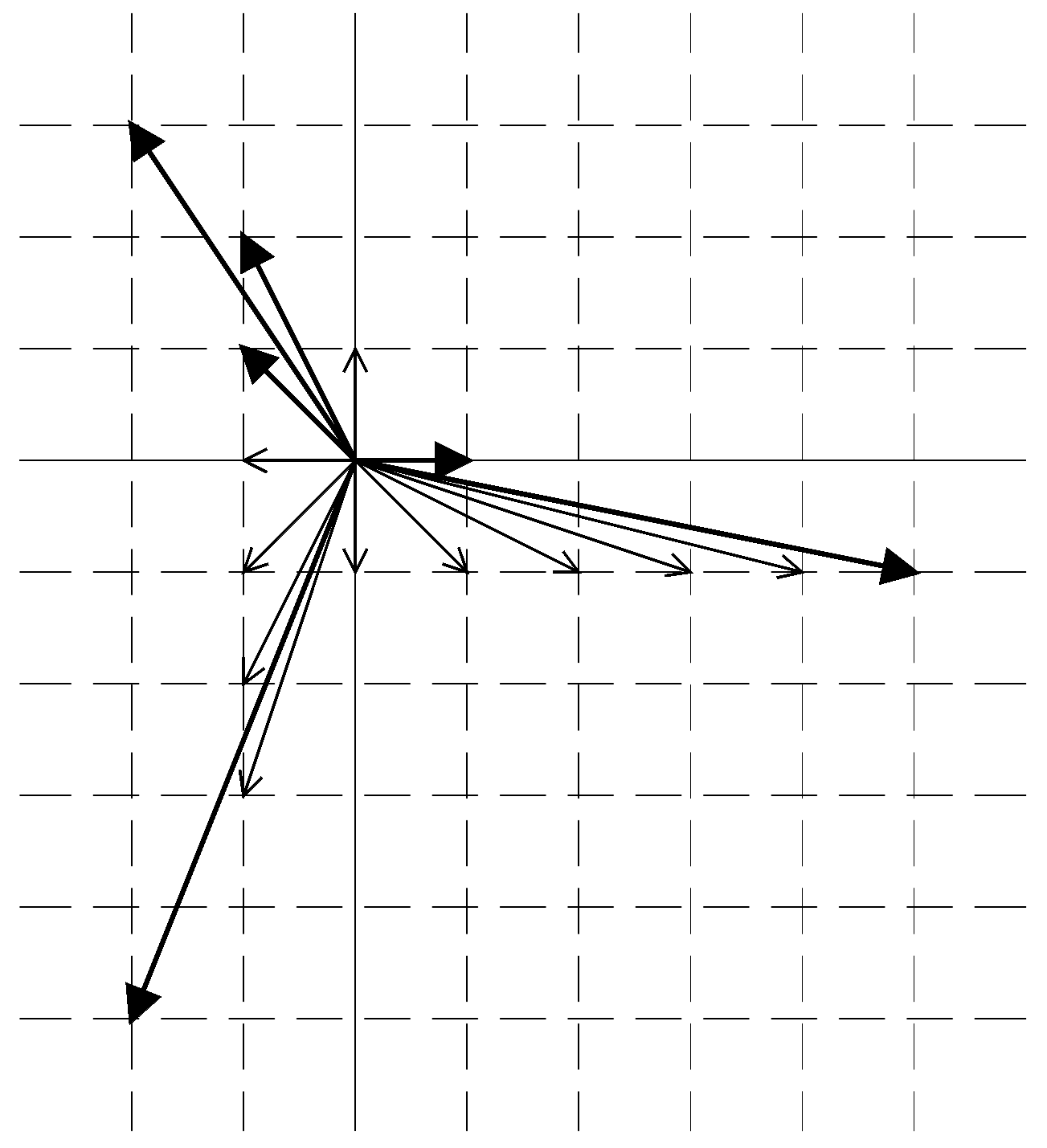}\includegraphics[width=2.8in]{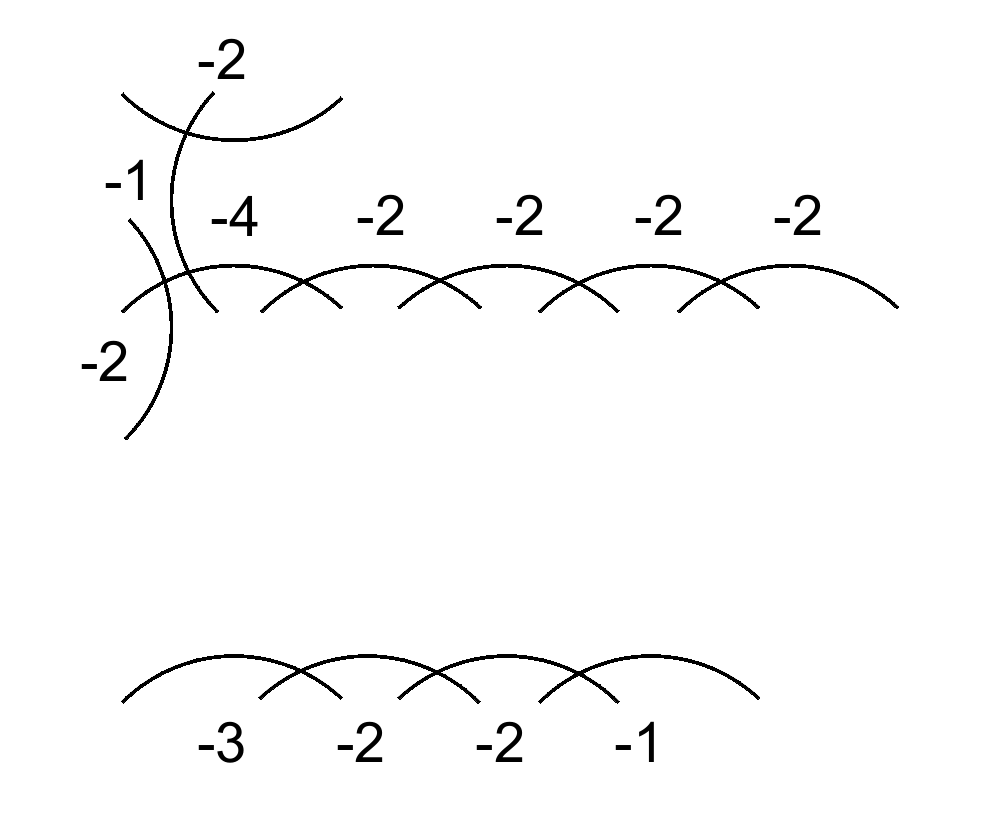}
\caption{Polygon 24}\label{EGwegwE}
\end{figure}
It follows that $Y$ has Picard number $3$ 
and  singularities $A_1$ and $A_6$.
The curve $\Gamma$ has a point of multiplicity $6$ at $e$ and equation
$$f=-\red{1}+\red{2v}+7uv-\red{3u^2v}-23uv^2+6u^2v^2+2u^3v^2+18uv^3+20u^2v^3$$
$$-26u^3v^3+10u^4v^3-\red{2u^5v^3}-12uv^4-11u^2v^4$$
$$+6u^3v^4+5u^4v^4-4u^5v^4+\red{u^6v^4}+5uv^5+3u^2v^5-2u^3v^5-\red{uv^6},$$
which has a required Newton polygon when $p\ne2,3$.
From the Dynkin classification it follows that $\Cl_0(Y)\simeq\bZ$.
Let $\alpha$ be a generator.
The images of roots of $\bE_8$ are equal to $\pm k\alpha$ for 
$0\le k\le 4$.
Thus the polyhedrality condition is that $k\rd(\alpha)\not\in\langle\rd(C)\rangle$ in $\ch p$ for $k=1,2,3,4$.
The minimal equation of the elliptic curve $C$ is 
$$y^2 + y = x^3 + x^2.$$
It is the curve
\href{https://www.lmfdb.org/EllipticCurve/Q/43/a/1}{43.a1}
from the LMFDB database~\cite{lmfdb} of elliptic curves.
Its Mordell-Weil group is $\bZ$ generated by $(0,0)$.
We have $\rd(C)=Q=6\left(0, 0\right)$ and $\rd(\alpha)=P=-\left(0, 0\right)$.
It follows that $\rd(C)$ is not torsion and thus $\Delta$ is Lang-Trotter and 
$\overline\Eff(X)$, $\overline\Eff(Y)$ are not polyhedral in characteristic~$0$. 
In characteristic $p$, $k\bar P$ is not contained in the cyclic subgroup 
of $C(\bF_p)$ generated by $\bar Q$ for $k=1,2,3,4$ for all prime numbers in the table.
Thus these primes are not polyhedral.
The positive density follows from Lemma~\ref{jhgfjhgf} with 
$p=223$, when the index of $\bar P$ is $1$ and the index of $\bar Q$ is $6$.
\end{example}

\begin{example}
{\bf Polygon 111} (discussed in Example \ref{ex:exp} and followed through in computations \ref{sdfvwefvwefv}--\ref{asdcvq}). 
The corresponding curve has the required Newton polygon in all characteristics $p\neq 2,3,5$. 
The minimal resolution $\tilde\bP_\Delta$ has the  
fan from  Figure~\ref{figure 111},
where bold arrows indicate  the fan of $\bP_\Delta$.
Note that $\tilde\bP_\Delta$ has a toric map to $\bP^1\times\bP^1$ 
and proper transforms of $1$-parameter subgroups $C_1$, $C_2$ are preimages of rulings\footnote{The $1$-parameter subgroups are in this case 
$\{u=1\}$ and $\{u=v\}$.}; hence, they have self-intersection $-1$
after blowing up $e$. The Zariski decomposition of $K+C$ is 
$2C_1+C_2+C_3$, where $C_3$ is a curve whose image in $\bP_\Delta$ has  multiplicity $3$ at $e$. The Newton polygon of $C_3$ has vertices
\[
\left[
\begin{matrix}
 3 & 0 & 0 & 1 \\
 1 & 3 & 2 & 0  
\end{matrix}
\right]
\]
and equation
$$u^3v - 3u^2v - uv^2 + 5uv - u + u^3 - 2u^2=0.$$
On $X$ the curve $C_3$ is disjoint from $C_1$ and $C_2$. The minimal resolution of $X$ contains the configuration of curves from the right of 
Figure~\ref{figure 111} (toric boundary divisors and curves $C_1, C_2, C_3$). The curves $C_1$, $C_2$, $C_3$ are contracted by the morphism $X\to Y$. Equivalently, the surface $Y$ is obtained by contracting the chain of rational curves above. It follows that the root lattice is 
$\bA_6\oplus \bA_1$ and the Picard number of $Y$ is $3$.
From the Dynkin classification, we have that 
$$\Cl_0(Y)=\bE_8/\bA_6\oplus\bA_1\simeq\bZ.$$ 
Let $\alpha$ be a generator.
The images of the roots of 
$\bE_8$ are equal to $\pm k\alpha$, for $0\leq k\leq 4$. 
Thus in characteristic $p$ the non-polyhedrality condition is that $k\rd(\alpha)\not\in\langle\rd(C)\rangle$, for $k=1,2,3,4$.
To prove that this holds for a set of primes of positive density, 
we apply Lemma \ref{qfvqefvefv} to $x_i=\rd(i\alpha)$, $x_0=\rd(C)$, 
for $i=1,2,3,4$.
Let us check that the conditions in the lemma are satisfied.
Using the minimal equation of the curve $C$ (see Example \ref{ex:exp})
and Computation \ref{asdcvq}, we find that $\rd(\alpha)=P=(0,2)$ and 
$\rd(C)=Q=(-1,-2)$. The curve $C$ (labeled 
\href{https://www.lmfdb.org/EllipticCurve/Q/446/a/1}{446.a1}
in the LMFDB database~\cite{lmfdb}) 
has no complex multiplication and has Mordell--Weil group $\bZ\times\bZ$
generated by $P$ and $-Q=(-1,3)$. Hence, the points $P$ and $Q$ have
infinite order and no multiple of $Q$ is contained in the subgroup generated by $P$.  

\begin{figure}[htbp]
\includegraphics[width=2in]{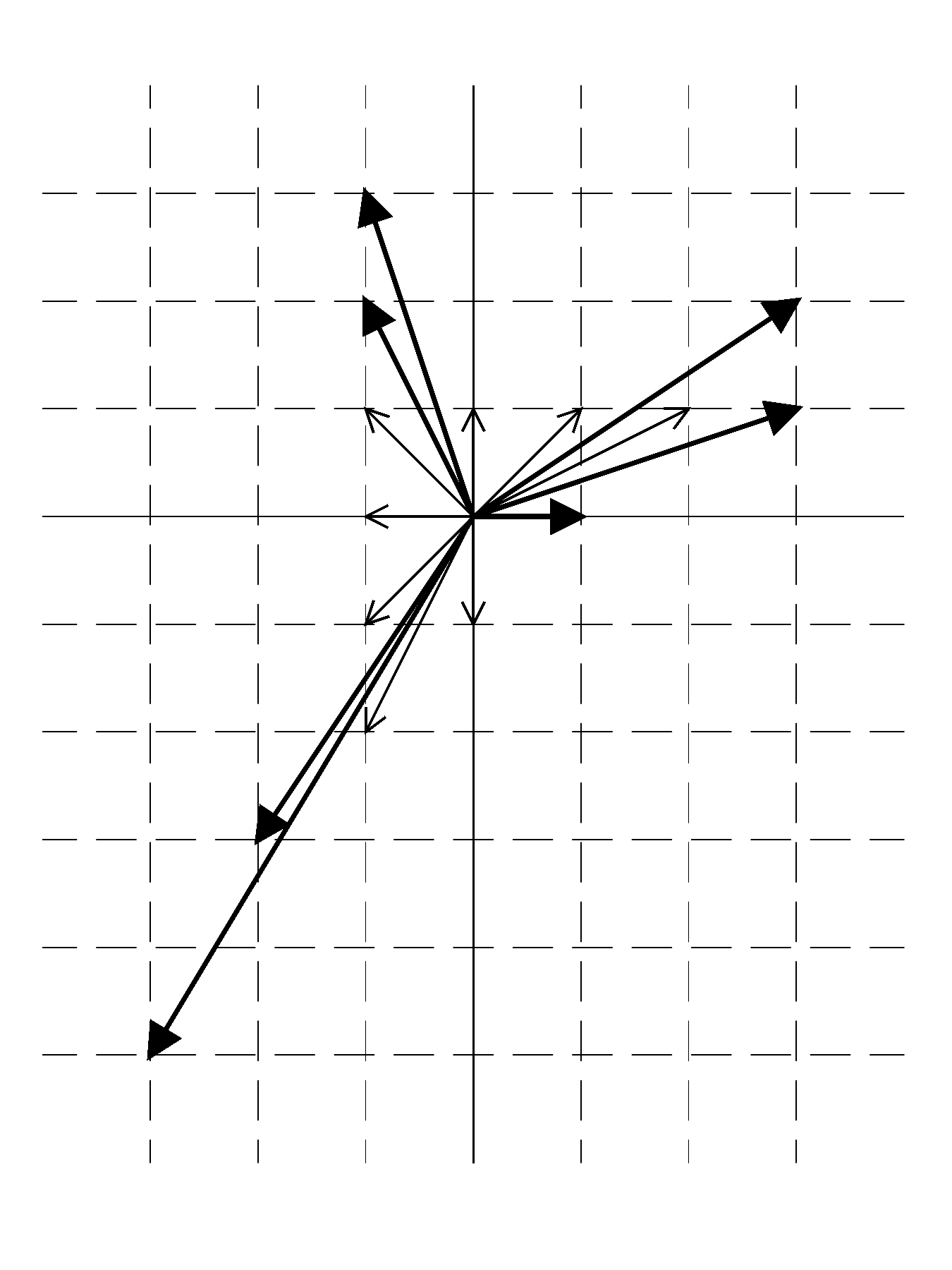}
\includegraphics[width=2.8in]{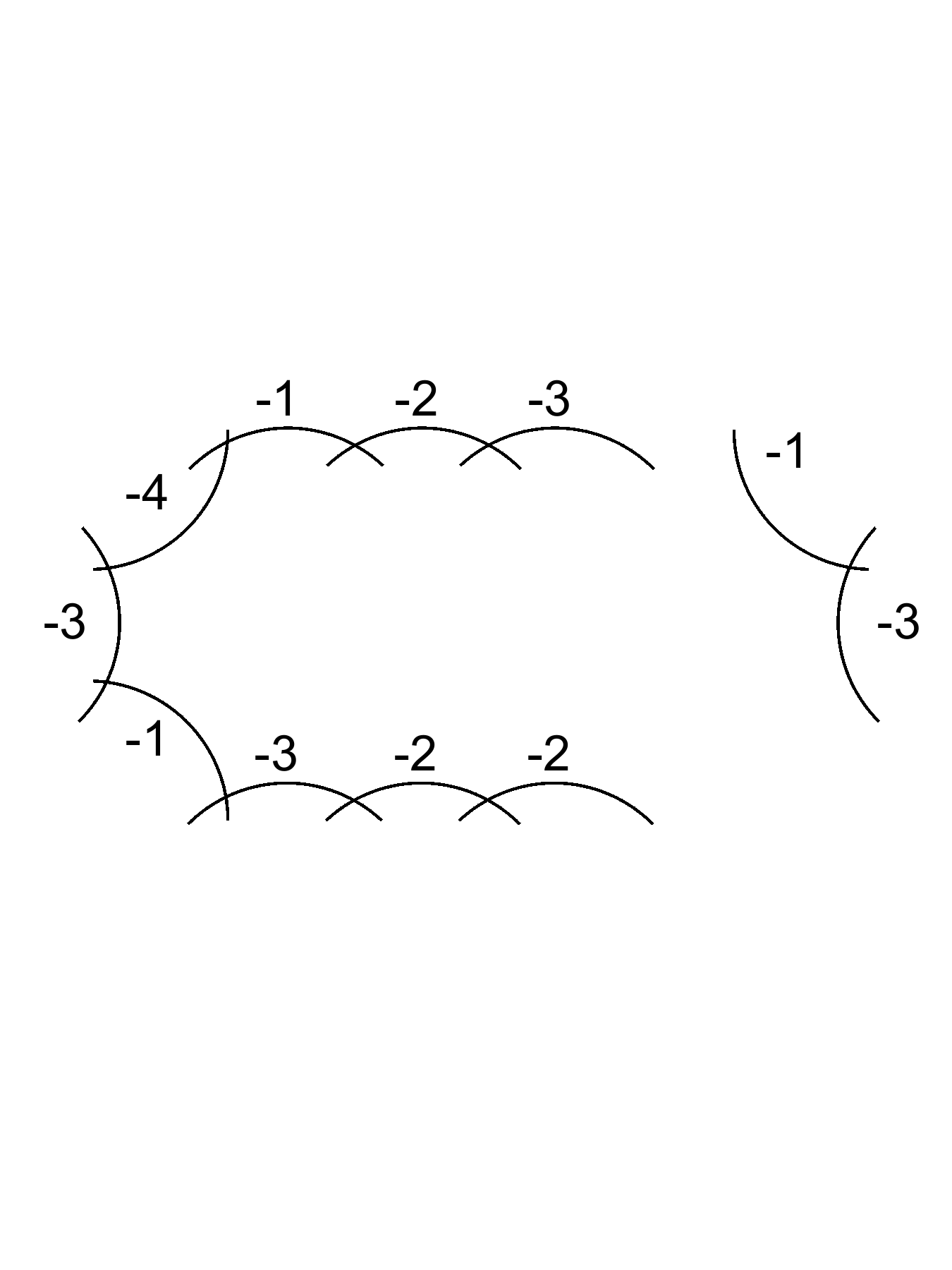}
\caption{Polygon 111}\label{figure 111}
\end{figure}
\end{example}

\begin{example}\label{asrgwgfeew4g}
{\bf Polygon 128.}
This is a polygon with vertices
\[
\left[
\begin{array}{ccccccc}
 0 & 1 & 6 & 7 & 6 & 3 & 1 \\
 5 & 6 & 7 & 7 & 5 & 0 & 3 \\
\end{array}
\right]
\]
The minimal resolution $\tilde \bP$ of $\bP$ has the  fan from the left side of Figure~\ref{Poly128},
where bold arrows indicate the fan of $\bP$.
\begin{figure}[htbp]
\includegraphics[width=2.6in]{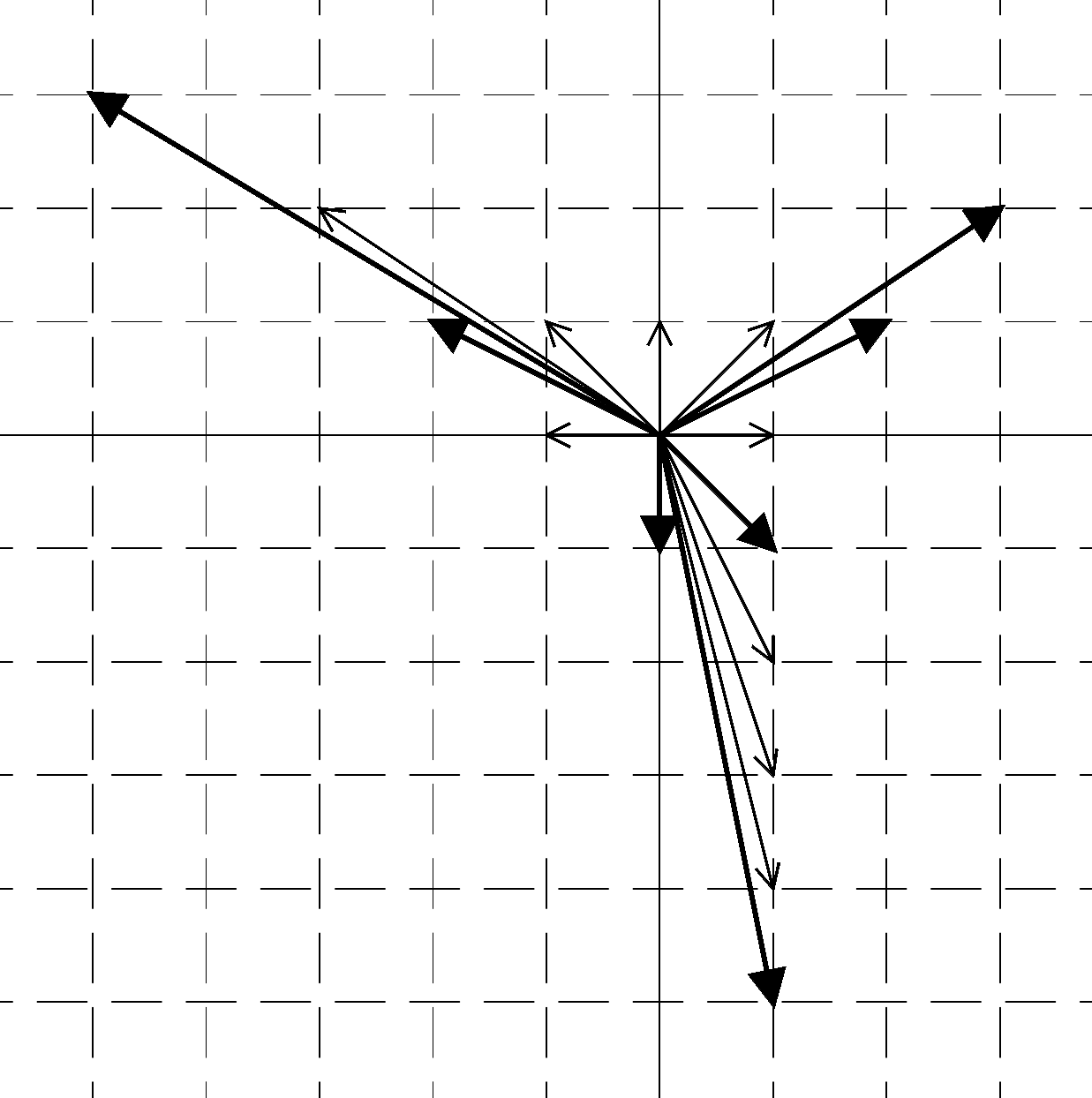}\includegraphics[width=2.2in]{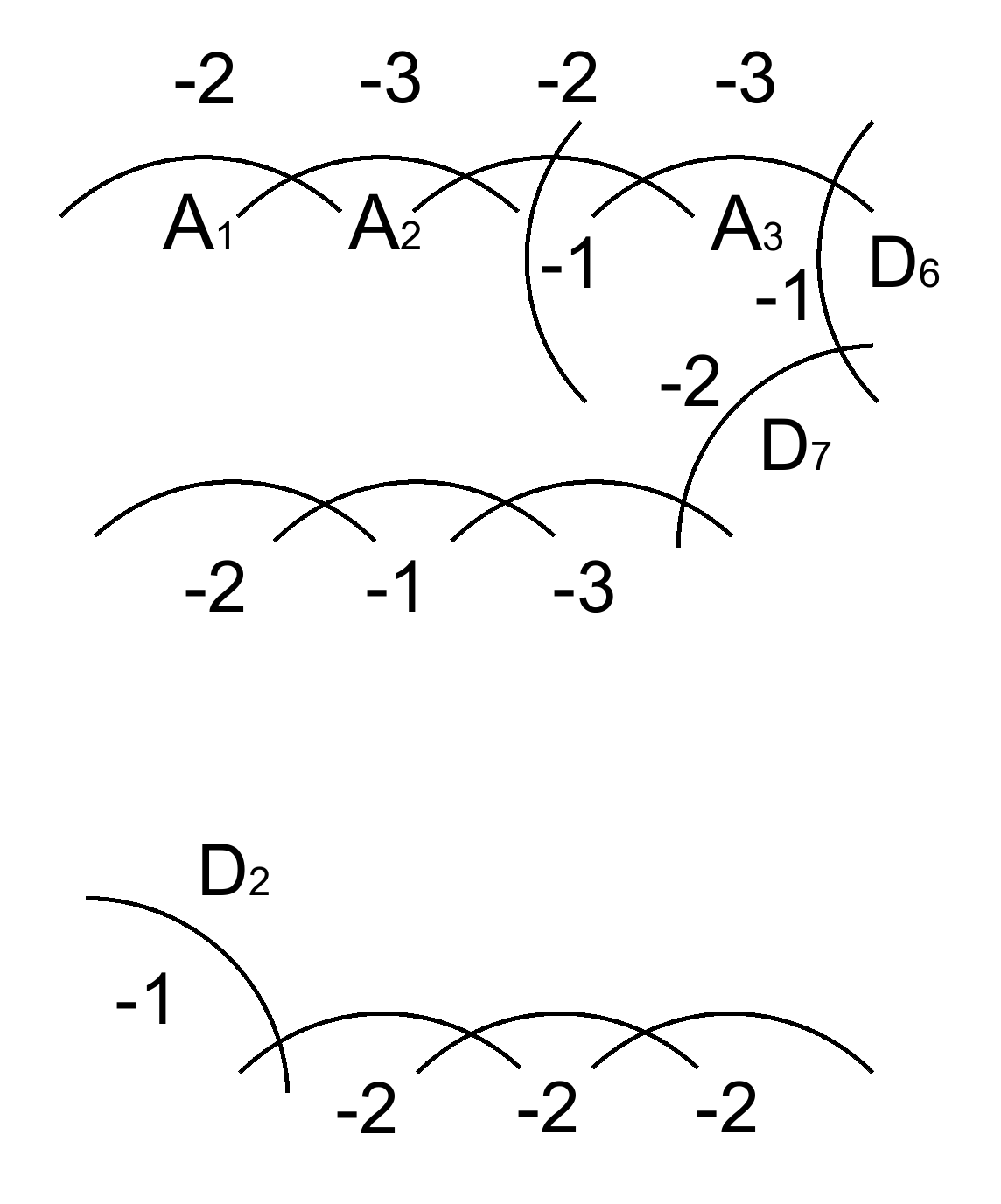}
\caption{Polygon $128$}\label{Poly128}
\end{figure}
The proper transforms of $1$-parameter subgroups $C_1$, $C_2$ are the only curves contracted by the map
$X\to Y$. Here $Y$ can be obtained from $\Bl_e\tilde \bP$ by contracting
a configuration of rational curves from the right of 
Figure~\ref{Poly128}, where we also indicate three boundary divisors, $D_2$, $D_6$ and $D_7$ 
(the only ones in Figure~\ref{Poly128} that do not get contracted by the map to $Y$). 
The root lattice is  $\bA_3\oplus \bA_3$, the Picard number of $Y$ is $4$.
One of the $\bA_3$'s is indicated with the chain $A_1$, $A_2$, $A_3$ of $(-2)$-curves (after contracting all $(-1)$-curves).
By the Dynkin classification, $\bE_8$ contains two lattices $\bA_3\oplus \bA_3$, one primitive and one non-primitive.
In our case  $\Cl_0(Y)\simeq\bZ^2$ is torsion-free, and therefore we have the primitive one.
Next we describe the images in $\Cl_0(Y)$ of roots in ~$\bE_8$. In other words, we have a grading
of the Lie algebra $\mathfrak e_8$ by the abelian group $\Cl_0(Y)$
$${\mathfrak e_8}=\sum_{\bar\beta\in \Cl_0(Y)}({\mathfrak e_8})_{\bar\beta}$$
and we need to describe the subset of non-empty weight spaces $\cB\subset \Cl_0(Y)$.
A~convenient interpretation of the lattice $\bE_8$ is the lattice $K^\perp\subset\Pic(\Bl_8\bP^2)$
with the standard basis $h,e_1,\ldots,e_8$.
The positive roots are $e_i-e_j$ for $i<j$, $h-e_i-e_j-e_k$, $2h-e_1-\ldots-\hat e_i-\ldots-\hat e_j-\ldots-e_8$
and $3h-e_1-\ldots-2e_i-\ldots -e_8$.
The~primitive sublattice $\bA_3\oplus \bA_3$ is generated by simple roots marked black on Figure~\ref{sFsrg}.
\begin{figure}[htbp]
\includegraphics[width=2.2in]{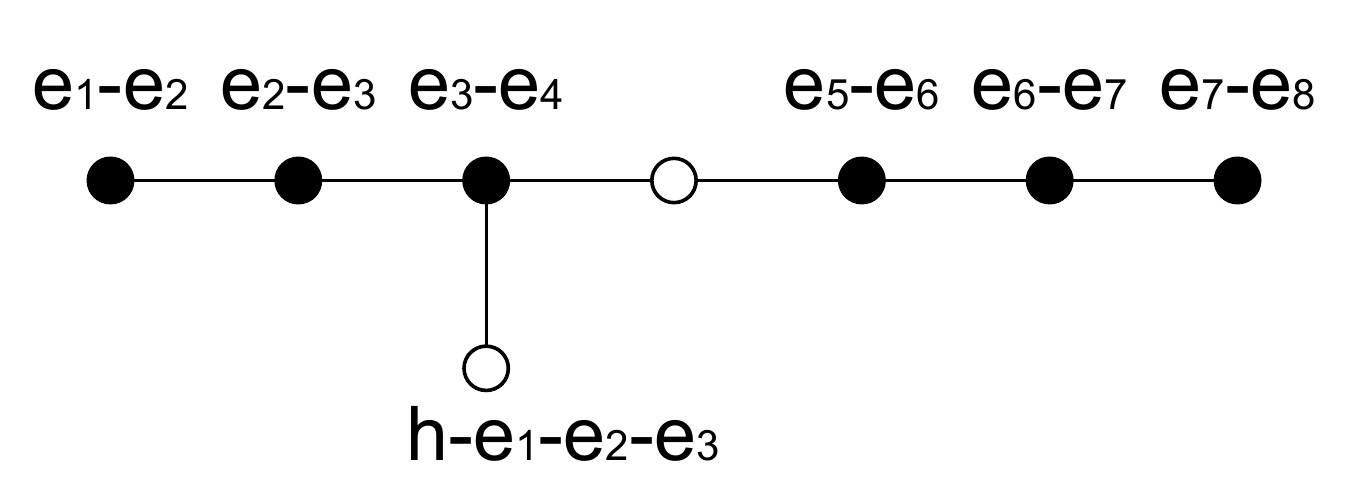}
\caption{$\bA_3\oplus \bA_3\subset\bE_8$}\label{sFsrg}
\end{figure}
It~follows that the $\bZ^2$ grading on $\bE_8$ is obtained by pairing with fundamental weights $h$ and
$e_5+e_6+e_7+e_8$ that correspond to white vertices of the Dynkin diagram.
The $\bZ^2$ grading of ${\mathfrak e_8}$ has the following non-empty weight spaces (in coordinates given by pairing with 
$h$ and $e_5+e_6+e_7+e_8$, respectively), where we also indicate dimensions.

\vspace{-23mm}
\[
\rotatebox{-45}{$
\begin{matrix}
&&&&&&4\\
&&&&&6&4\\
&&&&4&16&\\
&&&&24&6&\\
&&&16&24&&\\
&&4&{\mathbf 32}&4&&\\
&&24&16&&&\\
&6&24&&&&\\
&16&4&&&&\\
4&6&&&&&\\
4&&&&&&
\end{matrix}
$}
\]
\vspace{-23mm}

\noindent
It follows that the subset $\cB\subset\Cl_0(Y)$ is given by the $\pm$ columns of the matrix
\begin{equation}\label{vkbvv}
\begin{matrix}
1 &0 &1 &1 &1 &2 &2 &2 &3 &3\cr
0 &1 &1 &2 &3 &2 &3 &4 &4 &5\cr
\end{matrix}
\end{equation}
in the basis $u,v$, where $u$ (resp., ~$v$) is the image of the simple root $h-e_1-e_2-e_3$ (resp., ~$e_4-e_5$).
Next we compute vectors $u$ and $v$ in $\Cl_0(Y)$. By inspecting Figure~\ref{Poly128},
one can prove that, in the minimal resolution $Z$ of $Y$, 
$h-e_1-e_2-e_3$ corresponds to the $(-2)$-class $D_2-D_7$ and 
$e_4-e_5$ to $D_2-D_6-A_1-A_2-A_3$, which has pushforward $D_2-D_6$ on $X$.
Next we compute $\rd(C)$, $\rd(u)$ and $\rd(v)$.

The curve $\Gamma$ has a point of multiplicity $7$ at $e$ and its Newton polygon is 
$\Delta$ for $p\ne 2,3,7,11$.
The minimal equation of the elliptic curve $C$ is 
$$y ^2+y=x^3+x^2-240x+1190,$$
which is the curve
\href{https://www.lmfdb.org/EllipticCurve/Q/29157b1/}{29157b1}
from the LMFDB database of elliptic curves. It has Mordell--Weil group $\bZ^3$
generated by $P=(12,13)$, $R=(-6,49)$ and $Q=(-15,40)$.  
We have
$$\rd(C)=(15,-35)=P-Q-R,$$
$$\rd(u)=(120 ,1309)=Q-R,\quad \rd(v)=(-6,  49)=R.$$
We see that  $\rd(C)$ is not torsion in characteristic $0$ and thus
$\overline\Eff(Y)$ is not polyhedral.
In characteristic $p$, the condition of polyhedrality is that there exist 
two linearly independent column-vectors of the matrix \eqref{vkbvv} which, when dotted with the row vector
$(\rd(u), \rd(v))$ are contained in the cyclic subgroup of $C(\bF_p)$ generated by $\rd(C)$.
This gives the list of non-polyhedral primes in the table.
To prove the  positive density, we apply
Lemma~\ref{qfvqefvefv} (with $r=10$).
\end{example}

\begin{remark}
In Example~\ref{asrgwgfeew4g},
by Lemma~\ref{qfvqefvefv}, we get 
positive density not only for 
the set of non-polyhedral
primes but also for the set
of primes $p$ such that the Halphen pencil $|e_pC|$
on $Y$ has only irreducible fibers.
For example,
$\rd(C)$ has order $2$ in characteristic $23$
and none of the elements
of $\cB$, when restricted to $C$, are contained in the cyclic subgroup of $C(\bF_{23})$ generated by $\rd(C)$.
It follows that $|2C|$ on $Y$ is a Halphen pencil with only irreducible fibers.
This property is stronger than non-polyhedrality: 
in characteristic $13$, $\rd(C)$ has torsion $5$ and 
$\rd(u+v)$ is contained in the cyclic subgroup generated by $\rd(C)$
but no other linearly independent vector in $\cB$ is.
It follows that $\oEff(Y)$ is not polyhedral, 
but the Halphen pencil $|5C|$ on $Y$ contains a reducible fiber with two components
and no other reducible fibers.
\end{remark}

\section{Infinite sequences of Lang--Trotter polygons}\label{asrgasrharh}

\subsection*{An infinite sequence of pentagons.}\label{adrhdjd}

\begin{notation}
Let $k\geq 1$ be an integer and let $m=2k+4$. 
Let $\Delta$ be the pentagon with vertices
$(0,0)$, $(m-4,0)$, $(m,1)$, $(m-2,m)$, $(m-3,m-1)$.
\begin{figure}[htbp]
\begin{tikzpicture}[scale=.5]
\tkzInit[xmax=8,ymax=8]\tkzGrid
 \tkzDefPoint(5,7){P1}
 \tkzDefPoint(6,8){P2}
 \tkzDefPoint(8,1){P3}
 \tkzDefPoint(4,0){P4}
 \tkzDefPoint(0,0){P5}
 \tkzDrawSegments[color=black](P1,P2)
 \tkzDrawSegments[color=black](P2,P3)
 \tkzDrawSegments[color=black](P3,P4)
 \tkzDrawSegments[color=black](P4,P5)
 \tkzDrawSegments[color=black](P1,P5)
 
\node[below] at (4,0) {$G$};
\node[right] at (8,4) {$B$};
\node[left] at (0,4) {$F$};
\node[above] at (4,8) {$A$};

\node[above] at (2.3,0) {$D_3$};
\node[above] at (6,.7) {$D_2$};
\node[right] at (7,5) {$D_1$};
\node[below] at (5.7,7.6) {$D_5$};
\node[left] at (3,4.3) {$D_4$};
\end{tikzpicture}
\caption{Polygon $\Delta$ for $k=2$, $m=8$}\label{Sgsrhsrhsr}
\end{figure}
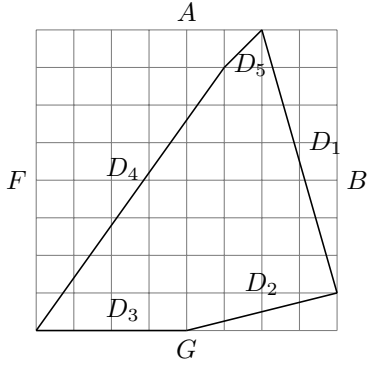
\end{notation}

\begin{theorem}\label{asgsgsrSRH}
The  polygon $\Delta$ is Lang--Trotter for every $k\ge1$.
Furthermore, every prime is a polyhedral prime of $\Delta$.
\end{theorem}

\begin{notation}\label{setup}
Consider an elliptic curve $C\subseteq \bP^2$ with the Weierstrass equation 
$$y^2=x(x^2+ax+b),\quad\text{ where}$$ 
$$a=-(12k^2+24k+11),\quad b=4(k+1)^2(3k+2)(3k+4).$$
Let 
$$x_0=2(k+1)(3k+2), \quad x_1=2(k+1)(3k+4).$$ 
Consider the following points on $C$ in homogeneous coordinates:
$$d_1=[0:1:0],\quad d_2=[x_0:-x_0:1],\quad  \tilde{d}_2=[x_0:x_0:1],$$
$$d_4=[0:0:1],\quad d_5=[x_1:x_1:1].$$
Define rational functions on $C$ as follows: 
$$f(x,y)=\frac{x^{k+1}(x-y)}{x-x_0},\quad g(x,y)=\frac{(x-x_0)(x^{k+1}-x^ky-2x_0^{k+1})}{x^k(x-y)}.$$
\end{notation}

\begin{lemma}\label{basics}
The curve $C$ is a smooth elliptic curve defined over $\bQ$. 
The points 
$d_1$, $d_2$, $\tilde{d}_2$, $d_4$, $d_5$ are mutually distinct and have the following properties: 
\begin{itemize}
\item[(i) ] 
The given lines intersect  $C$ at the following points, with multiplicities:
$$\begin{matrix}
z=0: &3d_1,&\qquad&
x=0: &d_1+2d_4,\cr
y=x: &d_4+d_5+\tilde{d}_2,&\qquad&
x=x_0: &d_1+d_2+\tilde{d}_2.\cr
\end{matrix}$$
In particular, we have equivalences of divisors on $C$ as follows: 
$$2d_1\sim 2d_4,\quad d_1+d_2\sim d_4+d_5.$$
\item[(ii) ] 
The divisors of zeros and poles of the rational functions $f$ and $g$ are 
$$(f)=\big((m-1)d_4+d_5\big)-\big((m-1)d_1+d_2\big),$$
$$(g)=\big(4d_2+\gamma)-\big(2d_1+(m-3)d_4+d_5\big),$$
where $\gamma$ is an effective divisor of degree $m-4$ disjoint from  $d_1,d_4,d_5$. 
\item[(iii) ] The line bundle $\cO(d_2-d_1)$ is not a torsion element of $\Pic^0(C)$. 
\end{itemize}
\end{lemma}

\begin{proof}
The discriminant equals $16a^2(b^2-4a)$ and it is non-zero for all integers $k$, hence, the curve $C$ is smooth. 
Part (i) is immediate noticing that $x=x^2+ax+b$ has solutions $x=x_0$ and $x=x_1$.  It follows from (i) that away from the point at infinity $d_1$, 
the rational function $f$ has zeros at $(m-1)d_4+d_5$ and a single pole at $d_2$, while at $d_1$, there is a pole of order $(m-1)$. Similarly, 
$g$ has poles at $d_1$, $d_4$, $d_5$ (of orders $2$, $(m-3)$ and $1$ respectively) and a zero of order at least $2$ at $d_2$. After a change of variables
$u=x-x_0$, $v=y+x_0$, we see that  $C$ has $v-(2k+1)u=0$ tangent line at $(0,0)$. 
After a further change of variables $w=v-(2k+1)u$, one can see that $x^{k+1}-x^ky-2x_0^{k+1}$ has multiplicity at least $3$ at $d_2$. Hence, 
$g$ has multiplicity at least $4$ at $d_2$. This proves (ii). 
To prove (iii), choose $d_1$ as the identity element of the Mordell-Weil group~$C(\bQ)$.  
By Mazur's theorem \cite{Mazur}, it suffices to prove that $nd_2\neq 0$ for $1\leq n\leq 12$. We check this in Computation~\ref{MZNdc,manBDC}.
\end{proof}

\begin{notation}\label{maps}
We label the sides of $\Delta$ as $D_1,D_2,D_3,D_4,D_5$ (see Figure~\ref{Sgsrhsrhsr}).
Note that $\Delta$ is inscribed in the square of side $m$ in the first quadrant, with one vertex at $(0,0)$ and sides labeled, starting from the $x$-axis and going counterclockwise, $G, B, A, F$. 
The normal fans of $\Delta$ and the square
give rise to toric surfaces $\bP_\Delta$ and $\bP^1\times\bP^1$. 
Let $S$ be the toric surface corresponding to the common 
refinement of the two fans and let 
$$\pi: S\ra \bP^1\times \bP^1,\quad \rho: S\ra \bP_{\Delta}$$
be the corresponding toric morphisms. For each of $\bP_{\Delta}$, $\bP^1\times\bP^1$, $S$, we denote the torus invariant divisor corresponding to a ray of the fan by the same letter, so 
$$B=(1,0),\quad A=(0,1),\quad F=(-1,0),\quad G=(0,-1),$$
$$D_1=(m-1,2),\quad D_2=(1,-4),\quad D_3=(0,-1),$$
$$D_4=(-(m-1),m-3),\quad D_5=(-1,1).$$
On $S$ we have $G=D_3$ and 
$\rho$ contracts divisors $A, B, F$, while $\pi$ contracts $D_1,D_2, D_4, D_5$. 
\end{notation}

\begin{lemma}\label{BAFG divisors}
The following equalities of divisors hold on $S$:
$$\pi^{-1}B=B+(m-1)D_1+D_2,\quad \pi^{-1}F=F+(m-1)D_4+D_5,$$
$$\pi^{-1}A=A+2D_1+(m-3)D_4+D_5,\quad \pi^{-1}G=D_3+4D_2.$$
\end{lemma}

\begin{proof}
If $\pi: Y\ra X$ is the weighted blow-up of a toric surface obtained by adding 
a ray generated by a primitive vector $f:=\alpha e_1+\beta e_2$ to a smooth cone of the fan of~$X$ generated by primitive vectors $e_1,e_2$,  
then the multiplicity of $V(f)$ in $\pi^{-1}V(e_1)$ is~$\alpha$. (Here $V(r)$ is the torus invariant divisor 
corresponding to the ray~$r$). 
\end{proof}

As is customary, we view a rational function $f$ on a curve $C$ 
as the map $C\ra\bP^1$. 

\begin{proposition}\label{pentagons} 
Let $\phi=(f,g): C\ra\bP^1\times\bP^1$ be the morphism given by the rational functions $f$, $g$.
Let $U$ be the open torus in  $\bP^1\times\bP^1$, with coordinates $(u,v)=([1,u], [1,v])$ and let 
$\Gamma:=\phi(C)\cap U$. There are unique morphisms
$$\chi: C\ra \bP_{\Delta},\quad \psi: C\ra S,$$
that commute with $\phi$ and $\pi$, $\rho$ as defined in Notation \ref{maps}.
Then:
\begin{itemize}
\item[(1) ] The map  $\phi$ is birational onto its image and 
the equation of $\Gamma$ in $U$ is
$$(uv+2x_0^{k+2})\big(u-2x_0^{k+1}\big)^{m-1}-2u^{k+1}(v+x_0)^{k+2}\big(u-2x_0^{k+1}\big)^{k+2}-$$
$$-u^{m-3}(v+x_0)^{m-1}\big(uv+u(x_0-x_1)+2x_1x_0^{k+1}\big)=0.$$
The Newton polygon of $\Gamma$ is $\Delta$ and the multiplicity of $\Gamma$ at the point 
$q$ with $u=2x_0^{k+1}$, $v=-x_0$ is $m$.
\item[(2) ] For $D_i$ ($i=1,2,4,5$) in $\bP_{\Delta}$, we have $\chi^{-1}(D_i)=d_i$ (see Notation \ref{setup}). 
\item[(3) ] The induced map $\chi: C\ra \Bl_q\bP_{\Delta}$ is an embedding and the linear system $\cL_{\Delta}(m)$ has $C$ as an irreducible member.
Via this identification, we have  
$$\cO(C)_{|C}\cong\cO_C(2d_1-2d_2).$$ 
Furthermore, if $E$ is the exceptional divisor in  $\Bl_q\bP_{\Delta}$, then $\chi^{-1}(E)$ is a common fiber of the maps 
$C\rightarrow \bP^1$ induced by $f$ and $g$. In particular, 
$$\chi^{-1}(E)\sim (m-1)d_4+d_5\sim (m-1)d_1+d_2\sim 4d_2+\gamma\sim 2d_1+(m-3)d_4+d_5.$$
\end{itemize}
\end{proposition}

\begin{proof}
We first prove (1). 
Set $f(x,y)=u$, $g(x,y)=v$ and solve for $x,y$. Noticing that 
$uv=x\big(x^{k+1}-x^ky-2x_0^{k+1}\big)$, 
we obtain after some calculations that
\begin{equation}\label{x,y}
x=\frac{u(v+x_0)}{u-2x_0^{k+1}},\quad y=\frac{x^{k+2}-u(x-x_0)}{x^{k+1}}.
\end{equation}
In particular, the map $\phi$ is birational onto its image. It follows that $(u,v)\in\Gamma$ must satisfy the equation 
obtained by plugging in the above formulas for $x,y$ in the Weierstrass equation of $C$. After clearing denominators, this equation is
$$\big(x^{k+2}-u(x-x_0)\big)^2=x^{2k+2}(x^3+ax^2+bx).$$
This equation has a solution $x=x_0$, since the point $y=x=x_0$ lies on $C$. More precisely, one can factor out $(x-x_0)$, by noticing that
$$(x^{k+2})^2-x^{2k+3}(x^2+ax+b)=x^{2k+3}(x-x_0)(x-x_1).$$
As $\Gamma$ is irreducible and $x$ is not always equal to $x_0$ along $C$, it follows that $(u,v)\in\Gamma$ must satisfy the equation 
$$u^2(x-x_0)-2ux^{k+2}=x^{2k+3}(x-x_1),$$
where $x$ is as in (\ref{x,y}). Substituting $x$ with the formula in  (\ref{x,y}) and simplifying $u^2$ ($u$ is not constant equal  to $0$ along $\Gamma$, otherwise $x=0$)
it follows $(u,v)\in\Gamma$ must satisfy the given equation. Note, the equation is of type $(m,m)$ in $\bP^1\times\bP^1$. 
Since each of the maps given by the rational functions $f$ and $g$ has degree $m$ and $\phi$ is birational 
onto its image, it follows that the closure of $\Gamma$ in $\bP^1\times\bP^1$ is a curve of type $(m,m)$. In particular, the equation we obtained is irreducible and defines 
$\Gamma$ on~$U$. 
As already noted, the Newton polygon is inscribed in the square with vertices $(0,0)$, $(0,m)$, $(m,0)$, $(m,m)$, the terms 
$$u^{m-3}v^{m-1},\quad u^{m-2}v^m,\quad u^mv,\quad 1$$ 
appear with non-zero coefficients, and there are no terms $u^mv^i$ except when $i=1$, or terms $u^{m-1}v^i$ when $i\geq k+3$. It follows that the Newton polygon 
has as an edge the segment joining the points $(m-2,m)$, $(m,1)$. Similarly, to check that the edges joining $(0,0)$, $(m-3,m-1)$ and $(m-3,m-1)$, $(m-2,m)$ respectively, it 
suffices to check that there are no terms $u^iv^j$ with $j/i>(m-1)/(m-3)$ and no term $u^{m-3}v^m$ respectively. This is straightforward. 
It remains to prove that 
$$u^{m-3},\quad u^{m-2},\quad u^{m-1},\quad u^m$$
appear with zero coefficients, but $u^{m-4}$ has a non-zero coefficient. Setting $v=0$, the equation becomes (after simplifying $x_0^{k+2}$)
$$2\big(u-2x_0^{k+1}\big)^{m-1}-2u^{k+1}\big(u-2x_0^{k+1}\big)^{k+2}-x_0^{k+1}u^{m-3}\big(u(x_0-x_1)+2x_1x_0^{k+1}\big)=0.$$
Recall that $m=2k+4$. Clearly, $u^{m-1}$ and $u^m$ appear with $0$ coefficient. It is straightforward to check that the coefficients of $u^{m-2}$ and $u^{m-3}$ are 
$$-4x_0^{k+1}\binom{m-1}1 
+4x_0^{k+1}
\binom{k+2}1-x_0^{k+1}(x_0-x_1)=0,\quad \text{and }$$
$$2(2x_0^{k+1})^2
\binom{m-1}2-2(2x_0^{k+1})^2
\binom{k+2}2-2x_0^{2k+2}x_1=0,\quad \text{respectively}.$$
The coefficient of $u^{m-4}$ is
$$-2(2x_0^{k+1})^3\binom{m-1}3 +2(2x_0^{k+1})^3\binom{k+2}3,$$
which is non-zero for all $k\geq0$. 
Making the change of variables $s:=u-2x_0^{k+1}$, $t=v+x_0$, the equation of $\Gamma$ becomes
$$s^{m-1}\big(st+2x_0^{k+1}t-x_0s\big)-2(st)^{k+2}\big(s+2x_0^{k+1}\big)^{k+1}-t^{m-1}\big(st+2x_0^{k+1}t-x_1s\big)=0,$$
which has a point of multiplicity $m$ at $s=t=0$. This finishes the proof of (1). 

We now prove (2) and (3). 
Denote $d'_i=\chi^{-1}(D_i)$.  
Clearly, $\vol(\Delta) = m^2$ and $|\partial\Delta\cap \bZ^2| = m$. 
Let $C'$ be the proper transform in $\Bl_q{\bP_{\Delta}}$ 
of the closure of $\Gamma$ in $\bP_{\Delta}$. Note that the map $\phi$ factors through $C'$, and 
$C$ is the desingularization of $C'$. Up to a change of coordinates on $U$, we are in the situation 
of  Proposition~\ref{prop:gen}. In particular, $C'$ has arithmetic genus one and hence it must be isomorphic to $C$. We  identify $C$ with its image in $\Bl_q{\bP_{\Delta}}$. 
As the edges $D_i$ for $i\neq 3$ of $\Delta$ have lattice length $1$, it follows that 
each of  $d'_i$, for $i=1,2,4,5$, is a point. Since $C$ does not pass through the torus invariant points of $\bP_{\Delta}$, 
 the cycle $d'_3$ is disjoint from $d'_i$ for $i=1,2,4,5$ and $C$  embeds into $\Bl_e S$ and is disjoint from the torus invariant divisors $A, B, F$. Hence, 
 $d'_i=\psi^{-1}(D_i)$ for all $i$. By Lemma \ref{BAFG divisors} 
$$\phi^{-1}B=(m-1)d'_1+d'_2,\quad \phi^{-1}F=(m-1)d'_4+d'_5,$$
$$\phi^{-1}A=2d'_1+(m-3)d'_4+d'_5,\quad \phi^{-1}G=d'_3+4d'_2.$$
By the definition of the map $\phi$, the preimages of the torus invariant divisors in $\bP^1\times\bP^1$ are given by the zeros and poles of the rational functions
$f$ and $g$, so by Lemma \ref{basics}, these are
$$\phi^{-1}(u=\infty)=(m-1)d_1+d_2,\quad \phi^{-1}(u=0)=(m-1)d_4+d_5,$$
$$\phi^{-1}(v=\infty)=2d_1+(m-3)d_4+d_5,\quad \phi^{-1}(v=0)=2d_2+\gamma.$$
where $\gamma$ is an effective divisor disjoint from $d_i$ for $i=1,2,4,5$.  
By considering multiplicities, the only possibility that these divisors match is when $d_i=d'_i$ for all $i$. For example, the divisor $\phi^{-1}(v=\infty)$ must equal $\phi^{-1}A$, hence, 
$d_i=d'_i$ for $i=1,4,5$. Similarly, $\phi^{-1}(u=\infty)$ must equal one of $\phi^{-1}B$ or $\phi^{-1}F$ and as $d_1=d'_1$, it must be that $d_2=d'_2$. 
The exceptional divisor $E$ of $\Bl_q(\bP_\Delta)$
restricts to $C$ as an effective degree $m$ divisor which is contracted by both maps $C\ra\bP^1$. Hence, it is a common fiber of the two maps and 
$E_{|C}\sim (m-1)d_1+d_2$. 

Up to a change of coordinates on $U$, the linear system $\cL_{\Delta}(m)$ has $C$ as an irreducible member. To prove that  
 $\Delta$ is a Lang--Trotter polygon, it suffices to prove that $\rd(C)\in\Pic^0(C)$ is non-torsion (this also implies that $\dim \cL_{\Delta}(m)=0$). Let $X=\Bl_q(\bP_\Delta)$
and let $E$ be the exceptional divisor.  We have the following relations between the torus invariant divisors on $\bP_\Delta$, and hence, on $X$:
 $$(m-1)D_1+D_2\sim (m-1)D_4+D_5,\quad D_3\sim 2D_1-4D_2+(m-3)D_4+D_5.$$
From the fan of $\bP_\Delta$, we can compute the intersection numbers  $D_i\cdot D_j$.  
Using that $C\cdot D_i=1$ ($i\neq 3$), $C\cdot E=m$, we obtain 
 $$C\sim m(m+1)D_1+(m-2)D_2-2(m-1)D_4-mE.$$ 
It follows that $\rd(C)=2d_1-2d_2$. 
\end{proof}

\begin{proof}[Proof of Theorem~\ref{asgsgsrSRH}]
By Prop.~\ref{pentagons}(3), up to a change of coordinates on $U$, the linear system $\cL_{\Delta}(m)$ has $C$ as an irreducible member. It follows from Prop.~
\ref{pentagons}(3) and Lemma \ref{basics}(iii) that $\rd(C)\in\Pic^0(C)$ is non-torsion. This also implies that $\dim \cL_{\Delta}(m)=0$ and hence, 
 $\Delta$ is a Lang--Trotter polygon.  The proper transforms in $\Bl_q S$ of the two one-parameter subgroups $C_1$ and $C_2$ of $\bP^1\times\bP^1$ 
have classes $\pi^{-1}B-E$ and  $\pi^{-1}A-E$, respectively. It follows  by Lemma \ref{BAFG divisors} that their proper transforms $C_1$, $C_2$  in $X=\Bl_q\bP_\Delta$ 
have classes $(m-1)D_1+D_2-E$ and $2D_1+(m-3)D_4+D_5-E$, respectively. It follows that on $X$, we have $C\cdot C_1=C\cdot C_2=0$ and $C_1$, $C_2$ are 
$(-1)$-curves. Since $\rho(X)=4$, it follows that the minimal model $(C,Y)$ of the elliptic pair $(C,X)$ is obtained by contracting $C_1$ and $C_2$ and $\rho(Y)=2$ and every prime is polyhedral. 
\end{proof} 
 
\begin{remark}
The classes of the two one-parameter subgroups $C_1$, $C_2$ can be found from Lemma \ref{BAFG divisors}. Using the relations between torus invariant divisors, one obtains
$$C_1\sim (m+1)D_1+D_2-2D_4-E,\quad C_2\sim (m-1)D_1+D_2-E.$$
It follows that 
$K+C=(k+1)C_1+(k+2)C_2$. 
For $1\leq k\leq 5$ we checked using Computation~\ref{efvwefvwef} 
that the root lattice is $\bD_6\oplus \bA_1\oplus \bA_1$ and the Mordell-Weil group of $C$ is $\bZ\times\bZ/2\bZ$. 
\end{remark}
 
\begin{remark}\label{strategy}
The reader may wonder how did we divinate the Weierstrass 
equation of $C$ in Notation~\ref{setup}.
We explain how 
to arrive at the equation of $C$ starting from the polygon~$\Delta$,
assuming 
it can be inscribed in a square with sides of length $m$.
In this case, we may add the normal rays of the square
to the rays of the normal fan of $\Delta$ to obtain a toric surface $S$ with maps
$S\ra\bP_{\Delta}$, $S\ra\bP^1\times\bP^1$. If the hypothetical curve $C$ defined by the polygon $\Delta$ is smooth, then the canonical map $\chi: C\rightarrow\Bl_e\bP_\Delta$ lifts to a map 
$C\rightarrow \Bl_e S$. The divisor $E_{|C}$ has degree $m$ and is contracted by the maps $C\rightarrow S$, and hence also by $\phi: C\rightarrow \bP^1\times\bP^1$.   
As the width of $\Delta$ in horizontal and vertical directions
is $m$, the two maps $C\ra\bP^1$ are of degree $m$. As $E_{|C}$ has degree~ $m$ and is contracted by both maps, it follows that $E_{|C}$ is a common fiber of both maps. If $\phi$ is given by $(f,g)$, where $f$ and $g$ are rational functions on $C$, it follows that 
the divisors of zeros and poles of both $f$ and $g$ (that is, the preimages of the torus invariant divisors in $\bP^1\times\bP^1$) are linearly equivalent to $E_{|C}$.  
The preimages of the torus invariant divisors of  $\bP^1\times\bP^1$ in $S$ can be computed directly from the fan of $S$ (as in Lemma \ref{BAFG divisors}).  
Letting $d_i=\chi^{-1}(D_i)$, where $D_i$ are the torus invariant divisors on $\bP_{\Delta}$, we obtain linear relations satisfied by the cycles~$d_i$ (points if the corresponding edge has lattice length $1$) that eventually determine a Weierstrass model of $C$.   For example, for the pentagons in Notation~\ref{maps}, one obtains from Lemma \ref{BAFG divisors}
and the above argument that 
$$E_{|C}\sim (m-1)d_4+d_5\sim (m-1)d_1+d_2\sim 4d_2+\gamma\sim 2d_1+(m-3)d_4+d_5.$$
It follows that $2d_4\sim 2d_1$, $d_4+d_5\sim d_1+d_2$. Choosing $d_1$ as the point at infinity and $d_4=(0,0)$ for an elliptic curve with Weierstrass equation $y^2=x^3+ax^2+bx$, we obtain a formula for the rational functions $f,g$ whose zeros and poles are as in Lemma \ref{basics}. Along the way, one has to impose the condition that in the linear system given by $(m-1)d_4+d_5\sim (m-1)d_1+d_2\sim  2d_1+(m-3)d_4+d_5$ there exists an element vanishing with multiplicity $\geq 4$ at $d_2$. 
\end{remark}

\begin{remark}
Pentagonal curves are fibers $C_k$
of an elliptic fibration $\cC\to \bP^1$ with the Weierstrass normal form
of Notation~\ref{setup} (the field of rational functions on $\bP^1$
is the field of rational functions in variable~$k$).
By~Computation~\ref{MZNdc,manBDC}, $\cC$ is a rational
elliptic fibration of Kodaira type $I_4I_2^{\oplus3}I_1^{\oplus2}$.
One can compute the Neron--Tate height of 
the section of this fibration 
corresponding to  
$d_2$ to conclude that it is not torsion in the Mordell--Weil group
of the  elliptic fibration.
This shows that $d_2$
is not torsion in a fiber $C_k$ for almost all $k$ 
by Silverman's specialization theorem~\cite{SilST}.
Mazur's theorem gives a more precise statement for every $k$ as above. 
\end{remark}

\subsection*{An infinite sequence of heptagons.}\label{hepta}
Let $k\geq 2$ be an integer and $m=2k+4$. Let $\Delta$ be the heptagon with vertices
$$(0,0),\quad (1,0),\quad (m,2),\quad (m,m-4),\quad (m-1,m),\quad (m-2,m),\quad (k,k+1).$$

\begin{figure}
\begin{tikzpicture}[scale=.5]
\tkzInit[xmax=8,ymax=8]\tkzGrid
\tkzDefPoint(0,0){P1}
\tkzDefPoint(1,0){P2}
\tkzDefPoint(8,2){P3}
\tkzDefPoint(8,4){P4}
\tkzDefPoint(7,8){P5}
 \tkzDefPoint(6,8){P6}
\tkzDefPoint(2,3){P7}
\tkzDrawSegments[color=black](P1,P2)
\tkzDrawSegments[color=black](P2,P3)
\tkzDrawSegments[color=black](P3,P4)
\tkzDrawSegments[color=black](P4,P5)
\tkzDrawSegments[color=black](P5,P6)
\tkzDrawSegments[color=black](P6,P7)
\tkzDrawSegments[color=black](P1,P7)
\node[below] at (4,0) {$G$};
\node[right] at (8,4) {$B$};
\node[left] at (0,4) {$F$};
\node[above] at (4,8) {$A$};
\node[below] at (.5,0) {$D_4$};
\node[below] at (4.5,1) {$D_3$};
\node[right] at (8,3) {$D_2$};
\node[left] at (7.7,6) {$D_1$};
\node[above] at (6.5,8) {$D_7$};
\node[left] at (4.2,6) {$D_6$};
\node[left] at (1.2,2) {$D_5$};
\end{tikzpicture}
\caption{Polygon $\Delta$ for $k=2$}\label{lkwjRBGkgr;r}
\end{figure}
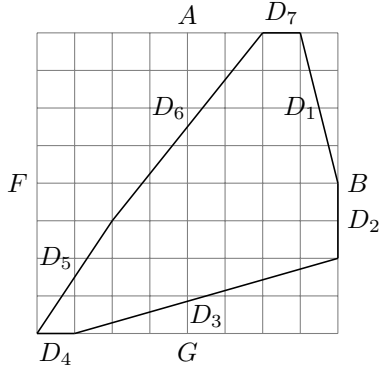

\begin{theorem}\label{hepta thm}
The  polygon $\Delta$ is Lang--Trotter for every $k\ge2$.
In particular, $\oEff(\Bl_e\bP_\Delta)$ is not polyhedral in characteristic $0$.
Furthermore, for all but finitely many $k$, the set of non-polyhedral primes of $\Delta$ has positive density.
\end{theorem}

\begin{proof}
The strategy is the one in Remark \ref{strategy}. 
The corresponding curve $C$ is a smooth elliptic curve defined over $\bQ$ with equation 
$$y^2+exy+by=x^3+ax^2,\quad \text{where}$$
$$e=-(4k+2),\quad a=-\frac{k(2k+1)}{k+2},\quad b=\frac{4k(k+1)^4(2k+1)}{(k+2)^2(k-1)^2}.$$

Labeling the edges and the corresponding torus invariant divisors in $\bP_{\Delta}$ as in Figure~\ref{lkwjRBGkgr;r}, we let $d_i=\chi^{-1}(D_i)$. Then $d_2$ is an effective divisor of degree $m-6$ and
all $d_i$ for $i\neq 2$ are points on $C$.  Let $\tilde{d}_4$ be defined by $d_4+\tilde{d}_4\sim 2d_3$. The points $d_i$ have the following properties:
$$d_1+d_7\sim 2 d_3,\quad d_5+d_6\sim 2d_3,\quad 2d_6+\tilde{d}_4\sim 3 d_3.$$

We choose $d_3$ to be the point at infinity $[0,1,0]$. In our Weierstrass model, the following lines intersect $C$ at the following points with multiplicities:
\begin{align*}
x=0: d_5+d_6+d_3, \quad & y=0: 2d_6+\tilde{d}_4,\\
x=x_0: d_1+d_7+d_3, \quad & x=-a: d_4+\tilde{d}_4.
\end{align*}
The points are 
$d_6=(0,0)$, $d_5=(0,-b)$, $\tilde{d}_4=(-a,0)$, $d_4=(-a,ae-b)$, $d_1=(x_0,y_0)$, $d_7=(x_0,y_1)$, where $y_1=-y_0-ex_0-b$ and 
$$x_0=\frac{2k(k+1)^2}{(k-1)(k+2)},\quad y_0=-\frac{2k(k+1)^2(5k+3)}{(k-1)^2(k+2)}.$$

The torus invariant divisors $D_1,\ldots,D_7$ satisfy 
$$D_1+kD_5+(k+2)D_6+D_7\sim(m-1)D_3+D_4,\quad 4D_1+D_2+2D_3\sim (k+1)D_5+(k+3)D_6.$$
Using that $C\cdot D_i=1$ ($i\neq 2$), $C\cdot E=m$, we obtain
that  in $\Cl(\Bl_e\bP_\Delta)$ we have
$$C\sim -4D_1+(m-2)(m-1)D_3+mD_4+mD_5+(m+2)D_6-mE.$$

There are three one-parameter subgroups $C_1$,$C_2$,$C_3$ corresponding to lattice directions $\lambda_1=(1,0)$, 
$\lambda_2=(0,1)$,  $\lambda_3=(1,-1)$, and with respect to which the width of 
$\Delta$ is $m$ (hence, $C\cdot C_i=0$ for $i=1,2,3$). It follows that the following hold on $C$
$$d_1+d_7\sim d_5+d_6\sim d_4+\tilde{d}_4\sim 2d_3,\quad 2d_6+\tilde{d}_4\sim 3d_3,$$
$$\cO(C)_{|C}=\cO(2d_3+2d_6-4d_1).$$

There is a map $\phi: C\ra\bP^1\times\bP^1$, corresponding to rays $\lambda_1$, $\lambda_2$, and 
given by rational functions
$$f(x,y)=\frac{x^{k+1}y}{\alpha(x-x_0-1)+\beta(x+a)},\quad g(x,y)=\frac{x+a}{x^k(x-x_0)y},$$
$$\text{where}\quad \alpha=x_0+a=\frac{k(5k+3)}{(k-1)(k+2)},\quad \beta=x_0^{k+1}y_0.$$
The pullbacks of the torus invariant divisors of $\bP^1\times\bP^1$ corresponding to the edges $A,B,F,G$ correspond 
to the zeros and poles of $f,g$ by $f=F/B$, $g=G/A$. 
The map $\phi$ is birational onto its image. Letting $u,v$ be coordinates on $\bP^1\times\bP^1$ and solving for $x,y$ in $f(x,y)=u$, $g(x,y)=v$, we obtain that 
$\phi(C)$ has equation
$$\big((\alpha\beta)uv+(\alpha x_0)u-a\big)h_2(u,v)^{m-1}-h_1(u,v)^{m-2}h_3(u,v)^2v^2+$$
$$+\big(\beta(b+ex_0)uv+(ex_0\alpha)u-b\big)h_1(u,v)^kh_2(u,v)^{k+2}h_3(u,v)v=0,\quad \text{where}$$
$$h_1(u,v)=(x_0\beta)uv+(x_0\alpha)u,\quad h_2(u,v)=(\beta)uv-1,\quad h_3(u,v)=(x_0\alpha)u+x_0.$$
It is straightforward to check that this equation has $\Delta$ as its Newton polygon. 

One computes the classes of the one-parameter subgroups $C_1$, $C_2$, $C_3$ as
$$C_1\sim (k+1)D_5+(k+3)D_6-E,\quad C_2\sim (m-1)D_3+D_4-E,$$
$$C_3\sim (m-3)D_3+D_4+D_5+D_6-E.$$
It follows that $K+C\sim C_1+(k+1)C_2+(k+1)C_3$. 

Computation~\ref{MZNdc,manBDC} (based on Mazur's theorem as in \S~\ref{adrhdjd})
shows that $\cO(C)|_C$ is not torsion for $k\ge2$ by showing that $\cO(2d_1-d_6-d_3)$ is not torsion. In particular, 
$\Delta$ is a Lang--Trotter polygon in this range.
The point $p\in C$ such that $\cO(2d_1-d_6-d_3)=\cO(p-d_3)$ is given by 
$$x=\frac{4k(k+1)^2(2k+1)}{(k-1)^2(k+2)},\quad y=\frac{4k^2(k+1)^2(2k+1)(3k+1)}{(k-1)^2(k-2)}.$$

Denote $X=\Bl_e\bP_\Delta$ and let $\pi: X\ra Y$ be the map that contracts the one parameter subgroups $C_1$, $C_2$ and $C_3$. 
We now compute directly generators $\Cl(Y)$ and $\Cl_0(Y)$. The group $\Cl(X)$ is generated over $\bZ$ by $D_1,D_3,D_4,D_5,D_6$ and $E$. 
It follows that $C^\perp$ is generated over $\bZ$ by 
$$D_3-D_1,\quad D_4-D_1,\quad D_5-D_1,\quad D_6-D_1,\quad E-m D_1.$$
Denote by $\overline{D}_i$, $\overline{E}$ the classes of $D_i$, $E$ in $\Cl(Y)$. Setting the classes of $C_1$, $C_2$, $C_3$ to zero, we obtain the following relations in $\Cl(Y)$: 
$$\overline{D}_4=3\overline{D}_3-2\overline{D}_5,\quad 
\overline{D}_6=2\overline{D}_3-\overline{D}_5,\quad \overline{E}=(m+2)\overline{D}_3-2\overline{D}_5.$$
Then $\Cl(Y)$ is generated by $\overline{D}_1$,
$\overline{D}_3$, $\overline{D}_5$ and $C_Y^\perp\subseteq \Cl(Y)$ is generated by 
$$\alpha:=\overline{D}_3-\overline{D}_1,\quad \beta:=\overline{D}_5-\overline{D}_1,$$
with $\overline{D}_4-\overline{D}_1=3\alpha-2\beta$, 
$\overline{D}_6-\overline{D}_1=2\alpha-\beta$, 
$\overline{E}-m\overline{D}_1=(m+2)\alpha-2\beta$. Since the class of $C_Y$ can be expressed as $6\alpha-2\beta$, it follows that
$$\bE_8/T=\Cl_0(Y)=\bZ\{\alpha,\beta\}/\bZ\{6\alpha-2\beta\}\cong \bZ\times\bZ/2\bZ.$$ 
The class of $C$ is divisible by $2$. In $\Cl_0(Y)$ the 
class $\frac{1}{2}C_Y=3\alpha-\beta$ is the unique non-zero torsion element. It follows that there is a commutative diagram
\begin{equation}\label{diagram}
\begin{CD}
\bE_8/T=\Cl_0(Y) @>\ored>> \Pic^0(C)/\langle\rd(C)\rangle  \\
@VVV        @VVV\\
\bE_8/\hat{T}=\Cl_0(Y)/\text{torsion}   @>>> \Pic^0(C)/\langle\rd(\frac{1}{2}C)\rangle
\end{CD}
\end{equation}

One can compute directly using a resolution of $X$ that the root lattice is $$T=\bA_3\oplus \bA_2\oplus\bA_1^{2}.$$ 
There is a unique embedding of $T$ in $\bE_8$ and it follows that the group $\bE_8/T$ is isomorphic to $\bZ\{a,b\}/\bZ\{6a+2b\}$, with the roots in $\bE_8\setminus \hat{T}$ having image in $\bE_8/T$ belonging to the set $\{\pm a,\pm b,\pm (a+b), \pm (2a+b)\}$. It follows that in 
$\bE_8/\hat{T}\cong\bZ$, we have $a=\pm \alpha$ and in $\bE_8/\hat{T}\cong\bZ$ the images of these roots are the classes of 
$\{\pm \alpha, \pm 2\alpha, \pm 3\alpha\}$. In order to prove that $\ored(\gamma)\neq0$ for any root $\gamma\in\bE_8\setminus\hat{T}$ (for some characteristic $p$), by  (\ref{diagram}) it suffices to prove that $\rd(\alpha)$ is not in the subgroup generated by $\rd(\frac{1}{2}C)$. 
To prove that this holds for a set of primes of positive density, 
we apply Lemma \ref{qfvqefvefv} to 
$$x_i=\rd(i\alpha)=\cO_C(id_3-id_1),\quad
x_0=\rd(\frac{1}{2}C)=\cO(d_3+d_6-2d_1)$$
for $i=1,2,3$.
We check that the conditions in the lemma are satisfied.
The curve $C$ does not have complex multiplication because its $j$-invariant
$$\scriptstyle
-4096(k^{12}+1) - 24576(k^{11}+k) - 58368(k^{10}+k^2) - 66560(k^9+k^3) - 9216(k^8+k^4) + 
    92160(k^7+k^5) + 141312k^6   
    \over 
    \scriptstyle
    27(k^8 + 4k^7 + 6k^6 + 4k^5 + k^4)$$
is not an integer (see \cite{Silv_Advanced}*{Thm. II.6.1}). 
We already proved that $x_0$ is not torsion in $\Pic^0(C)$.  To prove that 
$x_1$ also has infinite order, it suffices to prove that $\cO(d_6-d_1)$
is not torsion, which follows again by  Computation~\ref{MZNdc,manBDC} (based on Mazur's theorem as in \S\ref{adrhdjd}). It remains to prove 
that 
$\rd(\alpha)$ and $\rd(\frac{1}{2}C)$ 
(equivalently, $\cO_C(d_1-d_3)$ and $\cO_C(d_6-d_3)$)
are linearly independent for almost all $k$.
Using Silverman's specialization theorem \cite{Silv}*{App. C, Thm. 20.3} for the elliptic fibration defined by all the heptagonal curves $C$ for 
$k\ge 2$ (see Remark \ref{hepta fibration}), it suffices to prove the statement for a specific~$k$, which we do by a computer calculation.
\end{proof}

\begin{remark}
 The Mordell-Weil group of $C$ is $\bZ$ for $k=2$ and $\bZ\times\bZ$ for $3\leq k\leq 6$. 
\end{remark}

\begin{remark}\label{hepta fibration}
``Heptagonal'' 
curves $C$ for $k\ge 2$ can be viewed as fibers $C_k$
of an elliptic fibration $\cC\to \bP^1$ with the Weierstrass normal form
of Notation~\ref{setup} (here the field of rational functions on $\bP^1$
is the field of rational functions in parameter~$k$).
By~Computation~\ref{MZNdc,manBDC}, $\cC$ is a K3
elliptic fibration of Kodaira type $I_4^{\oplus3}IV^{\oplus3}$.
\end{remark}


\section{A smooth Lang--Trotter polygon}\label{sgasrhasrh}

In this section we discuss an example
of a minimal elliptic pair $(C,Y)$
with $Y$ smooth and such that 
$\Pic^0(C)(\mathbb Q)$ has 
rank $9$.
Let $\Delta\subseteq\mathbb Q^2$ be
the lattice polytope whose vertices 
are the $19$ columns of the following matrix:
\[\small
\left[
\begin{matrix}
3&6&8&23&27&30&30&29&21&18&16&13&12&11&9&7&1&0&0\\
0&1&2&12&15&18&19&20&26&28&29&30&30&29&25&20&4&1&0
\end{matrix}.
\right]
\]
The polygon $\Delta$ has width
$m := 30$, which is obtained along
the directions $[1,0]$, $[0,1]$, $[1,-1]$.
Its volume is $m^2$ and it has $m$
boundary points. It follows that 
any curve in the linear system 
$\mathcal{L}_\Delta(m)$ has arithmetic 
genus one.
A computer calculation shows that 
$\mathcal{L}_\Delta(m)$ is zero-dimensional
and that its unique element is an  irreducible curve of geometric genus
one, whose defining polynomial 
has Newton polygon $\Delta$.
Thus we get an elliptic pair $(C,X)$,
where $X$ is the blowing-up of the 
toric surface defined by $\Delta$
and $C$ is the strict transform of
the unique curve linearly equivalent
to the following Weil divisor:
\[\small
\begin{bmatrix}
19&30&12&7&7&1&0&0&1&3&6&16&11&29&48&117&187&72&30&-30
\end{bmatrix}
\]
where the first $19$ entries are the 
coordinates of the pullbacks 
$D_1,\dots D_{19}$ of the prime
invariant divisors of
the toric variety, while the last
coordinate is the coefficient of
the exceptional divisor $E$.
Observe that $X$ is smooth of Picard
rank~$18$.
The linear system $|K_X+C|$ contains
eight disjoint $(-1)$-curves, three
of which come from the one-parameter
subgroups defined by the width directions of $\Delta$,
while the remaining ones come from 
curves of multiplicity $2,3,5,5,11$
at $(1,1)$. A list of eight Weil
divisors, each of which is linearly equivalent to one of the
above eight curves, is given by the
rows of the following matrix, where
we have kept the same notation used
for the curve $C$ above:
\[\small
\begin{bmatrix}
1&3&2&3&4&1&0&0&0&0&0&0&0&0&0&0&0&0&0&-1\\
0&0&0&0&0&0&0&0&0&0&0&0&0&1&2&5&8&3&1&-1\\
1&2&1&1&1&0&0&0&0&0&0&0&0&0&1&3&5&2&1&-1\\
1&1&0&0&0&0&0&0&1&1&1&2&1&2&3&8&13&5&2&-2\\
2&3&1&0&0&0&0&0&1&1&1&2&1&3&5&12&19&7&3&-3\\
3&5&2&1&1&0&0&0&0&0&1&3&2&5&8&20&32&12&5&-5\\
3&5&2&1&1&0&0&0&1&1&1&3&2&5&8&19&31&12&5&-5\\
7&11&4&2&2&0&0&0&0&1&2&6&4&11&18&44&70&27&11&-11
\end{bmatrix}
\]
Each of the divisors $D_2, D_5$ and
$D_{12}$ is a $(-1)$-curve of $X$, 
having intersection number $1$
with $C$, so that it is disjoint 
from the curves in $|K_X+C|$.
This claim can be easily proved
looking at the primitive generators
$\varrho_1,\dots,\varrho_{19}$ of the
normal fan of~$\Delta$. For example
$\varrho_1 = [0,1]$, $\varrho_2 = [-1,3]$, $\varrho_3 = [-1,2]$ show
that $D_2$ is a $(-1)$-curve, so that
the equality $D_2\cdot C = D_2 \cdot
(19D_1+30D_2+12D_3)$ gives 
$D_2\cdot C = 1$. A similar analysis
can be performed for the divisors 
$D_5$ and $D_{12}$.
As a consequence each of the three
divisors remains a $(-1)$-curve in
$Y$, after contracting the curves 
in $|K_X+C|$. In particular the 
linear system $|C+D_2+D_5+D_{12}|$
defines a rational map which
factorizes through $Y$, and there
it is defined by $|-K_Y+D_2+D_5+D_{12}|$.
The image of $Y$ via this linear
system is a smooth cubic surface of 
$\mathbb P^3$ whose equation can be
calculated by determining the
unique cubic relation between the
elements of a basis of 
$H^0(X,C+D_2+D_5+D_{12})$.
A distinguished basis of the latter 
vector space is given by a defining 
polynomial $f_0$ for $C$ together
with three polynomials $f_2,f_5,f_{12}$,
such that $\{f_0,f_i\}$ is a basis of
$H^0(X,C+D_i)$.
If we denote by $\varphi_i\colon
X\to X_i$ the contraction of $D_i$
then $C+D_i$ is the pullback of 
$\varphi_i(C)$ and thus we have
an isomorphism $H^0(X,C+D_i)
\simeq H^0(X_i,\varphi_i(C))$.
The curve $\varphi_i(C)$ is defined
by a modification $\Delta_i$ of the
polygon $\Delta$ obtained in the 
following way: the $(i-1)$-th and $(i+1)$-th 
edges are extended up to their intersection 
point. The latter is an integer point 
if and only if the equation
$\varrho_{i-1}+\varrho_{i+1} = \varrho_i$ holds, equivalently if
$D_i$ is a $(-1)$-curve on $X$. 
We display the construction of the
polygon $\Delta_i$ in the following
picture.

\begin{center}
\begin{tikzpicture}[scale=.5]
 \tkzDefPoint(0,-1){A1}
 \tkzDefPoint(0,0){A2}
 \tkzDefPoint(1,0){A3}
 \tkzDefPoint(2,-1){A4}
 \tkzDefPoint(0,1){p}
 \tkzDrawPoints[fill=black,color=black,size=5](A1,A2,A3,A4,p)
 \tkzDrawSegments(A1,A2 A3,A4 A2,A3)
 \tkzDrawSegments[densely dotted](A2,p A3,p)
\end{tikzpicture}
\end{center}

The normal fan to $\Delta_i$ coincides
with that of $\Delta$ at all rays but
$\varrho_i$. The dimension of the 
linear system increases by one because
a new monomial, corresponding to the
new point, has been added. 
A minimal model for the curve $C$
has equation 
$y^2 = x^3 + x^2 - 7860946299156x + 
8357826814810214400$.
Ordering counterclockwise the facets of $\Delta$, starting from the 
facet $(0,0)$ -- $(3,0)$, the indices
of facets of integer length one are 
$\{2, 3, 5, 7, 8, 10, 11, 12, 13, 14, 16,
18, 19\}$. For each such index one can 
compute the point $d_i\in C(\mathbb Q)$
cut out by the corresponding toric invariant
divisor $D_i$. This information is then 
used to compute the images of the $240$
roots and to determine the non-polyhedral
primes of $X$.
Using Computation~\ref{goodprimes}
we found $85$ non-polyhedral primes in the interval $[1,2000]$, or $28\%$.



\section{Halphen polygons}\label{finite orders}

We consider a variant of the notion of arithmetic elliptic pairs  as follows:

\begin{definition}\label{variant}
Let $(C,X)$ be an elliptic pair with $e:=e(C,X)<\infty$, defined over a finite extension $K$ of $\bQ$.
Let $R\subset K$ be its ring of algebraic integers. There exists a dense open subset $U\subset\Spec R$
and a pair of schemes $(\cC,\cX)$ flat over $U$, which we call an \emph{arithmetic elliptic pair of finite order} $e<\infty$, such that
\begin{itemize}
\item Each geometric fiber $(C,X)$ of $(\cC,\cX)$ is an elliptic pair of order $e$.
\item The contraction morphism $X\to Y$ to the minimal elliptic pair extends to the contraction of schemes $\cX\to\cY$ flat over $U$.
\end{itemize} 
We call $(\cC,\cY)$ the \emph{associated minimal arithmetic elliptic pair}. 
Let $X$, $Y$ be geometric fibers over a place $b\in U$, $b\ne 0$.
As before, we call $b$ a \emph{polyhedral prime} if $\oEff(Y)$ is polyhedral.
If $b$ is not polyhedral, then $\oEff(X)$ is also not polyhedral.
\end{definition}

Since over $\bC$ the subgroup $\langle\rd(C)\rangle\subset\Pic^0(C)$ is finite of order $e<\infty$, the order of the elliptic pair given by each geometric fiber of  $(\cC,\cX)$
stays constant on an open set in $\Spec R$, as it is defined by the condition that $i\rd(C)\neq 0$, for $i=1,\ldots,e-1$.

\begin{proposition}\label{almost too good}
Let $(\cC,\cX)$ be an arithmetic elliptic pair of finite order $e<\infty$ over some open set $U\subset\Spec R$. 
Let $(\cC,\cY)$ be the associated minimal arithmetic elliptic pair. 
Assume that 
\begin{itemize}
\item The geometric fiber $Y_\bC$ of $\cY$ has Du Val singularities, 
\item The cone $\oEff(Y_\bC)$ is not polyhedral. 
\end{itemize}
Then all but finitely many primes $b\in U$ are non-polyhedral. 
\end{proposition}

\begin{proof}
If the minimal elliptic pair $(C_\bC,Y_\bC)$ has du Val singularities, by replacing $U$ with a smaller open set, we may assume
that all geometric fibers $(C,Y)$ of $(\cC,\cY)$ over $U$ are minimal elliptic pairs of order $e$, with Du Val singularities and the same root lattice $T\subseteq\bE_8$. 
Indeed, there exists a scheme $\cZ$, smooth over (a possibly smaller) $U$, and a morphism $\pi:\cZ\to\cY$, flat over $U$, 
such that on geometric generic fibers $Z$ and $Y$, of $\cZ$ and $\cY$, this gives the minimal resolution $Z\to Y$.  We may assume that the exceptional locus of $\pi$ has geometric irreducible components $\cE_1,\ldots, \cE_r\subset \cZ$, smooth over $U$, such that the geometric generic fibers $E_1,\ldots, E_r$ are the exceptional $(-2)$-curves of the resolution $Z\to Y$. 
As each $\cE_i$ is flat over $U$, intersection numbers $E_i\cdot E_j$ of the geometric generic fibers do not depend on $b\in U$. 
In particular, the root lattice is the same for all $b\in U$, and all geometric fibers of $\cY\to U$ have Du Val singularities.  

Consider now  any geometric fiber $(C,Y)$ of $(\cC,\cY)$ and let $Z\ra Y$ be its minimal resolution. 
Recall that by Lemma \ref{lkJSBwkjbg} and Cor. \ref{concrete}, the cone $\oEff(Y)$ is polyhedral if and only if $\oEff(Z)$ is polyhedral, or equivalently, 
the kernel of the map 
$$
\ored:\,\Cl_0(X):=C^\perp/\langle K\rangle\to\Pic^0(C)/\langle\rd(K)\rangle
$$
contains $8$ linearly independent roots of $\bE_8=\Cl_0(Z)$. By assumption, 
the subgroup $\langle\rd(C)\rangle$ of $\Pic^0(C)$ is finite of fixed order $e<\infty$, for all geometric fibers. By 
Theorem \ref{sHAHA}, $C\sim n(-K)$ for some integer $n$. 
It follows that the subgroup $\langle\rd(K)\rangle$ of $\Pic^0(C)$ is finite of order $\leq e$ for every geometric fiber. 
Since there are finitely many roots in $\bE_8$, it follows that by eventually discarding a finite set of places $b\in U$, $b\neq0$, 
the maximum number of linearly independent roots of $\bE_8=\Cl_0(Z)$ contained in $\ker(\ored)$ is constant. This finishes the proof. 
\end{proof}

As in Notation \ref{arithmetic set-up}, we may consider arithmetic \emph{toric} elliptic pairs of finite order. 
Consider a lattice polygon $\Delta\subseteq\bZ^2$ and 
let $\cP$ be the projective  toric scheme over $\Spec\bZ$ given by the normal fan of $\Delta$. 
Let $\cX$ be the blow-up of $\cP$ along the identity section of the torus group scheme.
We will assume that $\Delta$ is a good, but not Lang--Trotter polygon, a so-called Halphen polygon (Def. \ref{def:good}). 
Then $(C_\bC,X_\bC)$ is an elliptic pair of finite order $e:=e(C_\bC,X_\bC)<\infty$ and 
$(\cC,\cX)$ an arithmetic elliptic pair  of finite order, flat over an open subset $U\subset\Spec \bZ$ 
(Def.~\ref{variant}). Let $\cX\to\cY$ be the morphism inducing 
the map to the minimal model on each geometric fiber.
\begin{definition}
A polygon $\Delta\subseteq\bZ^2$ such that the associated toric arithmetic elliptic pair  $(\cC,\cX)$
satisfies the conditions in Proposition \ref{almost too good} will be called 
a \emph{Halphen}$^+$ polygon. 
\end{definition}

\begin{theorem}\label{parabolic}
Let $\Delta$ be a Halphen$^+$ polygon. Then  $\oEff(X_\Delta)$ is not polyhedral in characteristic $0$ and characteristic $p$, for all 
but finitely many primes $p$. 
\end{theorem}
\begin{proof}
This is an immediate consequence of Proposition \ref{almost too good}.
\end{proof}

\begin{theorem}\label{ex3} 
Consider the polygon $\Delta$ with vertices:
\[
\left[
\begin{matrix}
0 & 1 & 6 & 8 & 7 & 5 & 1\\
0 & 0 & 1 & 2 & 5 & 8 & 2
\end{matrix}
\right]
\]
Then $\Delta$ is a Halphen$^+$ polygon and $\oEff(X_\Delta)$ is not polyhedral 
in characteristic $0$, and in characteristic $p$ for all primes $p\neq 2, 3, 5, 7, 11, 19, 71$. 
\end{theorem}

We will use this polygon later in the proof of Theorem~\ref{asgarh}. 

\begin{proof}
We have $\vol(\Delta) = 64$ and $|\partial \Delta\cap\bZ| = 8$
(see Computation~\ref{sdfvwefvwefv}). By Computation~\ref{asarsgwRG}, in characteristic $0$ the linear system $\cL_\Delta(8)$ 
has dimension $0$ and the unique curve 
$\Gamma\in\cL_\Delta(8)$  has equation
\begin{gather*}
   4\textcolor{red}{u^8v^2} + 24\textcolor{red}{u^7v^5} -61 u^7v^4 + 58u^7v^3 - 53u^7v^2 + 10u^6v^6 - 126u^6v^5 + \\
   + 244u^6v^4 - 186u^6v^3 + 150u^6v^2 + 20\textcolor{red}{u^6v} - \textcolor{red}{u^5v^8} + 8u^5v^7 - 48u^5v^6 + \\
   + 230u^5v^5 - 286u^5v^4 + 120u^5v^3 -159u^5v^2 - 88u^5v + 10u^4v^6 - 66u^4v^5 - \\
   - 56u^4v^4 + 144u^4v^3 +94u^4v^2 + 154u^4v - 6u^3v^5 + 89u^3v^4 -  26u^3v^3 - \\
   - 135u^3v^2 - 146u^3v -   54u^2v^3 + 52u^2v^2 + 114u^2v + 19\textcolor{red}{uv^2} - 46uv - 5\textcolor{red}u + \textcolor{red}4= 0.
\end{gather*}
The exponents of the red monomials
are the vertices of $\Delta$, so that 
the Newton polygon of $\Gamma$ is
$\Delta$ in characteristic $0$ and characteristic $p\neq 2,3,5,19$. By Computations ~\ref{adfafgarg} and \ref{asdcvq}, 
the curve $\Gamma$ is irreducible and
its strict transform $C\subseteq X_\Delta$ 
is a smooth elliptic curve in characteristic $0$, with 
minimal equation
\begin{equation}\label{minimalW parabolic}
y^2 + xy +y = x^3 + x^2 - 520x + 4745. 
\end{equation}
This is the curve labelled 
\href{https://www.lmfdb.org/EllipticCurve/Q/2130/j/4}{2130.j4}
in the LMFDB database~\cite{lmfdb}.
The Mordell--Weil group is $\bZ\times\bZ/4\bZ$. 
By Computation~\ref{asarsgwRG}, in characteristic $0$ the linear system $\cL_{k\Delta}(8k)$ 
has dimension $0$ if $k=2,3$ and dimension $1$ if $k=4$.  
It follows that  $\rd(C)\in \Pic^0(C)(\bQ)$ is torsion, of order $e=4$. Hence, 
$\Delta$ is a Halphen polygon. 

The theorem now follows from Computation \ref{badprimes}. We give the details. 
By  Computation \ref{badprimes}, the curve $C$ is irreducible and smooth in 
characteristic $0$ or characteristic $p\neq 2, 3, 5, 7, 11, 19, 71$. Unless otherwise specified, we will from now on assume we are in one of these situations. 

The normal fan of $\Delta$ has rays $v_1=(0,1)$, $v_2=(-1,5)$, $v_3=(-1,2)$, $v_4=(-3,-1)$, 
$v_5=(-3,-2)$, $v_6=(3,-2)$, $v_7=(2,-1)$. We denote $D_1,\ldots, D_7$ the corresponding torus invariant divisors in
$\bP_\Delta$, and abusing notations, also their pull-backs to $X_\Delta$. The divisors $D_1,\ldots,D_5,E$ form a basis for $\Cl(X)$ and we have:
$$D_6\sim 2D_1+9D_2+3D_3-5D_4-7D_5,\quad D_7\sim -3D_1-13D_2-4D_3+9D_4+12D_5,$$
$$K_X\sim 3D_2-5D_4-6D_5+E,\quad C\sim 2D_1+10D_2+7D_3+21D_4+24D_5-8E.$$
Note, the class of $C$ is independent of the characteristic if the Newton polygon stays the same. 
Since $\Delta$ has lattice width $8$ in the horizontal and vertical direction, 
the proper transforms $C_1$ and $C_2$ on $X_\Delta$ of the $1$-parameter subgroups  $(u=1)$ and $(v=1)$, 
are among the curves that must be contracted by the morphism $X\ra Y$ to the associated minimal elliptic pair.
Using Computation \ref{sfsvwefv} we find that $K_X+C=2C_1+2C_2+C_3$, with curves $C_i$ with classes
$$C_1\sim D_2+D_3+3D_4+3D_5-E\quad C_2\sim D_1+5D_2+2D_3-E,$$
$$C_3\sim D_2+D_3+10D_4+12D_5-3E.$$
Computation \ref{sfsvwefv} gives that the curve $C_3$ has equation 
$$u^3v-u^2v^3+3u^2v^2-5u^2v+uv+2u-1=0,$$
and so its Newton polygon has vertices $(0,0)$, $(1,0)$, $(3,1)$, $(2,3)$ in all characteristics other than $2$. 
This polygon has no non-trivial Minkowski decompositions, so the curve $C_3$ is irreducible in the situations we consider.  
The curves $C_1$, $C_2$ are irreducible in all characteristics, as they are proper transforms of $1$-parameter subgroups. 

From the intersection numbers $D_i\cdot D_j$ on $\bP_\Delta$ (or using Computation \ref{efvwefvwef}) we find that 
$C_1^2=-\frac{1}{4}$, $C_2^2=-\frac{3}{14}$, $C_3^2=-\frac{8}{3}$ and $C_i\cdot C_j=0$ for all $i\neq j$. Since the intersection matrix $(C_i\cdot C_j)_{i,j}$ is negative definite,
it follows that the Zariski decomposition of $K_X+C=N+P$ has the positive part $P\sim0$. By Theorem \ref{asfhadrhad}, the minimal model $Y$  has 
Du Val singularities. Denote $\overline{D}$ the class of a divisor $D$ in $\Cl(Y)$. Setting the classes of $C_1$, $C_2$, $C_3$ to zero, we obtain that
$\Cl(Y)$ is freely generated by $\overline{D}_2$, $\overline{D}_3$ and $\overline{D}_5$ and 
$$\overline{D}_1\sim 2\overline{D}_2+5\overline{D}_3-6\overline{D}_5,\quad \overline{D}_4\sim 2\overline{D}_2+2\overline{D}_3-3\overline{D}_5,$$
$$\overline{E}\sim 7\overline{D}_2+7\overline{D}_3-6\overline{D}_5,\quad C_Y=\overline{C}\sim 3\overline{D}_3-3\overline{D}_5.$$

We consider $\al:=\overline{D}_2-\overline{D}_5$, $\be:=\overline{D}_3-\overline{D}_5$ in $\Cl(Y)$.
Then $C_Y^\perp=\bZ\{\al,\be\}$ and 
$$\Cl_0(Y)=\bZ\{\al,\be\}/\bZ\{3\be\}=\bZ\times\bZ/3\bZ.$$
By Computation \ref{efvwefvwef}, or using a minimal resolution of $\bP_\Delta$, the root lattice is $T=\bA_2^3\oplus\bA_1$
and $\rho(Y)=3$. 
By Cor. \ref{concrete}, the cone $\oEff(Y)$ is non-polyhedral if and only if $\ored(\gamma)\neq0$, for all roots $\gamma\in\bE_8\setminus\hat{T}$. 
There is a unique way to embed $\bA_2^3\oplus\bA_1$ in $\bE_8$ (\cite{OS}[p.86]). There are 
generators $a, b$ of $\bE_8/T$ with $\ord(a)=\infty$, $\ord(b)=3$ such that the images of the roots of $\bE_8$  in $\Cl_0(Y)=\bE_8/T$ are
$$\pm ka \ (k=0,1,2,3, 12),\ \ \pm (ka-b) \ (k=2,3,4,5,6),\ \  \pm (ka-2b) \ (k=6,7,8,9,10).$$
The sets of generators $\{a,b\}$, $\{\al,\be\}$ of $\Cl_0(Y)$ are related by $b\in\{\pm \be\}$, $a\in \{\pm \al,\pm \al\pm \be\}$. 
The images of the roots of $\bE_8$ in $\Cl_0(Y)$, in terms of $\al, \be$, are
$$\pm k\al \ (k=0,1,2,3,4,5,7,8,10,12),\quad \pm k\al\pm \be\ (k=1,\ldots, 10).$$

We denote $d_i$ the effective divisor on $C$ such that $\cO(d_i)=\cO(D_i)_{\vert C}$. For every $i\neq6$, we have that $d_i\in C(\bQ)$.
It follows that in $\Pic^0(C)$ we have 
$$\rd(\al)=\cO_C(d_2-d_5),\quad \rd(\be)=\cO_C(d_3-d_5).$$
Using Computation \ref{asdcvq}, the points $d_2,d_3,d_5\in\Pic^0(C)(\bQ)$, using (\ref{minimalW parabolic}), are  
$d_2=(9,25)$, $d_3=(23,63)$, $d_5=(53,-387)$. 
Using Magma, we compute 
$$\rd(\al)=(-7,93),\quad \rd(\be)=(13,13),\quad \rd(2\be)=(-27,13),\quad \rd(3\be)=(13,-27),$$
and the order of $\rd(\be)$ in $\Pic^0(C)(\bQ)$ is $4$. 
As $C$ has class $3\be$ in $\Cl_0(Y)$, it follows that 
$\oEff(Y)$ is non-polyhedral (in some characteristic) if and only if none of 
$$\rd(k\al) \ (k=1,2,3,4,5,7,8,10,12), \quad \rd(k\al\pm \be) \ (k=1,\ldots, 10)$$
belong to $\{0,\rd(\be), \rd(2\be), \rd(3\be)\}$ of $\Pic^0(C)$, which is the subgroup 
generated by $\rd(\be)$ (from the above formulas, one can see that the order of $\rd(\be)$ is $4$ in characteristic $0$ or $p\neq 2,5$). 
Clearly, this is equivalent to
$\rd(k\be)$, for all $k=7,8,9,10,12$, not belonging to this subgroup. This is done within Computation \ref{badprimes}, which 
gives that this is the case for all primes $p\neq 2, 3, 5, 7, 11, 19, 71$. 
\end{proof}


\section{On the effective cone of $\oM_{0,n}$}\label{adsfhasrdharh}

For any toric variety $X$, we denote by $\Bl_eX$ the blow-up of $X$ at the identity element of the torus. 
Let $\oLM_n$ be the Losev--Manin moduli space \cite{LM}, which is also a toric variety.
Its curious feature,
noticed in \cite{CT_Duke},
is that $\oLM_n$ is ``universal'' among all projective toric varieties.
Moreover, $\Bl_e\oLM_n$ is universal among $\Bl_eX$.
Here we make this philosophical statement very precise:

\begin{theorem}\label{srgwrG}
Let $X$ be a projective toric variety. For any $n$ large enough (see the proof for an effective estimate),
there exists a sequence of projective toric varieties $\oLM_n=X_1$, $\ldots$, $X_s=X$ and rational maps induced by toric rational maps
$$\Bl_e\oLM_n=\Bl_eX_1\dra\Bl_eX_2\dra\ldots\dra\Bl_eX_s=\Bl_eX.$$
Every map $\Bl_eX_k\dra\Bl_eX_{k+1}$ decomposes as a small $\bQ$-factorial modification (SQM) $\Bl_eX_k\dra Z_k$ and a surjective morphism $Z_k\to \Bl_eX_{k+1}$.
If the cone $\barEff(\Bl_e\oLM_n)$ is (rational) polyhedral then $\barEff(\Bl_eX)$ is also (rational) polyhedral.
\end{theorem}

\begin{remark}
In \cite{CT_Duke} we used an analogous implication that if 
$\barEff(\Bl_e\oLM_n)$ is a Mori Dream Space then $\barEff(\Bl_eX)$ is a Mori Dream Space.
\end{remark}

The second statement in Thm. \ref{srgwrG}  follows from the first, using Lemma \ref{fvqevf} and the fact that if $Z\dra Z'$ is an SQM, then we can identify 
$\Num^1(Z)_{\bR}=\Num^1(Z')_{\bR}$ and $\barEff(Z)=\barEff(Z')$. 
The proof of the first statement in Thm. \ref{srgwrG} is based on the main technical result of \cite{CT_Duke}, which we give here in a slightly reformulated form:

\begin{lemma}[{\cite{CT_Duke}*{Prop. 3.1}}]\label{toric}
Let $\pi:\,N\to N'$ be a surjective map of lattices 
with kernel of rank $1$ spanned by a vector $v_0\in N$. 
Let $\Gamma$ be a finite set of rays in $N_{\bR}$ spanned by elements of $N$, 
which includes both rays $\pm{R_0}$ spanned by $\pm{v_0}$.
Let $\cF'\subset N'_\bR$ be a complete simplicial fan with rays given by $\pi(\Gamma)$
(ignore two zero vectors in the image). 
Suppose that the corresponding toric variety $X'$ is projective
(notice that it is also $\bQ$-factorial because $\cF'$ is simplicial). Then
there exists a complete simplicial fan $\cF\subset N_\bR$
with rays given by $\Gamma$ and such that
the corresponding toric variety $X$ is projective. Moreover, there exists a rational map
$\Bl_eX\dra\Bl_eX'$ which decomposes into an SQM $\Bl_eX\dra Z$ and a surjective morphism $Z\to \Bl_eX'$ (of relative dimension $1$).
\end{lemma}

\begin{corollary}\label{sasrgasr}
Let $\pi:\,N\to N'$ be a surjective map of lattices 
with kernel spanned by vectors $v_1,\ldots,v_s\in N$. 
Let $\Gamma$ be a finite set of rays in $N_{\bR}$ spanned by elements of $N$, 
which includes the rays $\pm{R_i}$ spanned by $\pm{v_i}$ for $i=1,\ldots,s$. 
Let $\cF'\subset N'_\bR$ be a complete simplicial fan with rays given by $\pi(\Gamma)$
(ignore zero vectors in the image). 
Suppose that the corresponding toric variety $X'$ is projective
(notice that it is also $\bQ$-factorial because $\cF'$ is simplicial). 
Then there exists a complete simplicial fan $\cF\subset N_\bR$
with rays  $\Gamma\cup\{\pm R_1\}\cup\ldots\cup\{\pm{R_s}\}$ and such that
the corresponding toric variety $X$ is projective. Moreover, 
there exists a sequence of toric varieties $X=X_1$, $\ldots$, $X_s=X'$ and rational maps induced by toric rational maps
$$\Bl_eX=\Bl_eX_1\dra\Bl_eX_2\dra\ldots\dra\Bl_eX_s=\Bl_eX'$$
such that every map $\Bl_eX_k\dra\Bl_eX_{k+1}$ decomposes as an SQM $\Bl_eX_k\dra Z_k$ and a surjective morphism $Z_k\to \Bl_eX_{k+1}$.
\end{corollary}

\begin{proof}
We argue by induction on $s$, the case $s=1$ is Lemma~\ref{toric}.
We can assume $v_1$ is a primitive vector.
Let $N''=N/\langle v_1\rangle$. We have a factorization of $\pi$ into $\pi_0:\,N\to N''$ and $\pi':\,N''\to N'$.
Let $\Gamma''$ be the image under $\pi_0$ of $\Gamma$ (ignore zero vectors in the image). 
Then we are in the situation of Lemma~\ref{toric}.
For the  map $\pi'$, we use the step of the induction.
\end{proof}

\begin{proof}[Proof of Theorem~\ref{srgwrG}]
We follow the same strategy as \cite{CT_Duke}. 

Applying  $\bQ$-factorialization,
we can assume that $X$ is a $\bQ$-factorial toric projective variety of dimension $r$.
The toric data of $\oLM_n$ is as follows. Fix general vectors $e_1,\ldots,e_{n-2}\in\bR^{n-3}$
such that $e_1+\ldots+e_{n-2}=0$. The lattice $N$ is generated by $e_1,\ldots,e_{n-2}$.
The rays of the fan of $\oLM_n$ are spanned by the primitive lattice vectors $\sum_{i\in I}e_i$, for each subset 
$I$ of $S:=\{1,\ldots, n-2\}$ with $1\le |I|\le n-3$. 
Notice that rays of this fan come in opposite pairs.
We are not going to need cones of higher dimension of this fan.
We partition 
$$S=S_1\coprod \ldots  \coprod S_{r+1}$$ 
into subsets of equal size $m\geq3$ (so that $n=m(r+1)+2$).
We also fix some indices $n_i\in S_i$, for $i=1,\ldots,r+1$. 
Let $N''\subset N$ be a sublattice spanned by the following vectors:
\begin{equation}\label{axfafgfgfg}
e_{n_i}+e_j\quad\hbox{\rm for}\quad j\in S_i\setminus\{n_i\},\ i=1,\ldots,r+1.
\end{equation}
Let $N'=N/N''$ be the quotient group and let $\pi$ be the projection map.
Then we have the following:
\begin{enumerate}
\item $N'$ is a lattice;
\item $N'$ is spanned by the vectors $\pi(e_{n_i})$, for $i=1,\ldots,r+1$;
\item $\pi(e_{n_1})+\ldots+\pi(e_{n_{r+1}})=0$ is the only linear relation between these vectors.
\end{enumerate}
It follows at once that the toric surface with lattice $N'$ and rays spanned by $\pi(e_{n_i})$ for $i=1,\ldots,r+1$,  is a projective space $\bP^r$. 
Choose a basis $f_1,\ldots,f_r$ for the lattice $N'$ so that $\pi(e_{n_1})=-f_1$, $\ldots$,  $\pi(e_{n_r})=-f_r$.
Fix one of the indices $1,\ldots,r+1$, we start with $r+1$.
Choose $e=\sum_{i\in I}e_i$ such that $n_1,\ldots,n_r\not\in I$, 
$|I\cap S_1|=k_1$, $\ldots$, $|I\cap S_r|=k_r$ and $|I|=k_1+\ldots+k_r$. 
Then $\pi(e)=k_1f_1+\ldots+k_rf_r$
and
$$\pi(e+e_{n_{r+1}})=(k_1+1)f_1+\ldots+(k_r+1)f_r.$$
It follows that images of the rays of $\oLM_n$ contain 
all points with non-zero coordinates bounded by $m$. Repeating this for all $r+1$ octants
shows that the images of the rays of $\oLM_n$ span all lattice points within the region illustrated in Figure~\ref{hgfhg}
for $r=2$, which contains all rays of $X$ if $m$ is large enough.
To be precise, for each $i\in\{1,\ldots,r\}$, in the octant spanned by 
$$f_1\,\ldots,f_{i-1},f_{i+1},\ldots, f_{r+1}\quad (f_{r+1}:=\pi(-e_{n_{r+1}})=-f_1-\ldots-f_r),$$
the region containing all the images of rays of $\oLM_n$ is determined by 
$$mf_1,\quad \ldots,\quad mf_{i-1},\quad mf_{i+1}\quad\ldots\quad mf_{r+1}=-mf_1-\ldots-mf_{r}.$$

It remains to notice (see \cite{OP}, \cite{CT_Duke}*{Prop. 3.1}) that there exists a $\bQ$-factorial projective toric variety $W$ with rays given by the images of the rays of $\oLM_n$
and that the toric birational rational map $W\dashrightarrow X$ is a composition of birational toric morphisms and toric SQMs.
Thus  we are done by Corollary~\ref{sasrgasr}. 
\end{proof}

\begin{corollary}\label{wfvwgv}
Let $Y$ be a projective toric surface with lattice $\bZ^2$ and with fan 
spanned by rays contained in the polygon with vertices
\[
(\pm m,\pm m),\quad (0,\pm m),\quad (\pm m,0),
\]
for some $m\geq 3$ (see Figure~\ref{hgfhg} for $m=4$).   
If $\barEff(\Bl_eY)$ is not (rational) polyhedral then $\barEff(\oM_{0,3m+2})$ is not (rational) polyhedral.
\end{corollary}

\begin{proof}
We argue by contradiction.
If $\barEff(\oM_{0,n})$ is (rational) polyhedral then the pseudo-effective cone
$\barEff(\Bl_e\oLM_n)$ is also (rational) polyhedral
by Lemma~\ref{fvqevf} and \cite{CT_Duke}*{Theorem 1.1}. In this case $\barEff(\Bl_eY)$
 is (rational) polyhedral by Theorem~\ref{srgwrG}
 (and effective estimates in its proof).
 \end{proof}

\begin{figure}[htbp]
\begin{tikzpicture}[scale=.5]

\tkzInit[xmin=-4.8, xmax=4.8, ymin = -4.8, ymax=4.8]\tkzGrid

  \draw[line width=.2mm ] (-5.2,0) -- (5.2,0);
  \draw[line width=.2mm]  (0,-5.2) -- (0,5.2);
  

 \tkzDefPoint(1,-1){Q1}
 \tkzDrawPoints[fill=black,color=black,size=5](Q1)
 \tkzDefPoint(-1,0){Q2}
 \tkzDrawPoints[fill=black,color=black,size=5](Q2)
 \tkzDefPoint(1,-2){Q3}
 \tkzDrawPoints[fill=black,color=black,size=5](Q3)
 \tkzDefPoint(0,0){Q4}
 \tkzDrawPoints[fill=black,color=black,size=5](Q4)
 \tkzDefPoint(-1,1){Q5}
 \tkzDrawPoints[fill=black,color=black,size=5](Q5)
 \tkzDefPoint(1,-3){Q6}
 \tkzDrawPoints[fill=black,color=black,size=5](Q6)
 \tkzDefPoint(0,1){Q7}
 \tkzDrawPoints[fill=black,color=black,size=5](Q7)
 \tkzDefPoint(-1,2){Q8}
 \tkzDrawPoints[fill=black,color=black,size=5](Q8)
 \tkzDefPoint(0,2){Q9}
 \tkzDrawPoints[fill=black,color=black,size=5](Q9)
 \tkzDefPoint(-1,3){Q10}
 \tkzDrawPoints[fill=black,color=black,size=5](Q10)
 \tkzDefPoint(0,3){Q11}
 \tkzDrawPoints[fill=black,color=black,size=5](Q11)
 \tkzDefPoint(0,4){Q12}
 \tkzDrawPoints[fill=black,color=black,size=5](Q12)
 \tkzDefPoint(0,-1){Q13}
 \tkzDrawPoints[fill=black,color=black,size=5](Q13)
 \tkzDefPoint(-2,0){Q14}
 \tkzDrawPoints[fill=black,color=black,size=5](Q14)
 \tkzDefPoint(0,-2){Q15}
 \tkzDrawPoints[fill=black,color=black,size=5](Q15)
 \tkzDefPoint(-2,1){Q16}
 \tkzDrawPoints[fill=black,color=black,size=5](Q16)
 \tkzDefPoint(1,0){Q17}
 \tkzDrawPoints[fill=black,color=black,size=5](Q17)
 \tkzDefPoint(0,-3){Q18}
 \tkzDrawPoints[fill=black,color=black,size=5](Q18)
 \tkzDefPoint(-2,2){Q19}
 \tkzDrawPoints[fill=black,color=black,size=5](Q19)
 \tkzDefPoint(1,1){Q20}
 \tkzDrawPoints[fill=black,color=black,size=5](Q20)
 \tkzDefPoint(0,-4){Q21}
 \tkzDrawPoints[fill=black,color=black,size=5](Q21)
 \tkzDefPoint(1,2){Q22}
 \tkzDrawPoints[fill=black,color=black,size=5](Q22)
 \tkzDefPoint(1,3){Q23}
 \tkzDrawPoints[fill=black,color=black,size=5](Q23)
 \tkzDefPoint(1,4){Q24}
 \tkzDrawPoints[fill=black,color=black,size=5](Q24)
 \tkzDefPoint(-3,0){Q25}
 \tkzDrawPoints[fill=black,color=black,size=5](Q25)
 \tkzDefPoint(2,0){Q26}
 \tkzDrawPoints[fill=black,color=black,size=5](Q26)
 \tkzDefPoint(-3,1){Q27}
 \tkzDrawPoints[fill=black,color=black,size=5](Q27)
 \tkzDefPoint(2,1){Q28}
 \tkzDrawPoints[fill=black,color=black,size=5](Q28)
 \tkzDefPoint(-4,-1){Q29}
 \tkzDrawPoints[fill=black,color=black,size=5](Q29)
 \tkzDefPoint(2,2){Q30}
 \tkzDrawPoints[fill=black,color=black,size=5](Q30)
 \tkzDefPoint(-4,-2){Q31}
 \tkzDrawPoints[fill=black,color=black,size=5](Q31)
 \tkzDefPoint(2,3){Q32}
 \tkzDrawPoints[fill=black,color=black,size=5](Q32)
 \tkzDefPoint(-4,-3){Q33}
 \tkzDrawPoints[fill=black,color=black,size=5](Q33)
 \tkzDefPoint(2,4){Q34}
 \tkzDrawPoints[fill=black,color=black,size=5](Q34)
 \tkzDefPoint(-4,-4){Q35}
 \tkzDrawPoints[fill=black,color=black,size=5](Q35)
 \tkzDefPoint(-4,0){Q36}
 \tkzDrawPoints[fill=black,color=black,size=5](Q36)
 \tkzDefPoint(3,0){Q37}
 \tkzDrawPoints[fill=black,color=black,size=5](Q37)
 \tkzDefPoint(3,1){Q38}
 \tkzDrawPoints[fill=black,color=black,size=5](Q38)
 \tkzDefPoint(-3,-1){Q39}
 \tkzDrawPoints[fill=black,color=black,size=5](Q39)
 \tkzDefPoint(3,2){Q40}
 \tkzDrawPoints[fill=black,color=black,size=5](Q40)
 \tkzDefPoint(-3,-2){Q41}
 \tkzDrawPoints[fill=black,color=black,size=5](Q41)
 \tkzDefPoint(3,3){Q42}
 \tkzDrawPoints[fill=black,color=black,size=5](Q42)
 \tkzDefPoint(-3,-3){Q43}
 \tkzDrawPoints[fill=black,color=black,size=5](Q43)
 \tkzDefPoint(3,4){Q44}
 \tkzDrawPoints[fill=black,color=black,size=5](Q44)
 \tkzDefPoint(-3,-4){Q45}
 \tkzDrawPoints[fill=black,color=black,size=5](Q45)
 \tkzDefPoint(4,0){Q46}
 \tkzDrawPoints[fill=black,color=black,size=5](Q46)
 \tkzDefPoint(4,1){Q47}
 \tkzDrawPoints[fill=black,color=black,size=5](Q47)
 \tkzDefPoint(-2,-1){Q48}
 \tkzDrawPoints[fill=black,color=black,size=5](Q48)
 \tkzDefPoint(4,2){Q49}
 \tkzDrawPoints[fill=black,color=black,size=5](Q49)
 \tkzDefPoint(-2,-2){Q50}
 \tkzDrawPoints[fill=black,color=black,size=5](Q50)
 \tkzDefPoint(4,3){Q51}
 \tkzDrawPoints[fill=black,color=black,size=5](Q51)
 \tkzDefPoint(-2,-3){Q52}
 \tkzDrawPoints[fill=black,color=black,size=5](Q52)
 \tkzDefPoint(4,4){Q53}
 \tkzDrawPoints[fill=black,color=black,size=5](Q53)
 \tkzDefPoint(-2,-4){Q54}
 \tkzDrawPoints[fill=black,color=black,size=5](Q54)
 \tkzDefPoint(-1,-1){Q55}
 \tkzDrawPoints[fill=black,color=black,size=5](Q55)
 \tkzDefPoint(-1,-2){Q56}
 \tkzDrawPoints[fill=black,color=black,size=5](Q56)
 \tkzDefPoint(-1,-3){Q57}
 \tkzDrawPoints[fill=black,color=black,size=5](Q57)
 \tkzDefPoint(-1,-4){Q58}
 \tkzDrawPoints[fill=black,color=black,size=5](Q58)
 \tkzDefPoint(3,-1){Q59}
 \tkzDrawPoints[fill=black,color=black,size=5](Q59)
 \tkzDefPoint(2,-1){Q60}
 \tkzDrawPoints[fill=black,color=black,size=5](Q60)
 \tkzDefPoint(2,-2){Q61}
 \tkzDrawPoints[fill=black,color=black,size=5](Q61)
 \node[right] at (4,4.5) {\tiny{$(m,m)$}};
 \node[left] at (-4,-4.5) {\tiny{$(-m,-m)$}};
\end{tikzpicture}

\caption{}\label{hgfhg}
\end{figure}

Variations in the choice of projections used in the proof of Thm. \ref{srgwrG} can lead to further variations and improvements, such as the following: 
\begin{corollary}\label{shape2}
Let $Y$ be a projective toric surface with lattice $\bZ^2$ and with fan 
spanned by rays contained in the polygon with vertices
\begin{equation}\label{polygon2}
(\pm l,\pm l),\quad (\pm 1,\mp l),\quad (\pm 1,\mp 1),\quad  (\pm l,\mp 1)
\end{equation}
for some $l\geq 2$ (see Figure~\ref{poly2} for $l=4$).   
If $\barEff(\Bl_eY)$ is not (rational) polyhedral then $\barEff(\oM_{0,2l+5})$ is not (rational) polyhedral.
\end{corollary}

\begin{figure}[htbp]

\begin{tikzpicture}[scale=.5]
\tkzInit[xmin=-4.5, xmax=4.5, ymin = -4.5, ymax=4.5]\tkzGrid

  \draw[line width=.2mm ] (-5,0) -- (5,0);
  \draw[line width=.2mm]  (0,-5) -- (0,5);

 \tkzDefPoint(1,-1){Q1}
 \tkzDrawPoints[fill=black,color=black,size=5](Q1)
 \tkzDefPoint(-1,0){Q2}
 \tkzDrawPoints[fill=black,color=black,size=5](Q2)
 \tkzDefPoint(1,-2){Q3}
 \tkzDrawPoints[fill=black,color=black,size=5](Q3)
 \tkzDefPoint(0,0){Q4}
 \tkzDrawPoints[fill=black,color=black,size=5](Q4)
 \tkzDefPoint(-1,1){Q5}
 \tkzDrawPoints[fill=black,color=black,size=5](Q5)
 \tkzDefPoint(1,-3){Q6}
 \tkzDrawPoints[fill=black,color=black,size=5](Q6)
 \tkzDefPoint(0,1){Q7}
 \tkzDrawPoints[fill=black,color=black,size=5](Q7)
 \tkzDefPoint(-1,2){Q8}
 \tkzDrawPoints[fill=black,color=black,size=5](Q8)
 \tkzDefPoint(1,-4){Q9}
 \tkzDrawPoints[fill=black,color=black,size=5](Q9)
 \tkzDefPoint(0,2){Q10}
 \tkzDrawPoints[fill=black,color=black,size=5](Q10)
 \tkzDefPoint(-1,3){Q11}
 \tkzDrawPoints[fill=black,color=black,size=5](Q11)
 \tkzDefPoint(0,3){Q12}
 \tkzDrawPoints[fill=black,color=black,size=5](Q12)
 \tkzDefPoint(-1,4){Q13}
 \tkzDrawPoints[fill=black,color=black,size=5](Q13)
 \tkzDefPoint(0,4){Q14}
 \tkzDrawPoints[fill=black,color=black,size=5](Q14)
 \tkzDefPoint(0,-1){Q15}
 \tkzDrawPoints[fill=black,color=black,size=5](Q15)
 \tkzDefPoint(-2,0){Q16}
 \tkzDrawPoints[fill=black,color=black,size=5](Q16)
 \tkzDefPoint(0,-2){Q17}
 \tkzDrawPoints[fill=black,color=black,size=5](Q17)
 \tkzDefPoint(-2,1){Q18}
 \tkzDrawPoints[fill=black,color=black,size=5](Q18)
 \tkzDefPoint(1,0){Q19}
 \tkzDrawPoints[fill=black,color=black,size=5](Q19)
 \tkzDefPoint(0,-3){Q20}
 \tkzDrawPoints[fill=black,color=black,size=5](Q20)
 \tkzDefPoint(1,1){Q22}
 \tkzDrawPoints[fill=black,color=black,size=5](Q22)
 \tkzDefPoint(0,-4){Q23}
 \tkzDrawPoints[fill=black,color=black,size=5](Q23)
 \tkzDefPoint(1,2){Q25}
 \tkzDrawPoints[fill=black,color=black,size=5](Q25)
 \tkzDefPoint(1,3){Q26}
 \tkzDrawPoints[fill=black,color=black,size=5](Q26)
 \tkzDefPoint(1,4){Q27}
 \tkzDrawPoints[fill=black,color=black,size=5](Q27)
 \tkzDefPoint(-3,0){Q28}
 \tkzDrawPoints[fill=black,color=black,size=5](Q28)
 \tkzDefPoint(-3,1){Q29}
 \tkzDrawPoints[fill=black,color=black,size=5](Q29)
 \tkzDefPoint(2,0){Q30}
 \tkzDrawPoints[fill=black,color=black,size=5](Q30)
 \tkzDefPoint(2,1){Q31}
 \tkzDrawPoints[fill=black,color=black,size=5](Q31)
 \tkzDefPoint(-4,-1){Q33}
 \tkzDrawPoints[fill=black,color=black,size=5](Q33)
 \tkzDefPoint(2,2){Q34}
 \tkzDrawPoints[fill=black,color=black,size=5](Q34)
 \tkzDefPoint(-4,-2){Q35}
 \tkzDrawPoints[fill=black,color=black,size=5](Q35)
 \tkzDefPoint(2,3){Q36}
 \tkzDrawPoints[fill=black,color=black,size=5](Q36)
 \tkzDefPoint(-4,-3){Q37}
 \tkzDrawPoints[fill=black,color=black,size=5](Q37)
 \tkzDefPoint(2,4){Q38}
 \tkzDrawPoints[fill=black,color=black,size=5](Q38)
 \tkzDefPoint(-4,-4){Q39}
 \tkzDrawPoints[fill=black,color=black,size=5](Q39)
 \tkzDefPoint(-4,0){Q40}
 \tkzDrawPoints[fill=black,color=black,size=5](Q40)
 \tkzDefPoint(-4,1){Q41}
 \tkzDrawPoints[fill=black,color=black,size=5](Q41)
 \tkzDefPoint(3,0){Q42}
 \tkzDrawPoints[fill=black,color=black,size=5](Q42)
 \tkzDefPoint(3,1){Q43}
 \tkzDrawPoints[fill=black,color=black,size=5](Q43)
 \tkzDefPoint(-3,-1){Q44}
 \tkzDrawPoints[fill=black,color=black,size=5](Q44)
 \tkzDefPoint(3,2){Q45}
 \tkzDrawPoints[fill=black,color=black,size=5](Q45)
 \tkzDefPoint(-3,-2){Q46}
 \tkzDrawPoints[fill=black,color=black,size=5](Q46)
 \tkzDefPoint(3,3){Q47}
 \tkzDrawPoints[fill=black,color=black,size=5](Q47)
 \tkzDefPoint(-3,-3){Q48}
 \tkzDrawPoints[fill=black,color=black,size=5](Q48)
 \tkzDefPoint(3,4){Q49}
 \tkzDrawPoints[fill=black,color=black,size=5](Q49)
 \tkzDefPoint(-3,-4){Q50}
 \tkzDrawPoints[fill=black,color=black,size=5](Q50)
 \tkzDefPoint(4,0){Q51}
 \tkzDrawPoints[fill=black,color=black,size=5](Q51)
 \tkzDefPoint(4,1){Q52}
 \tkzDrawPoints[fill=black,color=black,size=5](Q52)
 \tkzDefPoint(-2,-1){Q53}
 \tkzDrawPoints[fill=black,color=black,size=5](Q53)
 \tkzDefPoint(4,2){Q54}
 \tkzDrawPoints[fill=black,color=black,size=5](Q54)
 \tkzDefPoint(-2,-2){Q55}
 \tkzDrawPoints[fill=black,color=black,size=5](Q55)
 \tkzDefPoint(4,3){Q56}
 \tkzDrawPoints[fill=black,color=black,size=5](Q56)
 \tkzDefPoint(-2,-3){Q57}
 \tkzDrawPoints[fill=black,color=black,size=5](Q57)
 \tkzDefPoint(4,4){Q58}
 \tkzDrawPoints[fill=black,color=black,size=5](Q58)
 \tkzDefPoint(-2,-4){Q59}
 \tkzDrawPoints[fill=black,color=black,size=5](Q59)
 \tkzDefPoint(4,-1){Q60}
 \tkzDrawPoints[fill=black,color=black,size=5](Q60)
 \tkzDefPoint(-1,-1){Q61}
 \tkzDrawPoints[fill=black,color=black,size=5](Q61)
 \tkzDefPoint(-1,-2){Q62}
 \tkzDrawPoints[fill=black,color=black,size=5](Q62)
 \tkzDefPoint(-1,-3){Q63}
 \tkzDrawPoints[fill=black,color=black,size=5](Q63)
 \tkzDefPoint(-1,-4){Q64}
 \tkzDrawPoints[fill=black,color=black,size=5](Q64)
 \tkzDefPoint(3,-1){Q65}
 \tkzDrawPoints[fill=black,color=black,size=5](Q65)
 \tkzDefPoint(2,-1){Q67}
 \tkzDrawPoints[fill=black,color=black,size=5](Q67)
 \node[right] at (4,4.5) {\tiny{$(l,l)$}};
 \node[left] at (-4,-4.5) {\tiny{$(-l,-l)$}};
 \node[left] at (-1,4.5) {\tiny{$(-1,l)$}};
 \node[right] at (1,-4.5) {\tiny{$(1,-l)$}};
 \node[right] at (4,-1.5) {\tiny{$(l,-1)$}};
 \node[left] at (-4,1.5) {\tiny{$(-l,1)$}};
\end{tikzpicture}

\caption{}\label{poly2}
\end{figure}

\begin{proof}
Similarly, we argue by contradiction.
If $\barEff(\oM_{0,n})$ is (rational) polyhedral then the pseudo-effective cone
$\barEff(\Bl_e\oLM_n)$ is also (rational) polyhedral
by Lemma~\ref{fvqevf} and \cite{CT_Duke}*{Theorem 1.1}. In this case $\barEff(\Bl_eY)$
 is (rational) polyhedral using the same idea as in the proof of Theorem~\ref{srgwrG}. It suffices to prove 
 that one can project in such a way that the images of the rays of the fan of  $\oLM_n$ are contained in the polygon given by (\ref{polygon2}). 
 
The rays of the fan of $\oLM_n$ are spanned by the primitive lattice vectors $\sum_{i\in I}e_i$, for each subset 
$I$ of $S:=\{1,\ldots, n-2\}$ with $1\le |I|\le n-3$. 
We partition 
$$S=S_1\coprod S_2 \coprod S_3,\quad |S_1|=|S_2|=l+1,\quad |S_3|=1.$$ 

We fix some indices $n_i\in S_i$, for $i=1,2$ and let $S_3=\{n_3\}$.
 Let $N''\subset N$ be a sublattice spanned by the following vectors:
$$e_{n_i}+e_j\quad\hbox{\rm for}\quad j\in S_i\setminus\{n_i\},\ i=1,2.$$
Let $N'=N/N''$ be the quotient group and let $\pi$ be the projection map.
Then we have the following:
\begin{enumerate}
\item $N'$ is a lattice;
\item $N'$ is spanned by the vectors $\pi(e_{n_i})$, for $i=1,2,3$;
\item $-(l-1)\pi(e_{n_1})+-(l-1)\pi(e_{n_2})+\pi(e_{n_3})=0$ is the only linear relation between these vectors.
\end{enumerate} 
Choose a basis $f_1, f_2$ for the lattice $N'$ given by $\pi(e_{n_1})=f_1$, $\pi(e_{n_2})=f_2$. 
Then 
$$\pi(e_{n_3})=(l-1)f_1+(l-1)f_2.$$ 
We calculate the images $\pi(\sum_{i\in I} e_i)$ of the rays of the fan of $\oLM_n$. Consider the case when $n_1, n_2, n_3\notin I$. 
If $|I\cap S_1|=i$, $|I\cap S_2|=j$, then clearly the images of such rays are given by 
$-if_1-jf_2$ and all values $0\leq i,j\leq l$ are possible. This gives a square $P$ which in the given basis, has coordinates 
$$(-l,-l),\quad (-l,0),\quad (0,-l),\quad (0,0).$$ 
If $n_1\in I$, $n_2, n_3\notin I$, the images $\pi(\sum_{i\in I} e_i)$  will be contained in the translation of $P$ by $f_1=(1,0)$. 
Similarly,  if $n_3\notin I$, then $\pi(\sum_{i\in I} e_i)$ is contained in the union of $P$ with its translates
by $f_1=(1,0)$, $f_2=(0,1)$ and $f_1+f_2=(1,1)$, i.e., the square $Q$ with sides 
$$(-l,-l),\quad (-l,1),\quad (1,-l),\quad (1,1).$$ 
Finally, if $n_3\in I$, then $\pi(\sum_{i\in I} e_i)$ will be contained in the translate $Q'$ of $Q$ by $f_3=(l-1,l-1)$. Hence, all images 
of rays are contained in the sum of $Q$ and $Q'$, i.e., the polygon given in (\ref{polygon2}). 
 \end{proof}

%


\begin{corollary}\label{shape3}
Let $Y$ be a projective toric surface with lattice $\bZ^2$ and with fan 
spanned by rays contained in the polygon with vertices
\[
(\pm 3,\pm 1),\quad (\pm 3,\pm 5),\quad (\pm 2,\pm 6),\quad  (\pm 1,\pm 6),
\quad(\pm 1,\mp 3),
\]
(see Figure~\ref{poly3}).   
If $\barEff(\Bl_eY)$ is not (rational) polyhedral then $\barEff(\oM_{0,10})$ 
is not (rational) polyhedral.
\end{corollary}

\begin{proof}
It suffices to prove that $\barEff(\Bl_e\oLM_{10})$ is not 
(rational) polyhedral.
We do a variation of the method  in the proof of Thm. \ref{srgwrG},
projecting the lattice $\bZ^7$ of the Losev-Manin space $\oLM_{10}$ 
(spanned by $\{e_1,\ldots, e_8\}$ and subject to the relation 
$\sum_{i=1}^8 e_i=0$) from the following rays of the fan of  $\oLM_{10}$:
$e_1+e_2+e_4+e_6$, $e_1+e_2+e_5+e_7$,
$e_1+e_4+e_6+e_7$, $e_5+e_6$ and 
$e_1+e_5+e_8$.
These vectors generate the kernel of the map $\pi:\bZ^7\ra\bZ^2$
given by
\[
 \begin{pmatrix*}[r]
  1&0&1&-2&-1&1&0\\
  0&1&-1&-3&-2&2&1
 \end{pmatrix*}
\]
We conclude observing that the images of the rays of $\oLM_{10}$ via 
$f$ are the points of Figure~\ref{poly3}.
\end{proof}
\begin{figure}[htbp]

\begin{tikzpicture}[scale=.5]
\tkzInit[xmin=-3.5,xmax=3.5,ymin=-6.5,ymax=6.5]\tkzGrid
  \draw[line width=.2mm ] (-4,0) -- (4,0);
  \draw[line width=.2mm]  (0,-7) -- (0,7);

   \tkzDefPoint(-3,-5){P1}
   \tkzDefPoint(-3,-1){P2}
   \tkzDefPoint(-2,-6){P3}
   \tkzDefPoint(-1,-6){P4}
   \tkzDefPoint(-1,3){P5}
   \tkzDefPoint(1,-3){P6}
   \tkzDefPoint(1,6){P7}
   \tkzDefPoint(2,6){P8}
   \tkzDefPoint(3,1){P9}
   \tkzDefPoint(3,5){P10}
   \tkzDefPoint(1,-1){Q1}
   \tkzDrawPoints[fill=black,color=black,size=5](Q1)
   \tkzDefPoint(-1,0){Q2}
   \tkzDrawPoints[fill=black,color=black,size=5](Q2)
   \tkzDefPoint(1,-2){Q3}
   \tkzDrawPoints[fill=black,color=black,size=5](Q3)
   \tkzDefPoint(0,0){Q4}
   \tkzDrawPoints[fill=black,color=black,size=5](Q4)
   \tkzDefPoint(-1,1){Q5}
   \tkzDrawPoints[fill=black,color=black,size=5](Q5)
   \tkzDefPoint(1,-3){Q6}
   \tkzDrawPoints[fill=black,color=black,size=5](Q6)
   \tkzDefPoint(-1,2){Q7}
   \tkzDrawPoints[fill=black,color=black,size=5](Q7)
   \tkzDefPoint(0,1){Q8}
   \tkzDrawPoints[fill=black,color=black,size=5](Q8)
   \tkzDefPoint(-1,3){Q9}
   \tkzDrawPoints[fill=black,color=black,size=5](Q9)
   \tkzDefPoint(0,2){Q10}
   \tkzDrawPoints[fill=black,color=black,size=5](Q10)
   \tkzDefPoint(0,3){Q11}
   \tkzDrawPoints[fill=black,color=black,size=5](Q11)
   \tkzDefPoint(0,4){Q12}
   \tkzDrawPoints[fill=black,color=black,size=5](Q12)
   \tkzDefPoint(0,-1){Q13}
   \tkzDrawPoints[fill=black,color=black,size=5](Q13)
   \tkzDefPoint(-2,0){Q14}
   \tkzDrawPoints[fill=black,color=black,size=5](Q14)
   \tkzDefPoint(0,-2){Q15}
   \tkzDrawPoints[fill=black,color=black,size=5](Q15)
   \tkzDefPoint(-2,1){Q16}
   \tkzDrawPoints[fill=black,color=black,size=5](Q16)
   \tkzDefPoint(1,0){Q17}
   \tkzDrawPoints[fill=black,color=black,size=5](Q17)
   \tkzDefPoint(0,-3){Q18}
   \tkzDrawPoints[fill=black,color=black,size=5](Q18)
   \tkzDefPoint(1,1){Q19}
   \tkzDrawPoints[fill=black,color=black,size=5](Q19)
   \tkzDefPoint(0,-4){Q20}
   \tkzDrawPoints[fill=black,color=black,size=5](Q20)
   \tkzDefPoint(1,2){Q21}
   \tkzDrawPoints[fill=black,color=black,size=5](Q21)
   \tkzDefPoint(1,3){Q22}
   \tkzDrawPoints[fill=black,color=black,size=5](Q22)
   \tkzDefPoint(1,4){Q23}
   \tkzDrawPoints[fill=black,color=black,size=5](Q23)
   \tkzDefPoint(1,5){Q24}
   \tkzDrawPoints[fill=black,color=black,size=5](Q24)
   \tkzDefPoint(1,6){Q25}
   \tkzDrawPoints[fill=black,color=black,size=5](Q25)
   \tkzDefPoint(2,0){Q26}
   \tkzDrawPoints[fill=black,color=black,size=5](Q26)
   \tkzDefPoint(2,1){Q27}
   \tkzDrawPoints[fill=black,color=black,size=5](Q27)
   \tkzDefPoint(2,2){Q28}
   \tkzDrawPoints[fill=black,color=black,size=5](Q28)
   \tkzDefPoint(2,3){Q29}
   \tkzDrawPoints[fill=black,color=black,size=5](Q29)
   \tkzDefPoint(2,4){Q30}
   \tkzDrawPoints[fill=black,color=black,size=5](Q30)
   \tkzDefPoint(2,5){Q31}
   \tkzDrawPoints[fill=black,color=black,size=5](Q31)
   \tkzDefPoint(2,6){Q32}
   \tkzDrawPoints[fill=black,color=black,size=5](Q32)
   \tkzDefPoint(3,1){Q33}
   \tkzDrawPoints[fill=black,color=black,size=5](Q33)
   \tkzDefPoint(-3,-1){Q34}
   \tkzDrawPoints[fill=black,color=black,size=5](Q34)
   \tkzDefPoint(3,2){Q35}
   \tkzDrawPoints[fill=black,color=black,size=5](Q35)
   \tkzDefPoint(-3,-2){Q36}
   \tkzDrawPoints[fill=black,color=black,size=5](Q36)
   \tkzDefPoint(3,3){Q37}
   \tkzDrawPoints[fill=black,color=black,size=5](Q37)
   \tkzDefPoint(-3,-3){Q38}
   \tkzDrawPoints[fill=black,color=black,size=5](Q38)
   \tkzDefPoint(3,4){Q39}
   \tkzDrawPoints[fill=black,color=black,size=5](Q39)
   \tkzDefPoint(-3,-4){Q40}
   \tkzDrawPoints[fill=black,color=black,size=5](Q40)
   \tkzDefPoint(3,5){Q41}
   \tkzDrawPoints[fill=black,color=black,size=5](Q41)
   \tkzDefPoint(-3,-5){Q42}
   \tkzDrawPoints[fill=black,color=black,size=5](Q42)
   \tkzDefPoint(-2,-1){Q43}
   \tkzDrawPoints[fill=black,color=black,size=5](Q43)
   \tkzDefPoint(-2,-2){Q44}
   \tkzDrawPoints[fill=black,color=black,size=5](Q44)
   \tkzDefPoint(-2,-3){Q45}
   \tkzDrawPoints[fill=black,color=black,size=5](Q45)
   \tkzDefPoint(-2,-4){Q46}
   \tkzDrawPoints[fill=black,color=black,size=5](Q46)
   \tkzDefPoint(-2,-5){Q47}
   \tkzDrawPoints[fill=black,color=black,size=5](Q47)
   \tkzDefPoint(-2,-6){Q48}
   \tkzDrawPoints[fill=black,color=black,size=5](Q48)
   \tkzDefPoint(-1,-1){Q49}
   \tkzDrawPoints[fill=black,color=black,size=5](Q49)
   \tkzDefPoint(-1,-2){Q50}
   \tkzDrawPoints[fill=black,color=black,size=5](Q50)
   \tkzDefPoint(-1,-3){Q51}
   \tkzDrawPoints[fill=black,color=black,size=5](Q51)
   \tkzDefPoint(-1,-4){Q52}
   \tkzDrawPoints[fill=black,color=black,size=5](Q52)
   \tkzDefPoint(-1,-5){Q53}
   \tkzDrawPoints[fill=black,color=black,size=5](Q53)
   \tkzDefPoint(-1,-6){Q54}
   \tkzDrawPoints[fill=black,color=black,size=5](Q54)
   \tkzDefPoint(2,-1){Q55}
   \tkzDrawPoints[fill=black,color=black,size=5](Q55)

\end{tikzpicture}

\quad

\caption{}\label{poly3}
\end{figure}

\begin{proof}[Proof of Theorem~\ref{asgarh}]
If the characteristic is $0$ or 
any prime $p\neq 2,3,5,7,11,19,71$,
one can use the Halphen$^+$ polygon $\Delta$ from Theorem \ref{ex3}. 
Indeed, after the shear transformation $(x,y)\mapsto (x,x-y)$, 
the rays of the normal fan of $\Delta$ are
\[
(0,-1),\quad (-1,-6),\quad (-1,-3),\quad (-3,-2),\quad (-3,-1),\quad (3,5),\quad (2,3),
\]
which are among the points of Figure ~\ref{poly3}, so that
we can apply Corollary~\ref{shape3}.

In order to conclude we are going to produce, for any 
$p \in \{2, 3, 5, 7, 11, 19, 71\}$, a suitable good 
lattice polygon $\Delta$, whose normal fan has rays 
among the points of Figure~\ref{poly3}. In particular,
since the characteristic is positive, $\Delta$ is Halphen,
with $e := e(C,X_\Delta) < \infty$. The pencil $|eC|$ 
defines a fibration $\pi\colon X\to \mathbb P^1$. Let us
denote by $S_i := \pi^{-1}(q_i)$, for $i=1,\dots,\lambda$, 
the reducible fibers and by $\mu_i$
the number of irreducible components of $S_i$. 
It is not hard to see that any such irreducible 
component is defined over the field $\mathbb F_p$, so we only have to check a finite family. 
We then conclude showing that 
\[
\sum_{i=1}^{\lambda}(\mu_i-1) < {\rm Rank}({\rm Pic}(X)) - 2 
= \#{\rm Vertices}(\Delta) - 3,
\]
which, by Remark~\ref{rem:nonpol}, implies that 
the effective cone is not polyhedral.

In Computation~\ref{sporadic} we analize in detail the
case $p=2$, while in the following table we list, for any 
$p \in \{2, 3, 5, 7, 11, 19, 71\}$, the polygon $\Delta$,
the corresponding $e(C,X_\Delta)$, and the cardinality
of the reducible fibers.

\begin{center}
{
\scriptsize
\begin{longtable}{|c|c|c|c|}
\hline
$p$ & {\rm vertices}  & $e(C,X_\Delta)$ & $[\mu_1,\dots,\mu_\lambda]$\\
\hline
&&&\\
$2$ 
&
$
\left[
\begin{smallmatrix}
 0 & 6 & 9 & 10 & 9 & 3 & 1 \\
 0 & 2 & 4 & 5 & 6 & 10 & 4 
\end{smallmatrix}
\right]
$
&
$1$
&
$[2,3]$
\\
&&&\\
\hline
&&&\\
$3$ 
&
$
\left[
\begin{smallmatrix}
 0 & 5 & 7 & 12 & 13 & 14 & 12 & 6 & 2 \\
 0 & 2 & 3 & 6 & 8 & 11 & 12 & 14 & 6 
\end{smallmatrix}
\right]
$
&
$2$
&
$[2,5]$
\\
&&&\\
\hline
&&&\\
$5$ 
&
$
\left[
\begin{smallmatrix}
 0 & 2 & 12 & 13 & 13 & 12 & 11 & 9 & 4 \\
 0 & 1 & 7 & 8 & 9 & 11 & 12 & 13 & 12 
\end{smallmatrix}
\right]
$
&
$2$
&
$[3,4]$
\\
&&&\\
\hline
&&&\\
$7$ 
&
$
\left[
\begin{smallmatrix}
 0 & 1 & 4 & 8 & 10 & 4 & 3 & 1 \\
 0 & 0 & 1 & 4 & 7 & 10 & 8 & 3 
\end{smallmatrix}
\right]
$
&
$2$
&
$[2,4]$
\\
&&&\\
\hline
&&&\\
$11$ 
&
$
\left[
\begin{smallmatrix}
 0 & 12 & 13 & 13 & 12 & 11 & 9 & 7 & 1 \\
 0 & 4 & 8 & 9 & 11 & 12 & 13 & 12 & 2 
\end{smallmatrix}
\right]
$
&
$2$
&
$[3,3]$
\\
&&&\\
\hline
&&&\\
$19$ 
&
$
\left[
\begin{smallmatrix}
 0 & 2 & 8 & 9 & 3 & 2 \\
 0 & 1 & 5 & 6 & 10 & 8 
\end{smallmatrix}
\right]
$
&
$2$
&
$[2,2]$
\\
&&&\\
\hline
&&&\\
$71$ 
&
$
\left[
\begin{smallmatrix}
 0 & 3 & 8 & 12 & 13 & 12 & 11 & 5 & 4 \\
 0 & 1 & 4 & 7 & 8 & 10 & 11 & 13 & 12 
\end{smallmatrix}
\right]
$
&
$3$
&
$[2,2,4]$
\\
&&&\\
\hline
\end{longtable}
}
\end{center}


\end{proof}

\begin{remark}
\label{rem:LT}
By Computation~\ref{sdfvwefvwefv}, the rays of the normal 
fan of Polygon~111 are among the points
in Figure~\ref{poly3}. By Example~\ref{ex:exp}, $\Delta$ is a 
Lang--Trotter polygon, so that by Theorem~\ref{goodpairsthm}, 
we have another proof that 
$\overline{\Eff}(X_\Delta)$ is not polyhedral in characteristic~$0$. 
Moreover, in Database~\ref{adhadthstj}, we collect many more 
Lang--Trotter polygons such that
their normal fans (sometimes after a shear transformation)
fit into Figure~\ref{poly3}.
This shows that $\overline{\Eff}(\oM_{0,10})$
is not polyhedral in characteristic $p$ for $p<2000$.
We find that this is a strong indication that
one could also use Lang--Trotter polygons to prove that 
$\overline{\Eff}(\oM_{0,10})$ is not polyhedral in characteristic $p$, 
for all primes $p$.
\end{remark}


\section{Databases and Magma Computations}\label{asfgSrhasrh}
\begin{database}\label{vewfvefvef}
We give in Table~\ref{tab:good} the list of all Lang--Trotter polygons with $m\leq 7$.
It is obtained as follows.
We consider all lattice polygons
of volume up to $49$ (modulo equivalence)
appearing in the database~\cite{bel}.
We impose the conditions of Definition~\ref{def:good}
using our Magma package.
Computation~\ref{sdfvwefvwefv} gives (i) and (ii).
Computations~\ref{asarsgwRG} and \ref{adfafgarg} give (iii), (iv) and the equation of $\Gamma$.
This leaves $184$ lattice polygons and in all the cases 
the curve $C$ turns out to be 
smooth by Computation~\ref{adfafgarg}.
Furthermore, for all but one polygon in this list, 
we also have that the point $e$ is an ordinary 
multiple point of $\Gamma$.
The exceptional case is {\bf Polygon~23},
in which case the tangent cone to the 
curve $\Gamma$ at $e$ contains a double 
line. The curve $C$ turns out
to be tangent to the exceptional divisor
at the corresponding point,
so that also in this case $C$ is smooth.
Therefore, 
for any polygon in the list, $C$ 
is a smooth genus $1$ curve and
moreover, since $\Delta$ has at least $4$ vertices
and $|\partial\Delta\cap \bZ^2| = m \leq 7$, 
we also have that at least one
edge $F$ of $\Delta$ has lattice length~$1$.
By Proposition~\ref{prop:gen} we conclude that
the curve $C$ has a rational point $p_F$
that we can chose as the origin, so that
in what follows we can treat $C$ as an elliptic curve.
This fact allows to check the last condition of
the definition of a Lang--Trotter polygon, i.e., that 
$\mathcal O_X(C)|_C = \rd(C)$ is non-torsion.
Indeed, we can compute the minimal equation of the elliptic curve $C$
using Computation~\ref{asdcvq}.
We are then able to compute the
order $d$ of the torsion subgroup of the Mordell-Weil group
of the elliptic curve, and we have that
$\rd(C)$ is not torsion if and only if 
$\rd(dC)$ is non-trivial. 
By~Definition-Lemma~\ref{asfar}
this is equivalent to $h^0(X,dC) = 1$,
and the latter condition can be 
checked by Computation~\ref{asarsgwRG}.

Another approach is to find a multiple of $d$
using Nagell--Lutz Theorem \cite{ST}: if~$p$
is a prime of good reduction for $C$, then 
the specialization map induces an injective 
homomorphism of abelian groups
$C(\bQ)_{\rm tors}\to C(\bF_p)$.
Therefore, the torsion order $d$ of $C(\bQ)$
divides the order of $C(\bF_p)$ 
for any prime $p$ of good reduction, which is easy to compute from
the defining equation of $\Gamma$.
We then find a multiple of  $d$
by taking the greatest common divisor of the
orders of $C(\bF_p)$ as $p$ varies.
{
\scriptsize
\begin{longtable}{c|c|c|c|c|c|c|}
\cline{2-7}
& & & & & &\\
$5$
&
$
\left[
\begin{smallmatrix}
 0 & 4 & 2 & 1 & 5 \\
 0 & 4 & 5 & 5 & 3 
\end{smallmatrix}
\right]
$
&
$
\left[
\begin{smallmatrix}
 4 & 0 & 2 & 5 & 3 \\
 2 & 4 & 1 & 5 & 0 
\end{smallmatrix}
\right]
$
&
$
\left[
\begin{smallmatrix}
 4 & 0 & 5 & 1 & 2 \\
 2 & 1 & 0 & 4 & 5 
\end{smallmatrix}
\right]
$
&
$
\left[
\begin{smallmatrix}
 4 & 0 & 0 & 5 & 1 \\
 2 & 4 & 5 & 3 & 0 
\end{smallmatrix}
\right]
$
&
$
\left[
\begin{smallmatrix}
 3 & 1 & 0 & 5 & 0\\
 1 & 5 & 1 & 2 & 0
\end{smallmatrix}
\right]
$
&
$
\left[
\begin{smallmatrix}
 4 & 0 & 0 & 5 & 1 \\
 1 & 4 & 5 & 2 & 0 
\end{smallmatrix}
\right]
$
\\
& & & & & & \\
\cline{2-7}
\multicolumn{7}{c}{\ }\\
\cline{2-7}
& & & & & & \\
$6$
&
$
\left[
\begin{smallmatrix}
 1 & 0 & 6 & 5 & 6 \\
 3 & 4 & 2 & 6 & 0 
\end{smallmatrix}
\right]
$
&
$
\left[
\begin{smallmatrix}
 2 & 1 & 0 & 5 & 6 \\
 0 & 0 & 6 & 1 & 2 
\end{smallmatrix}
\right]
$
&
$
\left[
\begin{smallmatrix}
 6 & 0 & 2 & 1 & 0 \\
 4 & 2 & 1 & 6 & 0 
\end{smallmatrix}
\right]
$
&
$
\left[
\begin{smallmatrix}
 5 & 1 & 6 & 6 & 0 \\
 6 & 5 & 1 & 0 & 4 
\end{smallmatrix}
\right]
$
&
$
\left[
\begin{smallmatrix}
 2 & 5 & 0 & 7 & 5 \\
 5 & 1 & 4 & 0 & 6
\end{smallmatrix}
\right]
$
&
$
\left[
\begin{smallmatrix}
 0 & 6 & 3 & 6 & 4 \\
 6 & 4 & 1 & 5 & 0 
\end{smallmatrix}
\right]
$
\\
& & & & & & \\
\cline{2-7} & & & & & & \\
& 
$
\left[
\begin{smallmatrix}
 6 & 3 & 1 & 0 & 2 \\
 1 & 5 & 0 & 0 & 6 
\end{smallmatrix}
\right]
$
&
$
\left[
\begin{smallmatrix}
 4 & 2 & 6 & 6 & 5 & 0 \\
 5 & 1 & 5 & 6 & 0 & 2 
\end{smallmatrix}
\right]
$
&
$
\left[
\begin{smallmatrix}
 3 & 5 & 1 & 4 & 0 & 6 \\
 5 & 2 & 2 & 0 & 3 & 6 
\end{smallmatrix}
\right]
$
&
$
\left[
\begin{smallmatrix}
 2 & 3 & 1 & 6 & 0 & 0 \\
 5 & 0 & 1 & 2 & 5 & 6 
\end{smallmatrix}
\right]
$
&
$
\left[
\begin{smallmatrix}
 3 & 4 & 1 & 0 & 6 & 0 \\
 6 & 1 & 0 & 1 & 2 & 0 
\end{smallmatrix}
\right]
$
&
$
\left[
\begin{smallmatrix}
 3 & 2 & 5 & 0 & 0 & 6 \\
 5 & 0 & 2 & 5 & 6 & 3 
\end{smallmatrix}
\right]
$
\\
 & & & & & & \\
\cline{2-7} & & & & & & \\
 & 
$
\left[
\begin{smallmatrix}
1 & 6 & 5 & 0 & 3 & 4\\
2 & 1 & 4 & 0 & 5 & 6
\end{smallmatrix}
\right]
$
&
$
\left[
\begin{smallmatrix}
 4 & 1 & 0 & 2 & 6 & 3 \\
 1 & 6 & 6 & 1 & 4 & 0 
\end{smallmatrix}
\right]
$
&
$
\left[
\begin{smallmatrix}
 1 & 2 & 5 & 6 & 2 & 0 \\
 5 & 6 & 0 & 1 & 1 & 2 
\end{smallmatrix}
\right]
$
&
$
\left[
\begin{smallmatrix}
 5 & 0 & 5 & 3 & 6 & 6 \\
 3 & 4 & 6 & 0 & 5 & 6 
\end{smallmatrix}
\right]
$
&
$
\left[
\begin{smallmatrix}
 5 & 0 & 2 & 5 & 6 & 3 \\
 3 & 2 & 5 & 0 & 1 & 6 
\end{smallmatrix}
\right]
$
&
$
\left[
\begin{smallmatrix}
 1 & 2 & 5 & 0 & 6 & 0 \\
 6 & 1 & 3 & 1 & 4 & 0 
\end{smallmatrix}
\right]
$
\\
 & & & & & & \\
\cline{2-7} & & & & & & \\
 & 
$
\left[
\begin{smallmatrix}
 3 & 1 & 2 & 5 & 0 & 6 \\
 0 & 1 & 5 & 3 & 6 & 2 
\end{smallmatrix}
\right]
$
&
$
\left[
\begin{smallmatrix}
 2 & 3 & 5 & 2 & 6 & 0 \\
 1 & 0 & 6 & 5 & 5 & 4 
\end{smallmatrix}
\right]
$
&
$
\left[
\begin{smallmatrix}
 4 & 0 & 0 & 5 & 1 & 6 \\
 1 & 5 & 6 & 3 & 0 & 2 
\end{smallmatrix}
\right]
$
&
$
\left[
\begin{smallmatrix}
 1 & 6 & 5 & 6 & 0 & 4 \\
 2 & 5 & 6 & 6 & 3 & 0 
\end{smallmatrix}
\right]
$
&
&
\\
& & & & & & \\
\cline{2-7}
\multicolumn{7}{c}{\ }\\
\cline{2-7}
& & & & & & \\
$7$
&
$
\left[
\begin{smallmatrix}
 4 & 6 & 2 & 7 & 0 \\
 7 & 4 & 8 & 5 & 0 
\end{smallmatrix}
\right]
$
&
$
\left[
\begin{smallmatrix}
 1 & 1 & 7 & 0 & 3 \\
 3 & 0 & 1 & 0 & 7 
\end{smallmatrix}
\right]
$
&
$
\left[
\begin{smallmatrix}
 { 4} & { 1} & { 7} & { 0} & { 5} \\
 { 2} & { 6} & { 0} & { 5} & { 8} 
\end{smallmatrix}
\right]
$
&
$
\left[
\begin{smallmatrix}
 7 & 0 & 1 & 2 & 0 \\
 4 & 3 & 6 & 7 & 0 
\end{smallmatrix}
\right]
$
&
$
\left[
\begin{smallmatrix}
 6 & 5 & 1 & 0 & 7 \\
 6 & 7 & 3 & 4 & 0 
\end{smallmatrix}
\right]
$
&
$
\left[
\begin{smallmatrix}
 5 & 6 & 7 & 2 & 0 \\
 1 & 7 & 0 & 3 & 5 
\end{smallmatrix}
\right]
$
\\
 & & & & & & \\
\cline{2-7} & & & & & & \\
 & 
$
\left[
\begin{smallmatrix}
 1 & 0 & 3 & 7 & 5 \\
 0 & 0 & 5 & 1 & 7 
\end{smallmatrix}
\right]
$
&
$
\left[
\begin{smallmatrix}
 6 & 7 & 0 & 2 & 0 \\
 7 & 7 & 4 & 0 & 6 
\end{smallmatrix}
\right]
$
&
$
\left[
\begin{smallmatrix}
 1 & 4 & 2 & 7 & 0 \\
 2 & 7 & 0 & 7 & 6 
\end{smallmatrix}
\right]
$
&
$
\left[
\begin{smallmatrix}
 0 & 0 & 5 & 1 & 7 \\
 2 & 0 & 2 & 7 & 3 
\end{smallmatrix}
\right]
$
&
$
\left[
\begin{smallmatrix}
 7 & 2 & 6 & 0 & 4 & 5 \\
 0 & 6 & 5 & 4 & 7 & 7 
\end{smallmatrix}
\right]
$
&
$
\left[
\begin{smallmatrix}
 0 & 7 & 6 & 4 & 7 & 5 \\
 3 & 5 & 2 & 0 & 7 & 0 
\end{smallmatrix}
\right]
$
\\
 & & & & & & \\
\cline{2-7} & & & & & & \\
 & 
$
\left[
\begin{smallmatrix}
 4 & 7 & 2 & 6 & 7 & 0 \\
 5 & 6 & 0 & 1 & 7 & 2 
\end{smallmatrix}
\right]
$
&
$
\left[
\begin{smallmatrix}
 2 & 1 & 0 & 5 & 6 & 7 \\
 7 & 7 & 0 & 6 & 4 & 5 
\end{smallmatrix}
\right]
$
&
$
\left[
\begin{smallmatrix}
 6 & 7 & 5 & 2 & 0 & 0 \\
 3 & 1 & 4 & 0 & 6 & 7 
\end{smallmatrix}
\right]
$
&
$
\left[
\begin{smallmatrix}
 0 & 1 & 7 & 5 & 2 & 3 \\
 5 & 2 & 7 & 3 & 0 & 0 
\end{smallmatrix}
\right]
$
&
$
\left[
\begin{smallmatrix}
 5 & 7 & 7 & 6 & 1 & 0 \\
 6 & 6 & 7 & 0 & 3 & 2 
\end{smallmatrix}
\right]
$
&
$
\left[
\begin{smallmatrix}
 7 & 7 & 2 & 2 & 0 & 0 \\
 5 & 4 & 1 & 7 & 1 & 0 
\end{smallmatrix}
\right]
$
\\
 & & & & & & \\
\cline{2-7} & & & & & & \\
 & 
$
\left[
\begin{smallmatrix}
 4 & 1 & 7 & 1 & 0 & 0 \\
 0 & 4 & 5 & 7 & 6 & 7 
\end{smallmatrix}
\right]
$
&
$
\left[
\begin{smallmatrix}
 4 & 5 & 6 & 7 & 1 & 0 \\
 7 & 7 & 0 & 0 & 5 & 4 
\end{smallmatrix}
\right]
$
&
$
\left[
\begin{smallmatrix}
 0 & 1 & 7 & 5 & 3 & 4 \\
 7 & 4 & 5 & 1 & 1 & 0 
\end{smallmatrix}
\right]
$
&
$
\left[
\begin{smallmatrix}
 2 & 4 & 6 & 0 & 6 & 7 \\
 7 & 0 & 3 & 7 & 6 & 5 
\end{smallmatrix}
\right]
$
&
$
\left[
\begin{smallmatrix}
 6 & 1 & 7 & 2 & 0 & 3\\
 2 & 5 & 0 & 1 & 2 & 7
\end{smallmatrix}
\right]
$
&
$
\left[
\begin{smallmatrix}
 7 & 7 & 6 & 2 & 0 & 2 \\
 6 & 7 & 3 & 5 & 4 & 0 
\end{smallmatrix}
\right]
$
\\
 & & & & & & \\
\cline{2-7} & & & & & & \\
 & 
$
\left[
\begin{smallmatrix}
 5 & 1 & 7 & 4 & 0 & 2 \\
 3 & 4 & 6 & 7 & 0 & 7 
\end{smallmatrix}
\right]
$
&
$
\left[
\begin{smallmatrix}
 5 & 7 & 6 & 5 & 1 & 0 \\
 6 & 7 & 1 & 0 & 2 & 3 
\end{smallmatrix}
\right]
$
&
$
\left[
\begin{smallmatrix}
 7 & 1 & 0 & 2 & 2 & 0 \\
 7 & 0 & 3 & 6 & 0 & 5 
\end{smallmatrix}
\right]
$
&
$
\left[
\begin{smallmatrix}
 6 & 7 & 1 & 5 & 1 & 0 \\
 0 & 0 & 5 & 7 & 3 & 4 
\end{smallmatrix}
\right]
$
&
$
\left[
\begin{smallmatrix}
 5 & 1 & 0 & 3 & 5 & 7 \\
 2 & 7 & 7 & 0 & 6 & 5 
\end{smallmatrix}
\right]
$
&
$
\left[
\begin{smallmatrix}
 1 & 5 & 4 & 7 & 0 & 2 \\
 4 & 2 & 6 & 4 & 0 & 7 
\end{smallmatrix}
\right]
$
\\
 & & & & & & \\
\cline{2-7} & & & & & & \\
 & 
$
\left[
\begin{smallmatrix}
 3 & 1 & 6 & 0 & 5 & 7 \\
 0 & 6 & 5 & 5 & 7 & 7 
\end{smallmatrix}
\right]
$
&
$
\left[
\begin{smallmatrix}
 2 & 1 & 6 & 0 & 7 & 0 \\
 0 & 0 & 3 & 6 & 2 & 7 
\end{smallmatrix}
\right]
$
&
$
\left[
\begin{smallmatrix}
 7 & 0 & 0 & 7 & 4 & 5 \\
 5 & 1 & 2 & 7 & 0 & 0 
\end{smallmatrix}
\right]
$
&
$
\left[
\begin{smallmatrix}
 1 & 5 & 2 & 7 & 3 & 0 \\
 3 & 5 & 0 & 3 & 6 & 7 
\end{smallmatrix}
\right]
$
&
$
\left[
\begin{smallmatrix}
 0 & 0 & 7 & 6 & 7 & 5 \\
 1 & 0 & 4 & 6 & 2 & 7 
\end{smallmatrix}
\right]
$
&
$
\left[
\begin{smallmatrix}
 7 & 6 & 1 & 0 & 4 & 5 \\
 1 & 5 & 0 & 0 & 6 & 7 
\end{smallmatrix}
\right]
$
\\
 & & & & & & \\
\cline{2-7} & & & & & & \\
 & 
$
\left[
\begin{smallmatrix}
 6 & 2 & 5 & 0 & 6 & 7 \\
 1 & 7 & 0 & 7 & 6 & 5 
\end{smallmatrix}
\right]
$
&
$
\left[
\begin{smallmatrix}
 0 & 0 & 1 & 7 & 5 & 4 \\
 1 & 0 & 3 & 2 & 6 & 7 
\end{smallmatrix}
\right]
$
&
$
\left[
\begin{smallmatrix}
 4 & 3 & 3 & 0 & 5 & 7 \\
 6 & 7 & 1 & 3 & 0 & 2 
\end{smallmatrix}
\right]
$
&
$
\left[
\begin{smallmatrix}
 3 & 4 & 7 & 1 & 1 & 0 \\
 6 & 7 & 2 & 3 & 0 & 0 
\end{smallmatrix}
\right]
$
&
$
\left[
\begin{smallmatrix}
 6 & 7 & 0 & 1 & 3 & 2 \\
 2 & 0 & 3 & 1 & 6 & 7 
\end{smallmatrix}
\right]
$
&
$
\left[
\begin{smallmatrix}
 6 & 5 & 0 & 7 & 6 & 7 \\
 5 & 7 & 4 & 1 & 0 & 0 
\end{smallmatrix}
\right]
$
\\
 & & & & & & \\
\cline{2-7} & & & & & & \\
 & 
$
\left[
\begin{smallmatrix}
 6 & 7 & 2 & 4 & 1 & 0 \\
 6 & 5 & 3 & 0 & 7 & 7 
\end{smallmatrix}
\right]
$
&
$
\left[
\begin{smallmatrix}
 6 & 7 & 2 & 1 & 0 & 4 \\
 4 & 2 & 0 & 0 & 1 & 7 
\end{smallmatrix}
\right]
$
&
$
\left[
\begin{smallmatrix}
 6 & 7 & 4 & 0 & 1 & 2 \\
 4 & 3 & 1 & 6 & 7 & 0 
\end{smallmatrix}
\right]
$
&
$
\left[
\begin{smallmatrix}
 7 & 0 & 6 & 0 & 4 & 5 \\
 7 & 2 & 3 & 3 & 0 & 0 
\end{smallmatrix}
\right]
$
&
$
\left[
\begin{smallmatrix}
 4 & 2 & 7 & 6 & 7 & 0 \\
 2 & 6 & 2 & 7 & 0 & 5 
\end{smallmatrix}
\right]
$
&
$
\left[
\begin{smallmatrix}
 1 & 0 & 7 & 7 & 5 & 4 \\
 2 & 0 & 3 & 2 & 7 & 7 
\end{smallmatrix}
\right]
$
\\
 & & & & & & \\
\cline{2-7} & & & & & & \\
 & 
$
\left[
\begin{smallmatrix}
 3 & 0 & 5 & 2 & 3 & 7 \\
 6 & 5 & 3 & 0 & 0 & 7 
\end{smallmatrix}
\right]
$
&
$
\left[
\begin{smallmatrix}
 5 & 2 & 0 & 7 & 3 & 4 \\
 2 & 5 & 1 & 7 & 0 & 0 
\end{smallmatrix}
\right]
$
&
$
\left[
\begin{smallmatrix}
 2 & 3 & 6 & 7 & 7 & 0 \\
 5 & 7 & 2 & 4 & 3 & 0 
\end{smallmatrix}
\right]
$
&
$
\left[
\begin{smallmatrix}
 7 & 1 & 7 & 0 & 4 & 6 \\
 5 & 1 & 7 & 2 & 0 & 1 
\end{smallmatrix}
\right]
$
&
$
\left[
\begin{smallmatrix}
 2 & 0 & 7 & 6 & 7 & 5 \\
 6 & 5 & 2 & 0 & 3 & 7 
\end{smallmatrix}
\right]
$
&
$
\left[
\begin{smallmatrix}
 1 & 0 & 6 & 7 & 5 & 3 \\
 3 & 5 & 4 & 7 & 2 & 0 
\end{smallmatrix}
\right]
$
\\
 & & & & & & \\
\cline{2-7} & & & & & & \\
 & 
$
\left[
\begin{smallmatrix}
 0 & 2 & 7 & 5 & 3 & 4 \\
 4 & 1 & 7 & 1 & 0 & 0 
\end{smallmatrix}
\right]
$
&
$
\left[
\begin{smallmatrix}
 0 & 6 & 5 & 0 & 3 & 7 \\
 4 & 5 & 7 & 5 & 0 & 7 
\end{smallmatrix}
\right]
$
&
$
\left[
\begin{smallmatrix}
 3 & 2 & 4 & 5 & 0 & 7 \\
 6 & 7 & 0 & 0 & 2 & 1 
\end{smallmatrix}
\right]
$
&
$
\left[
\begin{smallmatrix}
 0 & 1 & 6 & 3 & 7 & 2 \\
 5 & 7 & 3 & 0 & 2 & 0 
\end{smallmatrix}
\right]
$
&
$
\left[
\begin{smallmatrix}
 4 & 7 & 6 & 1 & 0 & 5 \\
 1 & 0 & 4 & 5 & 4 & 7 
\end{smallmatrix}
\right]
$
&
$
\left[
\begin{smallmatrix}
 7 & 7 & 1 & 0 & 3 & 6 \\
 2 & 0 & 6 & 5 & 7 & 6 
\end{smallmatrix}
\right]
$
\\
 & & & & & & \\
\cline{2-7} & & & & & & \\
 & 
$
\left[
\begin{smallmatrix}
 5 & 7 & 1 & 4 & 3 & 0 \\
 6 & 5 & 4 & 1 & 0 & 7 
\end{smallmatrix}
\right]
$
&
$
\left[
\begin{smallmatrix}
 4 & 3 & 0 & 7 & 7 & 5 \\
 6 & 1 & 3 & 6 & 7 & 0 
\end{smallmatrix}
\right]
$
&
$
\left[
\begin{smallmatrix}
 3 & 2 & 5 & 0 & 7 & 7 \\
 0 & 0 & 4 & 7 & 1 & 2 
\end{smallmatrix}
\right]
$
&
$
\left[
\begin{smallmatrix}
 6 & 4 & 5 & 1 & 0 & 7 \\
 2 & 0 & 0 & 3 & 1 & 7 
\end{smallmatrix}
\right]
$
&
$
\left[
\begin{smallmatrix}
 1 & 2 & 3 & 7 & 0 & 0 \\
 5 & 7 & 1 & 3 & 1 & 0 
\end{smallmatrix}
\right]
$
&
$
\left[
\begin{smallmatrix}
 3 & 2 & 1 & 0 & 7 & 0 \\
 5 & 0 & 0 & 5 & 2 & 7 
\end{smallmatrix}
\right]
$
\\
 & & & & & & \\
\cline{2-7} & & & & & & \\
 & 
$
\left[
\begin{smallmatrix}
 2 & 3 & 5 & 7 & 3 & 0 \\
 5 & 7 & 2 & 6 & 1 & 0 
\end{smallmatrix}
\right]
$
&
$
\left[
\begin{smallmatrix}
 1 & 0 & 5 & 4 & 6 & 7 \\
 0 & 0 & 6 & 7 & 1 & 2 
\end{smallmatrix}
\right]
$
&
$
\left[
\begin{smallmatrix}
 5 & 2 & 0 & 7 & 0 & 1 \\
 3 & 7 & 2 & 5 & 1 & 0 
\end{smallmatrix}
\right]
$
&
$
\left[
\begin{smallmatrix}
 4 & 3 & 0 & 6 & 7 & 7 \\
 6 & 7 & 2 & 0 & 1 & 0 
\end{smallmatrix}
\right]
$
&
$
\left[
\begin{smallmatrix}
 0 & 5 & 0 & 7 & 2 & 1 \\
 2 & 6 & 0 & 5 & 7 & 7 
\end{smallmatrix}
\right]
$
&
$
\left[
\begin{smallmatrix}
 7 & 3 & 1 & 0 & 6 & 5 \\
 2 & 5 & 2 & 0 & 7 & 7 
\end{smallmatrix}
\right]
$
\\
 & & & & & & \\
\cline{2-7} & & & & & & \\
 & 
$
\left[
\begin{smallmatrix}
 3 & 4 & 3 & 1 & 0 & 0 & 7 \\
 5 & 0 & 0 & 1 & 6 & 7 & 2 
\end{smallmatrix}
\right]
$
&
$
\left[
\begin{smallmatrix}
 6 & 4 & 2 & 5 & 0 & 5 & 7 \\
 5 & 6 & 0 & 7 & 1 & 1 & 2 
\end{smallmatrix}
\right]
$
&
$
\left[
\begin{smallmatrix}
 6 & 1 & 7 & 0 & 0 & 3 & 4 \\
 4 & 3 & 2 & 1 & 0 & 6 & 7 
\end{smallmatrix}
\right]
$
&
$
\left[
\begin{smallmatrix}
 1 & 7 & 2 & 3 & 2 & 0 & 0 \\
 5 & 4 & 1 & 7 & 7 & 1 & 0 
\end{smallmatrix}
\right]
$
&
$
\left[
\begin{smallmatrix}
 1 & 4 & 6 & 0 & 7 & 4 & 7 \\
 5 & 1 & 4 & 4 & 1 & 7 & 0 
\end{smallmatrix}
\right]
$
&
$
\left[
\begin{smallmatrix}
 5 & 1 & 4 & 4 & 1 & 7 & 0 \\
 1 & 4 & 6 & 0 & 7 & 4 & 6 
\end{smallmatrix}
\right]
$
\\
 & & & & & & \\
\cline{2-7} & & & & & & \\
 & 
$
\left[
\begin{smallmatrix}
 1 & 6 & 0 & 5 & 7 & 7 & 4 \\
 1 & 3 & 3 & 6 & 6 & 7 & 0 
\end{smallmatrix}
\right]
$
&
$
\left[
\begin{smallmatrix}
 6 & 5 & 1 & 0 & 4 & 6 & 7 \\
 4 & 7 & 3 & 4 & 1 & 0 & 0 
\end{smallmatrix}
\right]
$
&
$
\left[
\begin{smallmatrix}
 2 & 7 & 1 & 6 & 0 & 6 & 5 \\
 1 & 6 & 3 & 7 & 2 & 2 & 0 
\end{smallmatrix}
\right]
$
&
$
\left[
\begin{smallmatrix}
 3 & 0 & 2 & 0 & 5 & 1 & 7 \\
 2 & 1 & 7 & 0 & 6 & 7 & 5 
\end{smallmatrix}
\right]
$
&
$
{\color{blue}
\left[
\begin{smallmatrix}
 {\bf 6} & {\bf 5} & {\bf 1} & {\bf 8} & {\bf 0} & {\bf 0} & {\bf 3} \\
 {\bf 1} & {\bf 4} & {\bf 3} & {\bf 2} & {\bf 6} & {\bf 7} & {\bf 0} 
\end{smallmatrix}
\right]
}
$
&
$
\left[
\begin{smallmatrix}
 1 & 4 & 5 & 4 & 7 & 0 & 0 \\
 6 & 1 & 5 & 7 & 0 & 5 & 4 
\end{smallmatrix}
\right]
$
\\
 & & & & & & \\
\cline{2-7} & & & & & & \\
 & 
$
\left[
\begin{smallmatrix}
 0 & 0 & 6 & 4 & 7 & 3 & 1 \\
 1 & 0 & 6 & 7 & 5 & 7 & 6 
\end{smallmatrix}
\right]
$
&
$
\left[
\begin{smallmatrix}
 1 & 3 & 6 & 2 & 7 & 3 & 0 \\
 2 & 6 & 5 & 0 & 7 & 0 & 5 
\end{smallmatrix}
\right]
$
&
$
\left[
\begin{smallmatrix}
 2 & 6 & 7 & 6 & 1 & 0 & 5 \\
 1 & 7 & 7 & 2 & 3 & 2 & 0 
\end{smallmatrix}
\right]
$
&
$
\left[
\begin{smallmatrix}
 5 & 4 & 4 & 1 & 0 & 7 & 6 \\
 1 & 0 & 6 & 4 & 3 & 6 & 7 
\end{smallmatrix}
\right]
$
&
$
\left[
\begin{smallmatrix}
 5 & 1 & 4 & 4 & 2 & 7 & 0 \\
 2 & 1 & 6 & 0 & 5 & 7 & 2 
\end{smallmatrix}
\right]
$
&
$
\left[
\begin{smallmatrix}
 1 & 7 & 2 & 3 & 3 & 0 & 0 \\
 1 & 7 & 0 & 6 & 0 & 3 & 4 
\end{smallmatrix}
\right]
$
\\
 & & & & & & \\
\cline{2-7} & & & & & & \\
 & 
$
\left[
\begin{smallmatrix}
 5 & 6 & 1 & 7 & 0 & 2 & 0 \\
 2 & 4 & 5 & 3 & 1 & 7 & 0 
\end{smallmatrix}
\right]
$
&
$
\left[
\begin{smallmatrix}
 3 & 0 & 4 & 6 & 3 & 7 & 7 \\
 1 & 3 & 6 & 3 & 7 & 1 & 0 
\end{smallmatrix}
\right]
$
&
$
\left[
\begin{smallmatrix}
 4 & 5 & 5 & 7 & 1 & 0 & 0 \\
 6 & 7 & 1 & 2 & 0 & 1 & 0 
\end{smallmatrix}
\right]
$
&
$
\left[
\begin{smallmatrix}
 1 & 1 & 0 & 6 & 7 & 6 & 4 \\
 2 & 0 & 0 & 6 & 4 & 2 & 7 
\end{smallmatrix}
\right]
$
&
$
\left[
\begin{smallmatrix}
 2 & 6 & 7 & 4 & 1 & 3 & 0 \\
 5 & 3 & 5 & 7 & 0 & 7 & 0 
\end{smallmatrix}
\right]
$
&
$
\left[
\begin{smallmatrix}
 6 & 0 & 0 & 1 & 6 & 3 & 7 \\
 2 & 5 & 4 & 2 & 0 & 7 & 0 
\end{smallmatrix}
\right]
$
\\
 & & & & & & \\
\cline{2-7} & & & & & & \\
 & 
$
\left[
\begin{smallmatrix}
 3 & 3 & 6 & 0 & 6 & 7 & 7 \\
 1 & 7 & 3 & 3 & 0 & 1 & 0 
\end{smallmatrix}
\right]
$
&
$
\left[
\begin{smallmatrix}
 0 & 0 & 6 & 7 & 7 & 4 & 2 \\
 4 & 3 & 3 & 1 & 0 & 6 & 7 
\end{smallmatrix}
\right]
$
&
$
\left[
\begin{smallmatrix}
 4 & 6 & 1 & 5 & 0 & 4 & 7 \\
 6 & 1 & 5 & 0 & 7 & 0 & 4 
\end{smallmatrix}
\right]
$
&
$
\left[
\begin{smallmatrix}
 1 & 6 & 6 & 7 & 1 & 0 & 3 \\
 3 & 5 & 7 & 7 & 6 & 5 & 0 
\end{smallmatrix}
\right]
$
&
$
\left[
\begin{smallmatrix}
 6 & 5 & 1 & 5 & 0 & 7 & 7 \\
 0 & 0 & 7 & 6 & 7 & 4 & 5 
\end{smallmatrix}
\right]
$
&
$
\left[
\begin{smallmatrix}
 2 & 7 & 6 & 7 & 0 & 6 & 5 \\
 3 & 3 & 1 & 2 & 0 & 6 & 7 
\end{smallmatrix}
\right]
$
\\
 & & & & & & \\
\cline{2-7} & & & & & & \\
 & 
$
\left[
\begin{smallmatrix}
 5 & 6 & 1 & 4 & 7 & 5 & 0 \\
 6 & 2 & 2 & 0 & 7 & 0 & 3 
\end{smallmatrix}
\right]
$
&
$
\left[
\begin{smallmatrix}
 1 & 0 & 6 & 7 & 6 & 4 & 5 \\
 6 & 5 & 0 & 0 & 4 & 7 & 7 
\end{smallmatrix}
\right]
$
&
$
\left[
\begin{smallmatrix}
 6 & 2 & 5 & 1 & 0 & 5 & 7 \\
 3 & 4 & 0 & 3 & 1 & 6 & 7 
\end{smallmatrix}
\right]
$
&
$
\left[
\begin{smallmatrix}
 6 & 4 & 4 & 5 & 0 & 0 & 7 \\
 2 & 5 & 0 & 0 & 1 & 2 & 7 
\end{smallmatrix}
\right]
$
&
$
\left[
\begin{smallmatrix}
 5 & 6 & 4 & 1 & 0 & 4 & 7 \\
 6 & 4 & 7 & 3 & 4 & 1 & 0 
\end{smallmatrix}
\right]
$
&
\\
& & & & & & \\
\cline{2-7}
\multicolumn{7}{c}{\ }\\
\caption{List of Lang--Trotter polygons for $m\leq 7$}
\label{tab:good}
\end{longtable}
}

\end{database}

\begin{database}\label{adhadthstj}
A database of Lang--Trotter polygons
that can be used to 
show that the pseudo-effective 
cone of $\oM_{0,n}$ is not polyhedral for $n\geq 10$
in characteristic $p$ for any prime $p<2000$
(see Remark~\ref{rem:LT}).
For each polygon, the
corresponding non-polyhedral primes
are displayed.

\begin{center}
{
\scriptsize
\begin{longtable}{|c|c|}
\hline
{\rm vertices}  & {\rm non-polyhedral primes} \\
\hline
&\\
$
\left[
\begin{smallmatrix}
 0 & 6 & 9 & 10 & 9 & 3 & 1 \\
 0 & 2 & 4 & 5 & 6 & 10 & 4  
\end{smallmatrix}
\right]
$
&
$\begin{smallmatrix}
2, 19, 29, 31, 53, 71, 83, 97, 103, 131, 167, 211, 233, 257, 263, 269, 277, 313, 347, 373, 419, 439,\\ 
461, 487, 491, 577, 593, 619, 643, 653, 661, 709, 761, 827, 907, 919, 941, 953, 991, 1013, 1061,\\ 
1097, 1123, 1213, 1223, 1231, 1249, 1289, 1367, 1451, 1481, 1483, 1499, 1543, 1549, 1583, 1721,\\ 
1723, 1741, 1787, 1871, 1873
\end{smallmatrix}
$
\\
&\\
\hline
&\\
$
\left[
\begin{smallmatrix}
 10 & 10 & 9 & 6 & 3 & 0 & 2 & 7 \\
 4 & 3 & 1 & 0 & 1 & 10 & 9 & 6 
\end{smallmatrix}
\right]
$
&
$\begin{smallmatrix}
7, 11, 13, 53, 59, 71, 107, 109, 127, 149, 157, 167, 173, 179, 181, 263, 271, 277, 283, 293, 337, 419,\\ 
421, 443, 449, 463, 487, 593, 601, 619, 643, 653, 677, 727, 751, 757, 761, 773, 797, 857, 859, 877,\\ 
887, 911, 929, 937, 997, 1019, 1031, 1049, 1061, 1069, 1087, 1091, 1103, 1163, 1231, 1249, 1291,\\ 
1301, 1319, 1373, 1427, 1439, 1447, 1451, 1459, 1489, 1493, 1523, 1553, 1559, 1571, 1609, 1613,\\ 
1669, 1721, 1741, 1747, 1777, 1787, 1811, 1871, 1889, 1901, 1933, 1973, 1987, 1993, 1997
\end{smallmatrix}
$
\\
&\\
\hline
&\\
$
\left[
\begin{smallmatrix}
 5 & 6 & 9 & 11 & 12 & 2 & 0 & 2 & 3 \\
 0 & 0 & 1 & 2 & 4 & 10 & 11 & 5 & 3 
\end{smallmatrix}
\right]
$
&
$\begin{smallmatrix}
23, 29, 41, 59, 67, 71, 131, 139, 179, 181, 191, 199, 223, 229, 241, 251, 307, 311, 331, 337, 349,\\ 
379, 401, 409, 419, 421, 443, 461, 491, 547, 571, 577, 587, 601, 631, 647, 661, 673, 701, 733, 739,\\ 
751, 787, 827, 839, 857, 859, 911, 919, 937, 971, 977, 983, 991, 1013, 1019, 1021, 1039, 1061, 1063,\\ 
1087, 1109, 1123, 1129, 1171, 1187, 1213, 1223, 1229, 1237, 1249, 1259, 1277, 1279, 1307, 1327,\\ 
1381, 1409, 1429, 1447, 1459, 1493, 1511, 1549, 1571, 1579, 1583, 1597, 1619, 1621, 1699, 1723,\\ 
1741, 1759, 1811, 1823, 1831, 1847, 1873, 1913, 1931, 1933, 1979, 1987
\end{smallmatrix}
$
\\
&\\
\hline
&\\
$
\left[
\begin{smallmatrix}
 0 & 5 & 9 & 10 & 12 & 11 & 5 & 4 \\
 12 & 10 & 7 & 6 & 3 & 2 & 0 & 0 
\end{smallmatrix}
\right]
$
&
 $\begin{smallmatrix}
 23, 31, 37, 41, 47, 53, 73, 101, 131, 139, 197, 199, 223, 233, 307, 317, 331, 383, 389, 401, 421, 439,\\ 
 449, 461, 479, 487, 499, 509, 569, 571, 593, 599, 607, 631, 641, 673, 701, 709, 743, 787, 811, 829,\\ 
 857, 863, 877, 881, 907, 911, 941, 1019, 1021, 1123, 1151, 1153, 1171, 1217, 1231, 1237, 1259, 1291,\\ 
 1297, 1423, 1429, 1481, 1583, 1609, 1657, 1723, 1753, 1783, 1823, 1871, 1879, 1889, 1901, 1907,\\ 
 1973, 1979, 1987, 1997
 \end{smallmatrix}
$
\\
&\\
\hline
&\\
$
\left[
\begin{smallmatrix}
 0 & 2 & 12 & 13 & 13 & 11 & 9 & 4 \\
 0 & 1 & 7 & 9 & 10 & 12 & 13 & 12 
\end{smallmatrix}
\right]
$
&
 $\begin{smallmatrix}
31, 37, 47, 79, 131, 139, 151, 181, 211, 223, 239, 257, 271, 281, 307, 331, 373, 389, 409, 433, 457,\\ 
461, 479, 523, 569, 577, 587, 641, 659, 683, 709, 719, 733, 743, 761, 769, 809, 821, 823, 853, 859,\\ 
863, 887, 953, 997, 1013, 1063, 1093, 1103, 1117, 1129, 1153, 1163, 1181, 1201, 1237, 1249, 1283,\\ 
1361, 1367, 1439, 1471, 1531, 1553, 1601, 1609, 1699, 1721, 1741, 1789, 1867, 1871, 1873, 1889,\\ 
1907, 1931, 1973, 1979, 1997
\end{smallmatrix}
$
\\
&\\
\hline
&\\
$
\left[
\begin{smallmatrix}
 0 & 2 & 12 & 13 & 12 & 11 & 8 & 7 & 4 \\
 0 & 1 & 7 & 9 & 11 & 12 & 13 & 13 & 12 
\end{smallmatrix}
\right]
$
&
$
\begin{smallmatrix}
31, 61, 71, 89, 97, 109, 127, 139, 149, 163, 173, 191, 193, 227, 233, 257, 271, 281, 311, 313, 347,\\ 
349, 353, 389, 421, 433, 457, 463, 467, 479, 491, 499, 541, 563, 571, 587, 607, 613, 631, 643, 683,\\ 
733, 743, 751, 757, 769, 797, 809, 821, 853, 857, 863, 907, 941, 967, 971, 991, 997, 1013, 1019,\\ 
1031, 1049, 1051, 1063, 1087, 1091, 1093, 1097, 1109, 1153, 1163, 1193, 1217, 1279, 1283, 1303,\\ 
1321, 1433, 1439, 1451, 1481, 1483, 1493, 1499, 1511, 1543, 1559, 1571, 1597, 1621, 1667, 1693,\\ 
1723, 1759, 1823, 1867, 1913, 1931, 1973, 1979, 1987
\end{smallmatrix}
$
\\
&\\
\hline
&\\
$
\left[
\begin{smallmatrix}
 13 & 9 & 5 & 4 & 2 & 1 & 0 & 1 & 11 \\
 8 & 0 & 3 & 4 & 7 & 9 & 12 & 13 & 9 
\end{smallmatrix}
\right]
$
&
$
\begin{smallmatrix}
11, 19, 59, 83, 101, 107, 113, 163, 167, 181, 197, 269, 293, 307, 313, 317, 337, 347, 349, 359, 373,\\ 
401, 461, 491, 499, 509, 521, 569, 617, 643, 647, 661, 677, 683, 739, 787, 797, 809, 821, 827, 829,\\ 
839, 859, 883, 887, 941, 983, 1087, 1109, 1117, 1163, 1213, 1237, 1277, 1283, 1291, 1303, 1307,\\ 
1429, 1451, 1483, 1493, 1553, 1597, 1621, 1637, 1667, 1733, 1801, 1901, 1933, 1993, 1997
\end{smallmatrix}
$
\\
&\\
\hline
&\\
$
\left[
\begin{smallmatrix}
 0 & 1 & 10 & 12 & 13 & 12 & 10 & 7 & 1 \\
 0 & 0 & 3 & 4 & 6 & 9 & 13 & 12 & 2 
\end{smallmatrix}
\right]
$
&
$
\begin{smallmatrix}
11, 23, 29, 31, 43, 59, 67, 73, 137, 149, 157, 223, 229, 271, 277, 281, 283, 293, 353, 367, 439, 457,\\ 
461, 491, 503, 577, 599, 601, 641, 643, 647, 653, 661, 691, 733, 757, 941, 977, 997, 1019, 1049, 1051,\\ 
1061, 1069, 1193, 1249, 1301, 1303, 1327, 1373, 1451, 1471, 1487, 1543, 1553, 1559, 1579, 1597,\\ 
1607, 1627, 1669, 1699, 1723, 1753, 1777, 1789, 1831, 1847, 1877, 1913, 1933, 1949, 1997, 1999
\end{smallmatrix}
$
\\
&\\
\hline
&\\
$
\left[
\begin{smallmatrix}
 0 & 12 & 13 & 13 & 11 & 9 & 7 & 1 \\
 0 & 4 & 9 & 10 & 12 & 13 & 12 & 2 
\end{smallmatrix}
\right]
$
&
$
\begin{smallmatrix}
7, 11, 67, 101, 139, 199, 251, 313, 331, 337, 353, 373, 383, 419, 421, 431, 503, 541, 557, 571, 587,\\ 
601, 607, 617, 619, 659, 709, 719, 733, 751, 857, 877, 883, 911, 947, 967, 1033, 1093, 1123, 1163,\\ 
1193, 1277, 1279, 1283, 1289, 1303, 1319, 1327, 1381, 1409, 1423, 1429, 1439, 1453, 1459, 1499,\\ 
1531, 1549, 1621, 1657, 1663, 1667, 1787, 1879, 1913, 1951
\end{smallmatrix}
$
\\
&\\
\hline
&\\
$
\left[
\begin{smallmatrix}
 0 & 2 & 12 & 13 & 13 & 12 & 11 & 9 & 4 \\
 0 & 1 & 7 & 8 & 9 & 11 & 12 & 13 & 12 
\end{smallmatrix}
\right]
$
&
$
\begin{smallmatrix}
5, 17, 23, 29, 41, 43, 67, 73, 79, 101, 103, 107, 113, 157, 173, 179, 191, 193, 227, 229, 239, 251,\\ 
263, 277, 281, 283, 313, 331, 337, 349, 353, 367, 379, 389, 397, 443, 449, 457, 463, 467, 479, 487,\\ 
503, 509, 521, 557, 563, 587, 617, 641, 643, 647, 653, 659, 701, 773, 787, 809, 823, 859, 887, 907,\\ 
911, 937, 941, 947, 983, 991, 1009, 1013, 1019, 1039, 1049, 1087, 1091, 1097, 1103, 1187, 1217,\\ 
1279, 1289, 1303, 1307, 1321, 1327, 1373, 1399, 1409, 1427, 1429, 1453, 1471, 1483, 1487, 1493,\\ 
1511, 1523, 1553, 1579, 1619, 1621, 1663, 1667, 1693, 1697, 1709, 1721, 1723, 1733, 1759, 1777,\\ 
1831, 1867, 1871, 1873, 1877, 1889, 1907, 1931, 1951, 1973, 1993, 1997
\end{smallmatrix}
$
\\
&\\
\hline
&\\
$
\left[
\begin{smallmatrix}
 0 & 5 & 10 & 12 & 14 & 14 & 5 & 4 \\
 0 & 2 & 5 & 7 & 10 & 11 & 14 & 12 
\end{smallmatrix}
\right]
$
&
$
\begin{smallmatrix}
31, 37, 47, 79, 103, 127, 137, 149, 151, 163, 199, 211, 223, 229, 257, 269, 271, 311, 347, 353, 359,\\ 
389, 397, 401, 419, 439, 443, 457, 461, 463, 487, 499, 503, 523, 569, 571, 631, 677, 701, 727, 751,\\ 
773, 823, 853, 883, 911, 919, 947, 953, 967, 991, 1019, 1039, 1063, 1097, 1123, 1151, 1153, 1171,\\ 
1193, 1201, 1217, 1223, 1231, 1279, 1283, 1289, 1303, 1307, 1327, 1373, 1423, 1447, 1453, 1471,\\ 
1499, 1511, 1523, 1543, 1567, 1571, 1607, 1693, 1699, 1723, 1733, 1753, 1759, 1777, 1783, 1801,\\ 
1831, 1861, 1877, 1879, 1889, 1913, 1951, 1987, 1999
\end{smallmatrix}
$
\\
&\\
\hline
&\\
$
\left[
\begin{smallmatrix}
 0 & 5 & 7 & 12 & 13 & 14 & 12 & 6 & 2 \\
 0 & 2 & 3 & 6 & 8 & 11 & 12 & 14 & 6 
\end{smallmatrix}
\right]
$
&
$
\begin{smallmatrix}
3, 17, 19, 61, 67, 127, 197, 223, 241, 251, 263, 271, 277, 307, 359, 367, 431, 463, 487, 563, 641,\\ 
659, 701, 719, 733, 751, 761, 797, 823, 829, 839, 877, 887, 911, 967, 977, 1031, 1049, 1093, 1123,\\ 
1153, 1163, 1223, 1249, 1277, 1321, 1327, 1433, 1447, 1453, 1481, 1571, 1613, 1627, 1663, 1709,\\ 
1733, 1759, 1787, 1801, 1847, 1901, 1997
\end{smallmatrix}
$
\\
&\\
\hline
\caption{}\label{tab-M010}
\end{longtable}
}
\end{center}

\end{database}
%
%
%


We give an overview of the MAGMA package, which can be downloaded from:
\begin{center}
\url{https://github.com/alaface/non-polyhedral}
\end{center}
and contains  descriptions of all functions.
We first use Polygon 111 as a running example, then we
study infinite families of pentagons and heptagons
from Section~\ref{asrgasrharh}. After that we find
non-polyhedral primes up to $2000$ for the polygon of
Section~\ref{sgasrhasrh}, and finally we study Halphen 
polygons, in particular the one in Example \ref{ex3}.

\begin{computation}\label{sdfvwefvwefv}
Normal fan of the lattice polygon $\Delta$, the fan of the minimal resolution of the toric surface $\bP_\Delta$, $\vol(\Delta)$, number of boundary points.
\begin{tbox}
{\footnotesize
\begin{verbatim}
> pol := Polytope([[6,1],[5,4],[1,3],[8,2],[0,6],[0,7],[3,0]]);
  Transpose(Matrix(Reorder(Rays(NormalFan(pol)))));
[ 3 -1 -1 -2 -3  1  3]
[ 2  3  2 -3 -5  0  1]
> Transpose(Matrix(Reorder(Rays(Resolution(NormalFan(pol))))));
[ 3  1  0 -1 -1 -1 -1 -1 -2 -3 -1  0  1  3  2]
[ 2  1  1  3  2  1  0 -1 -3 -5 -2 -1  0  1  1]
> [Volume(pol),#BoundaryPoints(pol)];
[ 49, 7 ]
\end{verbatim}
}
\end{tbox}
\end{computation}

\begin{computation}\label{asarsgwRG}
Dimension of linear systems $\cL_\Delta(m)$ 
and $\cL_{k\Delta}(km)$
(over different fields),
equation $f$ of $\Gamma\subset\bG_m^2$, Newton polytope of $f$.\\
\begin{tbox}
{\footnotesize
\begin{verbatim}
> m := Width(pol);
  #FindCurves(pol,m,Rationals());
1
> #FindCurves(2*pol,2*m,GF(5));
2
> f := FindCurves(pol,m,Rationals())[1];
 Transpose(Matrix(Vertices(NPolytope(f))));
[8 6 5 3 1 0 0]
[2 1 4 0 3 7 6]
\end{verbatim}
}
\end{tbox}
\end{computation}

\begin{computation}\label{adfafgarg}
Irreducibility and geometric genus of $\Gamma$.
\begin{tbox}
{\footnotesize
\begin{verbatim}
> IsIrreducible(FindCurve(pol,m,Rationals()));       
true
> Genus(FindCurve(pol,m,Rationals()));
1
\end{verbatim}
}
\end{tbox}
\end{computation}

\begin{computation}\label{sfsvwefv}
In the minimal resolution $\tilde X$
of $X$, a divisor linearly equivalent to the pullback of $C$
together with the prime components of the pullback of $K_X+C$, their multiplicities, Newton polygons and equations. 
\begin{tbox}
{\footnotesize
\begin{verbatim}
> AdjSys(pol);
[
    [ 19, 7, 2, 1, 0, 0, 0, 0, 0, 1, 2, 5, 8, 20, 13, -7 ],
    [ 3, 1, 0, 0, 0, 0, 0, 0, 0, 0, 0, 0, 1, 3, 2, -1 ],
    [ 1, 0, 0, 0, 0, 0, 0, 0, 1, 2, 1, 1, 1, 2, 1, -1 ],
    [ 8, 3, 1, 1, 0, 0, 0, 0, 0, 1, 1, 2, 3, 8, 5, -3 ]
]
> MultAdjSys(pol); 
[ 2, 1, 1 ]
> PolsAdjSys(pol);
[
    x[1] - 1,
    x[1] - x[2],
    x[1]^3*x[2] - 3*x[1]^2*x[2] - x[1]*x[2]^2 + 5*x[1]*x[2] - x[1] + x[2]^3 - 
        2*x[2]^2
]
\end{verbatim}
}
\end{tbox}
\end{computation}

\begin{computation}\label{efvwefvwef}
Root lattice of $\Delta$, the map $\Cl(X)\to\Cl(Y)$, 
intersection matrices of $X$ and $Y$ (the latter is not 
with respect to a basis).
\begin{tbox}
{\footnotesize
\begin{verbatim}
> RootLat(pol);
A6 A1
> Cl,g := MapToY(pol);
  Cl;
Full Quotient RSpace of degree 3 over Integer Ring
Column moduli:
[ 0, 0, 0 ]
> imatX(pol);
[-10/33   1/11      0      0      0      0    1/3      0]
[  1/11  -8/11      1      0      0      0      0      0]
[     0      1   -9/7    1/7      0      0      0      0]
[     0      0    1/7  -11/7      1      0      0      0]
[     0      0      0      1   -3/5    1/5      0      0]
[     0      0      0      0    1/5  -12/5      1      0]
[   1/3      0      0      0      0      1   -2/3      0]
[     0      0      0      0      0      0      0     -1]
> imatY(pol);
[ 17/14    8/7  13/14  25/14    4/7  31/14    1/2 103/14]
[   8/7    3/7    9/7    6/7    5/7   15/7      0   39/7]
[ 13/14    9/7   5/14  29/14    1/7  27/14    1/2  87/14]
[ 25/14    6/7  29/14  17/14   10/7  39/14    1/2 135/14]
[   4/7    5/7    1/7   10/7   -1/7   11/7      0   23/7]
[ 31/14   15/7  27/14  39/14   11/7  45/14    3/2 201/14]
[   1/2      0    1/2    1/2      0    3/2   -1/2    3/2]
[103/14   39/7  87/14 135/14   23/7 201/14    3/2 573/14]
\end{verbatim}
}
\end{tbox}
\end{computation}

\begin{computation}\label{asdcvq}
Minimal equation of $C$ and images of intersection points
with the toric boundary divisors
using the standard MAGMA algorithm, $\rd(C)$ and images of roots
in $\Pic^0(C)$ (identified with $C$),
polyhedrality of specific primes. The algorithm constructs a birational map
$u: C\dashrightarrow E$, with $E$ given by a minimal  Weierstrass equation in $\bP^2$. 
We consider only examples where  $E$ is smooth and the map $u$ is defined everywhere in characteristic $0$. 
Since $C$ has arithmetic genus $1$, it follows that $C$ is smooth and the map $u$ is an isomorphism. 
Similarly, for specific primes $p$, we discard those primes for which $E$ is not smooth, or for which the map $u$ is not defined 
everywhere. 

\begin{tbox}
{\footnotesize
\begin{verbatim}
> E,u := EllCur(pol);
  E;
Elliptic Curve defined by y^2 + x*y = x^3 - x^2 - 4*x + 4 over Rational Field
> Cl,g := MapToY(pol);
  C := FindCurve(pol,Width(pol),Rationals());
  A := Ambient(C);
  f := Equation(C);
  h := map<A->Ambient(E) | [Evaluate(p,[A.1,A.2,1]) : 
  p in DefiningEquations(u)]>;
  ff := [i : i in [1..#Vertices(pol)] | Volume(OrdFacets(pol)[i]) eq 1];
  pts := [E!PtsCur(h,f,u,pol,i) : i in ff];
  B := resC(pol,E,ff,pts);
  pts;
[ (6 : 10 : 1), (-3/16 : -133/64 : 1), (0 : 2 : 1), (496 : -11286 : 1),
(1 : -1 : 1), (32/49 : -510/343 : 1), (16/9 : -14/27 : 1) ]
> res := resC(pol,E,ff,pts);
  res;
[ (16/9 : -14/27 : 1), (1 : -1 : 1), (2 : 0 : 1) ]
> roots := FindRoots(pol);
  ImgRoots := [&+[Eltseq(v)[i]*B[i] : i in [1..#B]] : v in roots];
  ImgRoots;
[ (-2 : 2 : 1), (0 : 2 : 1), (85/49 : -244/343 : 1), (-19/16 : -119/64 : 1), 
(394519648/356869881 : -4479186863510/6741628921971 : 1), (4303/4489 : 
-323950/300763 : 1), (0 : -2 : 1), (-19/16 : 195/64 : 1) ]
> C := g(CinS(pol));
  ImgC := &+[Eltseq(C)[i]*B[i] : i in [1..#B]];
  ImgC;
(-1 : -2 : 1)
> NonPolyhedralPrimes(pol,2000);
{ 47, 71, 103, 197, 233, 239, 277, 313, 367, 379, 409, 503, 563, 599, 647, 
677, 683, 691, 719, 727, 761, 829, 911, 997, 1103, 1123, 1151, 1171, 1187, 1231,
1283, 1327, 1481, 1493, 1709, 1723, 1861, 1907, 1997 }
\end{verbatim}
}
\end{tbox}
\end{computation}

\begin{computation}\label{MZNdc,manBDC}
For the sequence of polygons $\Delta_k$ (we use pentagons from \S\ref{adrhdjd} and $k>0$ as an example),
verify that $\cO(C)|_C$ is not torsion using Mazur's theorem.
Find the type of the rational elliptic fibration
with fibers  $C_k$ (when $k$ varies).
\begin{tbox}
{\footnotesize
\begin{verbatim}
> K<t> := FunctionField(Rationals());
  a := -(12*t^2+24*t+11);
  b := 4*(t+1)^2*(3*t+2)*(3*t+4);
  E := EllipticCurve([0,a,0,b,0]);
  P<x,y,z> := Ambient(E);
<  p := E!(Points(Scheme(E,x-2*(t+1)*(3*t+2)*z))[2]);
  KodairaSymbols(E);
[ <I2, 1>, <I2, 1>, <I2, 1>, <I4, 1>, <I1, 2> ]
> val := {}; 
  for n in [1..12] do 
   q := n*p;
   d := Lcm([Denominator(r) : r in Eltseq(q)]);
   M := Matrix([[a*d : a in Eltseq(q)],[0,1,0]]);
   g := Gcd([Numerator(f) : f in Minors(M,2)]);
   val := val join {r[1] : r in Roots(g)};
  end for;
  val;
{ -2, -4/3, -1, -2/3, 0 }
\end{verbatim}
}
\end{tbox}

For the heptagons in  \S\ref{hepta} (and $k>1$) we use the following variation:
\begin{tbox}
{\footnotesize
\begin{verbatim}
> K<k> := FunctionField(Rationals());      
  a :=-k*(2*k+1)/(k+2);
  e := -(4*k+2);     
  x0 := 2*k*(k+1)^2/((k-1)*(k+2));
  y0 := x0*(5*k+3)/(1-k);
  b := -2*x0*(2*k+1)*(k+1)^2/((k+2)*(1-k));
  E := EllipticCurve([e,a,b,0,0]);  
  q := E![x0,y0];
  r := E![0,0];
  s := q+q;
  p := s-r;
  Order(p);    
0   
> KodairaSymbols(E);                
[ <IV, 1>, <I4, 1>, <I4, 1>, <I4, 1>, <IV, 1>, <IV, 1> ]

> val := {};
  for n in [1..12] do
   q := n*p;
   d := Lcm([Denominator(r) : r in Eltseq(q)]);
   M := Matrix([[a*d : a in Eltseq(q)],[0,1,0]]);
   g := Gcd([Numerator(f) : f in Minors(M,2)]);
   val := val join {r[1] : r in Roots(g)};
   end for;
  val;
{ -2, 1 }
\end{verbatim}
}
\end{tbox}
\end{computation}

\begin{computation}\label{goodprimes}
We find the non-polyhedral primes for
the polygon of \S\ref{sgasrhasrh} with a smooth minimal model $Y$. The~function DP3 computes the smooth cubic surface contraction of $Y$ and its hyperplane section, the elliptic curve. 
\begin{tbox}
{\footnotesize
\begin{verbatim}
> pol := Polytope([
  [3,0],[6,1],[8,2],[23,12],[27,15],[30,18],[30,19],
  [29,20],[21,26],[18,28],[16,29],[13,30],[12,30],
  [11,29],[9,25],[7,20],[1,4],[0,1],[0,0]
  ]);
  MX := imatX(pol);
  MY := imatY(pol);
  ind := [i : i in [1..#Vertices(pol)] | MX[i,i] eq MY[i,i]];
  vv,E,pts := DP3(pol,ind);
  roots := FindRoots(pol);
  B := resC(pol,E,vv,pts);
  ImgRoots := [&+[Eltseq(v)[i]*B[i] : i in [1..#B]] : v in roots];
  Cl,g := MapToY(pol);
  C := g(CinS(pol));
  ImgC := &+[Eltseq(C)[i]*B[i] : i in [1..#B]];
  {p : p in PrimesInInterval(2,2000) | p notin BadPrimes(E) and 
  not IsPolyhedralPrime(roots,ImgRoots,C,ImgC,p)};
{ 29, 43, 67, 71, 89, 101, 113, 167, 179, 181, 191, 197, 211, 233, 239, 241, 
263, 269, 313, 337, 349, 359, 379, 383, 409, 449, 461, 491, 557, 587, 617, 701, 
727, 733, 751, 769, 773, 809, 811, 829, 857, 877, 911, 929, 937, 977, 1031, 
1039, 1051, 1087, 1091, 1093, 1097, 1117, 1129, 1153, 1187, 1193, 1223, 1229, 
1231, 1237, 1249, 1259, 1303, 1319, 1321, 1433, 1481, 1489, 1511, 1523, 1553, 
1583, 1607, 1609, 1663, 1669, 1709, 1753, 1873, 1877, 1907, 1949, 1999 }
\end{verbatim}
}
\end{tbox}
\end{computation}

\begin{computation}\label{badprimes}
For a Halphen polygon $\Delta$ such that the corresponding
curve $C$ is smooth in characteristic $0$,
the function Bprimes computes the set of ``bad primes", namely, for any other prime $p$, in characteristic $p$  we have that:
\begin{itemize}
\item The Newton polygon of $C$ is equal to $\Delta$. 
\item The curve $C$ is smooth, with the same (smooth) Weierstrass model.
\item $K_X+C$ admits a uniform (over $p$) Zariski decomposition $N+P$, with 
$$P=0,\quad N=\sum a_i C_i,$$
with the curves $C_i$ irreducible. In particular, $Y$ has du Val singularities.  
\item If $Z\to Y$ is the minimal resolution, the roots in $\bE_8=\Cl_0(Z)$ that lie in $\Ker(\ored)$ stay 
the same as in characteristic $0$. 
\end{itemize}

For the Halphen polygon in Example \ref{ex3} we obtain: 
\begin{tbox}
{\footnotesize
\begin{verbatim}
> pol := Polytope([[0,0],[1,0],[6,1],[8,2],[7,5],[5,8],[1,2]]);
> Bprimes(pol);   
{ 2, 3, 5, 7, 11, 19, 71 } 
 \end{verbatim}
}
\end{tbox}
\end{computation}

\begin{computation}
 \label{sporadic}
We fix the prime $p := 2$ and the integer $e$ such that 
$|eC|$ is a pencil in characteristic $p$.
We then compute the cardinality of reducible fibers of the 
fibration $\pi\colon X\to {\mathbb P}^1$ associated to $|eC|$
(by the proof of Theorem~\ref{asgarh}, it is enough to consider 
fibers over points defined on ${\mathbb F}_p$). 

\begin{tbox}
{\footnotesize
\begin{verbatim}
> p := 2;
> e := 1;
> pol := Polytope([[0,0],[6,2],[9,4],[10,5],[9,6],[3,10],[1,4]]);
> ls := FindCurves(e*pol,e*Width(pol),GF(p));
> pencil := [ls[2]] cat [ls[1]+t*ls[2] : t in GF(p)];
> red := [];
> for g in pencil do
 mu := #[f : f in Factorization(g) | #Monomials(f[1]) gt 1];
 if mu gt 1 then Append(~red,mu); end if;
 end for;
> red;
[ 3, 2 ]
 \end{verbatim}
}
\end{tbox}
\end{computation}


\begin{bibdiv}
\begin{biblist}

\bib{MHV}{article}{
  title={On-shell structures of MHV amplitudes beyond the planar limit},
  author={{Arkani-Hamed}, Nima},
  author={Bourjaily, Jacob},
  author={Cachazo, Freddy},
  author={Postnikov, Alexander},
  author={Trnka, Jaroslav},
  journal={J. of High Energy Physics},
  volume={2015},
  number={6},
  pages={179},
  year={2015},
  publisher={Springer}
}


\bib{ADHL}{book}{
   author={Arzhantsev, Ivan},
   author={Derenthal, Ulrich},
   author={Hausen, J{\"u}rgen},
   author={Laface, Antonio},
   title={Cox rings},
   series={Cambridge Studies in Advanced Mathematics},
   volume={144},
   publisher={Cambridge University Press, Cambridge},
   date={2015},
   pages={viii+530},
}


\bib{Artin}{article}{
    AUTHOR = {Artin, Michael},
     TITLE = {Some numerical criteria for contractability of curves on
              algebraic surfaces},
   JOURNAL = {Amer. J. Math.},
    VOLUME = {84},
      YEAR = {1962},
     PAGES = {485--496},
     }



	
\bib{AS}{article}{
AUTHOR = {Alexeev, Valery},
AUTHOR={Swinarski, David},
     TITLE = {Nef divisors on {$\overline M_{0,n}$} from {GIT}},
 BOOKTITLE = {Geometry and arithmetic},
    SERIES = {EMS Ser. Congr. Rep.},
     PAGES = {1--21},
 PUBLISHER = {Eur. Math. Soc., Z\"{u}rich},
      YEAR = {2012},
 }


\bib{bel}{article}{
 AUTHOR = {Balletti, Gabriele},
 TITLE = {Enumeration of Lattice Polytopes by Their Volume},
 YEAR = {2020},
 JOURNAL = {Discrete Comput. Geom.},
 URL = {https://github.com/gabrieleballetti/small-lattice-polytopes/tree/master/data/2-polytopes}
}


\bib{Bash}{article}{
	year = {1972},
	volume = {27},
	number = {6},
	pages = {25--70},
	author = {Bashmakov, Marc},
	title = {The cohomology of abelian varieties over a number field},
	journal = {Russian Mathematical Surveys}
}


\bib{bauer}{article}{
 AUTHOR = {Bauer, Thomas},
     TITLE = {A simple proof for the existence of {Z}ariski decompositions
              on surfaces},
   JOURNAL = {J. Algebraic Geom.},
    VOLUME = {18},
      YEAR = {2009},
    NUMBER = {4},
     PAGES = {789--793},
}


\bib{mag}{article}{
    AUTHOR = {Bosma, Wieb},
    author={Cannon, John},
    author={Playoust, Catherine},
     TITLE = {The {M}agma algebra system. {I}. {T}he user language},
      NOTE = {Computational algebra and number theory (London, 1993)},
   JOURNAL = {J. Symbolic Comput.},
    VOLUME = {24},
      YEAR = {1997},
    NUMBER = {3-4},
     PAGES = {235--265},
}
	
	
\bib{BDPP}{article}{
AUTHOR = {Boucksom, S\'{e}bastien},
AUTHOR = {Demailly, Jean-Pierre}, 
AUTHOR = {P\u{a}un, Mihai},
AUTHOR = {Peternell, Thomas},
     TITLE = {The pseudo-effective cone of a compact {K}\"{a}hler manifold and
              varieties of negative {K}odaira dimension},
   JOURNAL = {J. Algebraic Geom.},
    VOLUME = {22},
      YEAR = {2013},
    NUMBER = {2},
     PAGES = {201--248},
}

	
\bib{BG}{article}{
AUTHOR = {Belkale, P.},
AUTHOR ={Gibney, A.},
     TITLE = {Basepoint free cycles on {$\overline{{\rm M}}_{0,n}$} from
              {G}romov-{W}itten theory},
   JOURNAL = {Int. Math. Res. Not. IMRN},
      YEAR = {2021},
    NUMBER = {2},
     PAGES = {855--884},
}
	
\bib{BGM}{article}{
AUTHOR = {Belkale, Prakash},
AUTHOR={Gibney, Angela}, 
AUTHOR={Mukhopadhyay,
              Swarnava},
     TITLE = {Vanishing and identities of conformal blocks divisors},
   JOURNAL = {Algebr. Geom.},
    VOLUME = {2},
      YEAR = {2015},
    NUMBER = {1},
     PAGES = {62--90},
}


\bib{BoMu}{article}{
AUTHOR = {Bombieri, E.},
AUTHOR = {Mumford, D.}, 
     TITLE = {Enriques' classification of surfaces in char. $p$. Part III},
        JOURNAL = {Invent. Math.},
    VOLUME = {35},
      YEAR = {1976},
     PAGES = {197--232},
}


\bib{BoMuII}{incollection}{
AUTHOR = {Bombieri, E.},
AUTHOR = {Mumford, D.}, 
     TITLE = {Enriques' classification of surfaces in char. $p$. Part II},
     BOOKTITLE = {Complex analysis and algebraic geometry},
     PAGES = {23--42},
      YEAR = {1977},
}


\bib{Cas}{article}{
    AUTHOR = {Castravet, Ana-Maria},
     TITLE = {The {C}ox ring of {$\overline M_{0,6}$}},
   JOURNAL = {Trans. Amer. Math. Soc.},
    VOLUME = {361},
      YEAR = {2009},
    NUMBER = {7},
     PAGES = {3851--3878},
}


\bib{Ca}{article}{
AUTHOR = {Castravet, Ana-Maria},
     TITLE = {Mori dream spaces and blow-ups},
 BOOKTITLE = {Algebraic geometry: {S}alt {L}ake {C}ity 2015},
    SERIES = {Proc. Sympos. Pure Math.},
    VOLUME = {97},
     PAGES = {143--167},
 PUBLISHER = {Amer. Math. Soc., Providence, RI},
      YEAR = {2018},
}

\bib{CD}{article}{
AUTHOR = {Cantat, Serge},
AUTHOR = {Dolgachev, Igor},
     TITLE = {Rational surfaces with a large group of automorphisms},
   JOURNAL = {J. Amer. Math. Soc.},
    VOLUME = {25},
      YEAR = {2012},
    NUMBER = {3},
     PAGES = {863--905},
}


\bib{CLS}{book}{
   author={Cox, David A.},
   author={Little, John B.},
   author={Schenck, Henry K.},
   title={Toric varieties},
   series={Graduate Studies in Mathematics},
   volume={124},
   publisher={American Mathematical Society, Providence, RI},
   date={2011},
   pages={xxiv+841},
}

\bib{Cox}{article}{   
    AUTHOR = {Cox, David A.},
     TITLE = {The homogeneous coordinate ring of a toric variety},
   JOURNAL = {J. Algebraic Geom.},
    VOLUME = {4},
      YEAR = {1995},
    NUMBER = {1},
     PAGES = {17--50},
}

\bib{CT_Comp}{article}{   
    AUTHOR = {Castravet, Ana-Maria},
    AUTHOR={Tevelev, Jenia},
     TITLE = {Rigid Curves on $\oM_{0,n}$ and Arithmetic Breaks},
   JOURNAL = {Contemporary Math.},
    VOLUME = {564},
      YEAR = {2012},
     PAGES = {19--67},
}

\bib{CT_Crelle}{article}{   
    AUTHOR = {Castravet, Ana-Maria},
    AUTHOR={Tevelev, Jenia},
     TITLE = {Hypertrees, projections, and moduli of stable rational curves},
   JOURNAL = {J. Reine Angew. Math.},
    VOLUME = {675},
      YEAR = {2013},
     PAGES = {121--180},
}



\bib{CT_Duke}{article}{   
  AUTHOR = {Castravet, Ana-Maria},
  AUTHOR={Tevelev, Jenia},
    TITLE = {{$\overline{M}_{0,n}$} is not a {M}ori dream space},
   JOURNAL = {Duke Math. J.},
    VOLUME = {164},
      YEAR = {2015},
    NUMBER = {8},
     PAGES = {1641--1667},
}


\bib{CT_AG}{article}{   
    AUTHOR = {Castravet, Ana-Maria},
    AUTHOR={Tevelev, Jenia},
     TITLE = {Derived category of moduli of pointed curves. {I}},
   JOURNAL = {Algebr. Geom.},
    VOLUME = {7},
      YEAR = {2020},
    NUMBER = {6},
     PAGES = {722--757},
 }


\bib{deb}{book}{
   author={Debarre, Olivier},
   title={Higher-dimensional algebraic geometry},
   series={Universitext},
   publisher={Springer-Verlag, New York},
   date={2001},
   pages={xiv+233},
}


\bib{DGJ}{article}{
    AUTHOR = {Doran, Brent},
    author={Giansiracusa, Noah},
    author={Jensen, David},
     TITLE = {A simplicial approach to effective divisors in
              {$\overline{M}_{0,n}$}},
   JOURNAL = {Int. Math. Res. Not. IMRN},
      YEAR = {2017},
    NUMBER = {2},
     PAGES = {529--565},
}


\bib{Dynkin}{article}{
    author = {Dynkin, E.B.},
    title = {Semisimple subalgebras of semisimple Lie algebras},
    journal = {Trans. Am. Math. Soc.},
    volume = {6},
    pages = {111--244},
    year = {1957}
}


\bib{FL}{article}{
AUTHOR = {Fulger, Mihai},
AUTHOR = {Lehmann, Brian},
     TITLE = {Zariski decompositions of numerical cycle classes},
   JOURNAL = {J. Algebraic Geom.},
    VOLUME = {26},
      YEAR = {2017},
    NUMBER = {1},
     PAGES = {43--106},
}


\bib{F}{article}{
author =  {Fedorchuk, Maksym},
TITLE =  {Symmetric F-conjecture for $g\leq 35$},
 JOURNAL = {arXiv:2007.13457}, 
   year={2020},
}


\bib{FS}{article}{
AUTHOR = {Fedorchuk, Maksym},
AUTHOR = {Smyth, David Ishii},
     TITLE = {Ample divisors on moduli spaces of pointed rational curves},
   JOURNAL = {J. Algebraic Geom.},
    VOLUME = {20},
      YEAR = {2011},
    NUMBER = {4},
     PAGES = {599--629},
}


\bib{FT}{article}{
author = {Fujino, Osamu},
author={ Tanaka, Hiromu},
journal = {Proc. Japan Acad. Ser. A Math. Sci.},
number = {8},
pages = {109--114},
title = {On log surfaces},
volume = {88},
year = {2012}
}


\bib{Fujino}{article}{
author = {Fujino, Osamu},
TITLE = {On minimal model theory for algebraic log surfaces},
   JOURNAL = {arXiv:2004.00246},
   year={2020},
}


\bib{Fulger}{article}{
author = {Fulger, Mihai},
TITLE = {Seshadri constants for curve classes},
   JOURNAL = {arXiv:1707.00734},
   year={2017},
}


\bib{GG}{article}{
  AUTHOR = {Giansiracusa, Noah},
 AUTHOR = {Gibney, Angela},
     TITLE = {The cone of type {$A$}, level 1, conformal blocks divisors},
   JOURNAL = {Adv. Math.},
    VOLUME = {231},
      YEAR = {2012},
    NUMBER = {2},
     PAGES = {798--814},
}


\bib{G}{article}{
AUTHOR = {Gibney, Angela},
     TITLE = {Numerical criteria for divisors on {$\overline M_g$} to be
              ample},
   JOURNAL = {Compos. Math.},
    VOLUME = {145},
      YEAR = {2009},
    NUMBER = {5},
     PAGES = {1227--1248},
}


\bib{GJM}{article}{
  AUTHOR = {Giansiracusa, Noah},
AUTHOR = {Jensen, David},
AUTHOR={Moon, Han-Bom},
     TITLE = {G{IT} compactifications of {$M_{0,n}$} and flips},
   JOURNAL = {Adv. Math.},
    VOLUME = {248},
      YEAR = {2013},
     PAGES = {242--278},
}


\bib{GKK}{article}{
    AUTHOR = {Gonz\'{a}lez, Jos\'{e} Luis},
    AUTHOR = {Karu, Kalle},
     TITLE = {Some non-finitely generated {C}ox rings},
   JOURNAL = {Compos. Math.},
    VOLUME = {152},
      YEAR = {2016},
    NUMBER = {5},
     PAGES = {984--996},
}


\bib{GKM}{article}{
 AUTHOR = {Gibney, Angela},
 AUTHOR={Keel, Sean}, 
 AUTHOR={Morrison, Ian},
     TITLE = {Towards the ample cone of {$\overline M_{g,n}$}},
   JOURNAL = {J. Amer. Math. Soc.},
    VOLUME = {15},
      YEAR = {2002},
    NUMBER = {2},
     PAGES = {273--294},
 }
	
\bib{GM_compositio}{article}{
	Author = {Gupta, Rajiv},
	Author= {Murty, M. Ram},
     TITLE = {Primitive points on elliptic curves},
   JOURNAL = {Compositio Math.},
    VOLUME = {58},
      YEAR = {1986},
    NUMBER = {1},
     PAGES = {13--44},
}	

\bib{GM}{article}{
	Author = {Gupta, Rajiv},
	Author= {Murty, M. Ram},
	Journal = {Inventiones mathematicae},
	Number = {1},
	Pages = {225--235},
	Title = {Cyclicity and generation of points mod p on elliptic curves},
	Volume = {101},
	Year = {1990},
}


\bib{GM_nef}{article}{
AUTHOR = {Gibney, Angela}, 
AUTHOR = {Maclagan, Diane},
TITLE = {Lower and upper bounds for nef cones},
   JOURNAL = {Int. Math. Res. Not. IMRN},
      YEAR = {2012},
    NUMBER = {14},
     PAGES = {3224--3255},
}


\bib{GNW}{article}{
    AUTHOR = {Goto, Shiro},
    author={Nishida, Koji},
    author={Watanabe, Keiichi},
     TITLE = {Non-{C}ohen-{M}acaulay symbolic blow-ups for space monomial
              curves and counterexamples to {C}owsik's question},
   JOURNAL = {Proc. Amer. Math. Soc.},
    VOLUME = {120},
      YEAR = {1994},
    NUMBER = {2},
     PAGES = {383--392},
}


\bib{Grivaux}{article}{
AUTHOR = {Grivaux, Julien},
     TITLE = {Parabolic automorphisms of projective surfaces (after {M}.
              {H}. {G}izatullin)},
   JOURNAL = {Mosc. Math. J.},
    VOLUME = {16},
      YEAR = {2016},
    NUMBER = {2},
     PAGES = {275--298},
}     


\bib{HK}{article}{
   author={Hu, Yi},
   author={Keel, Sean},
   title={Mori dream spaces and GIT},
   note={Dedicated to William Fulton on the occasion of his 60th birthday},
   journal={Michigan Math. J.},
   volume={48},
   date={2000},
   pages={331--348},
}


\bib{hkl}{article}{
   author={Hausen, J\"{u}rgen},
   author={Keicher, Simon},
   author={Laface, Antonio},
   title={Computing Cox rings},
   journal={Math. Comp.},
   volume={85},
   date={2016},
   number={297},
   pages={467--502},
}


\bib{HKL}{article}{
    AUTHOR = {Hausen, J\"{u}rgen},
    AUTHOR = {Keicher, Simon},
    AUTHOR = {Laface, Antonio},
     TITLE = {On blowing up the weighted projective plane},
   JOURNAL = {Math. Z.},
    VOLUME = {290},
      YEAR = {2018},
    NUMBER = {3-4},
     PAGES = {1339--1358},
 }


\bib{HM}{article}{
	Author = {Harris, J.},
	Author= {Mumford, D.},
     TITLE = {On the {K}odaira dimension of the moduli space of curves},
      NOTE = {With an appendix by William Fulton},
   JOURNAL = {Invent. Math.},
     VOLUME = {67},
      YEAR = {1982},
    NUMBER = {1},
     PAGES = {23--88},
}
	

\bib{HasT}{incollection}{
    AUTHOR = {Hassett, Brendan},
    author={Tschinkel, Yuri},
     TITLE = {On the effective cone of the moduli space of pointed rational
              curves},
 BOOKTITLE = {Topology and geometry: commemorating {SISTAG}},
    SERIES = {Contemp. Math.},
    VOLUME = {314},
     PAGES = {83--96},
 PUBLISHER = {Amer. Math. Soc., Providence, RI},
      YEAR = {2002},
}




\bib{HeYang}{article}{
   author={He, Zhuang},
   author={Yang, Lei},
     TITLE = {Birational geometry of blow-ups of projective spaces along
              points and lines},
   JOURNAL = {Int. Math. Res. Not. IMRN},
      YEAR = {2021},
    NUMBER = {9},
     PAGES = {6442--6497},
 }


\bib{KMcK}{article}{
   AUTHOR = {Keel, Sean},
   AUTHOR = {M\textsuperscript{c}Kernan, James},  
   title={Contractible Extremal Rays on $\overline{M}_{0,n}$},
   eprint={arXiv:alg-geom/9607009v1},
   date={1996},
}	


\bib{KM}{book}{
AUTHOR = {Koll\'{a}r, J\'{a}nos},
AUTHOR = {Mori, Shigefumi},
     TITLE = {Birational geometry of algebraic varieties},
    SERIES = {Cambridge Tracts in Mathematics},
    VOLUME = {134},
 PUBLISHER = {Cambridge University Press, Cambridge},
      YEAR = {1998},
     PAGES = {viii+254},
}


\bib{FA}{book}{
     editor = {Koll\'ar, J\'anos},
     title = {Flips and abundance for algebraic threefolds - A summer seminar at the University of Utah (Salt Lake City, 1991)},
     series = {Ast\'erisque},
     number = {211},
     year = {1992},
     url = {http://www.numdam.org/item/AST_1992__211__1_0},
}


\bib{LazI}{book}{
AUTHOR = {Lazarsfeld, Robert},
     TITLE = {Positivity in algebraic geometry. {I}},
    SERIES = {Ergebnisse der Mathematik und ihrer Grenzgebiete. 3. Folge. A
              Series of Modern Surveys in Mathematics},
    VOLUME = {49},
 PUBLISHER = {Springer-Verlag, Berlin},
      YEAR = {2004},
     PAGES = {xviii+385},
 }


\bib{LazII}{book}{
AUTHOR = {Lazarsfeld, Robert},
     TITLE = {Positivity in algebraic geometry. {II}},
    SERIES = {Ergebnisse der Mathematik und ihrer Grenzgebiete. 3. Folge. A
              Series of Modern Surveys in Mathematics},
    VOLUME = {49},
 PUBLISHER = {Springer-Verlag, Berlin},
      YEAR = {2004},
     PAGES = {xviii+385}
}


\bib{LM}{article}{
    AUTHOR = {Losev, A.},
    AUTHOR = {Manin, Y.},
     TITLE = {New moduli spaces of pointed curves and pencils of flat
              connections},
  JOURNAL = {Michigan Math. J.},
    VOLUME = {48},
      YEAR = {2000},
     PAGES = {443--472},   
}       


\bib{lmfdb}{techreport}{
  author       = {LMFDB},
  title        = {The $L$-functions and Modular Forms Database},
  note = {available at \url{http://www.lmfdb.org}},
  year         = {2020},
}


\bib{LT}{article}{
author = {Lang, S.},
author = {Trotter, H.},
journal = {Bull. Amer. Math. Soc.},
number = {2},
pages = {289--292},
title = {Primitive points on elliptic curves},
volume = {83},
year = {1977}
}


\bib{lu}{article}{
   author={Laface, Antonio},
   author={Ugaglia, Luca},
   title={On base loci of higher fundamental forms of toric varieties},
   journal={J. Pure Appl. Algebra},
   volume={224},
   date={2020},
   number={12},
   pages={106447, 18},
} 



\bib{Mazur}{article}{
     author = {Mazur, Barry},
     title = {Modular curves and the Eisenstein ideal},
     journal = {Publ.~Math.~IH\'ES},
     publisher = {Institut des Hautes \'Etudes Scientifiques},
     volume = {47},
     year = {1977},
     pages = {33--186},
}


\bib{Nikulin}{article}{
AUTHOR = {Nikulin, Viacheslav V.},
     TITLE = {A remark on algebraic surfaces with polyhedral {M}ori cone},
   JOURNAL = {Nagoya Math. J.},
    VOLUME = {157},
      YEAR = {2000},
     PAGES = {73--92},
}



\bib{Obinna}{techreport}{
 title={Database of blow-ups of toric surfaces of Picard number $2$},
 author={Obinna, Stephen},
 year={2017},
 series={University of Massachusetts},
 note = { \url{https://people.math.umass.edu/~tevelev/obinna/}},
}



\bib{OS}{article}{
  title={The Mordell-Weil lattice of a rational elliptic surface},
  author={Oguiso, Keiji},
  author={Shioda, Tetsuji},
  journal={Rikkyo Daigaku sugaku zasshi},
  volume={40},
  number={1},
  year={1991},
  pages={83--99}
}


\bib{OP}{article}{
author = {Oda, Tadao},
author={Park, Hye Sook},
journal = {Tohoku Math. J. (2)},
number = {3},
pages = {375--399},
publisher = {Tohoku University, Mathematical Institute},
title = {Linear Gale transforms and Gel'fand--Kapranov--Zelevinskij decompositions},
volume = {43},
year = {1991}
}	


\bib{Opie}{article}{
    AUTHOR = {Opie, Morgan},
     TITLE = {Extremal divisors on moduli spaces of rational curves with marked points},
   JOURNAL = {Michigan Math. J.},
    VOLUME = {65},
      YEAR = {2016},
    NUMBER = {2},
     PAGES = {251--285},
}


\bib{Se}{article}{
	Author = {Serre, Jean-Pierre},
	Journal = {Inventiones mathematicae},
	Number = {4},
	Pages = {259--331},
	Title = {Propri{\'e}t{\'e}s galoisiennes des points d'ordre fini des courbes elliptiques},
	Volume = {15},
	Year = {1971},
}


\bib{SilST}{article}{
author = {Silverman, Joseph H.},
journal = {Journal für die reine und angewandte Mathematik},
pages = {197--211},
title = {Heights and the specialization map for families of abelian varieties},
volume = {342},
year = {1983},
}

\bib{Silv_Advanced}{book}{
AUTHOR = {Silverman, Joseph H.},
     TITLE = {Advanced topics in the arithmetic of elliptic curves},
    SERIES = {Graduate Texts in Mathematics},
    VOLUME = {151},
 PUBLISHER = {Springer-Verlag, New York},
      YEAR = {1994},
     PAGES = {xiv+525},
}


\bib{Silv}{book}{
  AUTHOR = {Silverman, Joseph H.},
     TITLE = {The arithmetic of elliptic curves},
    SERIES = {Graduate Texts in Mathematics},
    VOLUME = {106},
   EDITION = {Second},
 PUBLISHER = {Springer, Dordrecht},
      YEAR = {2009},
     PAGES = {xx+513},
}


\bib{ST}{book}{
  author={Silverman, Joseph H.},
   author={Tate, John T.},
   title={Rational points on elliptic curves},
   series={Undergraduate Texts in Mathematics},
   edition={2},
   publisher={Springer, Cham},
   date={2015},
   pages={xxii+332},
}


\bib{Tanaka}{article}{
AUTHOR = {Tanaka, Hiromu},
     TITLE = {Minimal models and abundance for positive characteristic log
              surfaces},
   JOURNAL = {Nagoya Math. J.},
    VOLUME = {216},
      YEAR = {2014},
     PAGES = {1--70},
}

\bib{Scattering}{article}{
    AUTHOR = {Tevelev, Jenia},
  title={Scattering amplitudes of stable curves},
   date={2020},
    JOURNAL = {arXiv:2007.03831},
    EPRINT =   {https://arxiv.org/pdf/2007.03831.pdf}
}


\bib{Chebotarev}{article}{
author={Tschebotareff, N.},
year={1926},
title={Die Bestimmung der Dichtigkeit einer Menge von Primzahlen, welche zu 
einergegebenen Substitutionsklasse geh\"oren},
journal={Math.Ann.},
volume={95},
pages={191--228}
}


\bib{Ver}{article}{
    AUTHOR = {Vermeire, Peter},
     TITLE = {A counterexample to {F}ulton's conjecture on {$\overline
              M_{0,n}$}},
   JOURNAL = {J. Algebra},
    VOLUME = {248},
      YEAR = {2002},
    NUMBER = {2},
     PAGES = {780--784},
}


\bib{W}{article}{
  title={Kummer theory of abelian varieties and reductions of Mordell-Weil groups},
  author={Weston, Tom},
  journal={Acta Arithmetica},
  volume={110},
  pages={77--88},
  year={2003}
}

		
	
\end{biblist}
\end{bibdiv}
\end{document}